\documentclass[12pt]{amsart}

\usepackage{lineno,hyperref}
\usepackage{amsmath,amsfonts,amsthm,mathrsfs,mathtools}
\usepackage{amssymb,latexsym}
\usepackage{cases}
\usepackage[numbers]{natbib}
\usepackage{booktabs}
\usepackage{bm} 
\usepackage{cleveref}
\usepackage{paralist}
\usepackage{multirow}
\usepackage{mfirstuc}
\usepackage{amscd} 

\renewcommand\appendix{\setcounter{secnumdepth}{-2}}

\usepackage{url,cancel}
\usepackage{float}

\usepackage{caption,enumitem}
\newtheorem{thm}{Theorem}[section]
\newtheorem*{Thm}{Theorem}

\newtheorem{lem}[thm]{Lemma}
\newtheorem{prop}[thm]{Proposition}
\newtheorem{cor}[thm]{Corollary}
\theoremstyle{remark}
\newtheorem{re}[thm]{Remark}

\theoremstyle{definition}
\newtheorem{ex}[thm]{Example}
\newtheorem{defn}[thm]{Definition}
\newcommand{\pmid}{{\mathfrak{P}\mid \mathfrak{p}}}
\newcommand{\Ker}{\mbox{Ker}}
\newcommand{\gen}{\mbox{gen}}
\newcommand{\spn}{\mbox{spn}}
\newcommand{\cls}{\mbox{cls}}
\newcommand{\sgp}{\mbox{Sgp}}
\newcommand{\ord}{\mathrm{ord}}

\newcommand{\pr}{\mathrm{pr}}
\newcommand{\rank}{\mbox{rank}\,}
\newcommand{\chartic}{\mbox{char} }
 \newcommand{\cl}{\mbox{Cl} }

\def\rep{{\to\!\!\! -}}

\makeatletter

\@addtoreset{equation}{section}
\makeatother
\numberwithin{equation}{section}

\begin{document}

	\author{Zilong He}
	\address{Department of Mathematics,	Dongguan University of Technology, Dongguan 523808, China}
	\email{zilonghe@connect.hku.hk}

	\title[ ]{Arithmetic Springer theorem and n-universality under field extensions}
	\thanks{ }
	\date{\today}
	\begin{abstract}
		Based on BONGs theory, we prove the norm principle for integral  and relative integral spinor norms of quadratic forms over general dyadic local fields, respectively. By virtue of these results, we further establish the arithmetic version of Springer's theorem for indefinite quadratic forms. Moreover, we solve the lifting problems on $n$-universality over arbitrary local fields.
	\end{abstract}
	\maketitle
	
	\section{Introduction}
	  A classical theorem by Springer \cite{springer-surlesformes-1952} states that an anisotropic quadratic space over an algebraic number field remains anisotropic over any field extension of odd degree, which was recently extended to the semilocal ring case by Gille and Neher \cite{gille-neher-springerthm-2021}. Let $F$ be an algebraic number field and $E$ a finite extension of $F$. Then Springer's result can be formally rephrased with respect to representations of quadratic spaces.
	\begin{Thm}[Springer Theorem]
		 Let $V$ and $U$ be quadratic spaces over $F$. Suppose that $[E:F]$ is odd. If $V\otimes_{F} E$ represents  $U\otimes_{F} E$, then $V$ represents $U$. 
	\end{Thm}

The question on the behaviour of the genus of a positive definite integral quadratic form when lifted to a totally real number field, was proposed by Ankeny \cite{earnest-hisa-springer-1975}; similar problems on spinor genus was studied by Earnest and Hsia. In a series of papers \cite{earnest-hisa_spinorgeneraI-1977,earnest-hisa_spinorgeneraII-1978}, they proved Springer-type theorems for spinor equivalence under certain conditions (also see \cite{earnest_thesis_1975}). In \cite{xu_springer_1999}, Xu further conjectured that the arithmetic version of Springer Theorem holds for indefinite quadratic forms over the ring of integers $\mathcal{O}_{F}$ (cf. Theorem \ref{thm:AST}) and proved the case when $2$ is unramified over $F$. However, such conjecture does not hold in general for positive definite forms when lifted to a number field $E$ that is not totally real (cf. \cite[Chap 7]{kitaoka-arithmetic} and \cite[\S 7]{xu_springer_1999}). When $E$ is totally real, Daans, Kala, Kr\'{a}sensk\'{y} and Yatsyna \cite{dkky_Springer-2024} also found that it may fail infinitely often.
 
 The proof of Arithmetic Springer Theorem for isometries (resp. representations), i.e., Theorem \ref{thm:AST}(i) (resp. (ii)) can be reduced to two key parts (I1) and (I2) (resp. (R1) and (R2)) in non-archimedean local fields (also see \cite[Introduction]{beli_representations_2019}):

      \begin{itemize}
      	    
      	\item[(I1)] Norm principle for integral spinor norms.
      	
      	\item[(I2)] Equivalent conditions for isometries of quadratic forms.

      	\item[(R1)] Norm principle for relative integral spinor norms.
      	
      	\item[(R2)]  Equivalent conditions for representations of quadratic forms.

         \end{itemize}
   In non-dyadic local fields, O'Meara \cite{omeara_integral_1958} gave (I2) and (R2). Earnest and Hsia \cite{earnest-hisa_spinorgeneraI-1977} proved (I1) based on Kneser's work in \cite{kneser_indefiniter_1956}. To compute relative integral spinor norms, Hsia, Shao and Xu \cite{hsia_indefinite_1998} developed various reduction formulas, which Xu \cite{xu_springer_1999} used to show (R1). Therefore, it remains to address the dyadic cases. Many results described in terms of Jordan splittings have contributed to the calculation of integral and relative integral spinor norms, such as \cite{xu-chan_spinorgeneraI_2004,earnest-hisa_spinorormI-1975,hisa_spinorormI-1975,lv-xu-spinor-norma-2010,schulze-xu_representation_2004,xu-remark-1989,xu_integral-spinor-norm_1993,xu_integral-spinor-norm_1993-2,xu_integral-spinor-norm_1995,xu_minimal_2003,xu_spinorgeneraII_2005} and so on. On the other hand, O'Meara \cite{omeara_quadratic_1963} solved (I2)  for the general dyadic case. For (R2), O'Meara \cite{omeara_integral_1958} and Riehm  \cite{riehm_integral_1964} solved the 2-adic and modular cases, respectively. In general dyadic local fields, Koerner \cite{koerner_integral_1973} provided (I2) and (R2) for the binary case, but for higher dimensional cases, only some necessary conditions for (R2) were given by Xu \cite{xu_springer_1999}. Therefore, only partial results could be obtained in \cite{earnest-hisa_spinorgeneraI-1977,earnest-hisa_spinorgeneraII-1978,xu_springer_1999} for Springer-type theorems.
    
    The breakthrough in this problem involved describing quadratic forms through bases of norm generators (abbr. BONGs). By developing BONGs theory, Beli completed the calculation of integral spinor norms and relative spinor norms \cite{beli_thesis_2001,beli_integral_2003}, and proved the isometry and representation theorem \cite{beli_representations_2006, beli_Anew_2010, beli_representations_2019} on quadratic forms over arbitrary dyadic local fields. As a result, all the necessary ingredients for Arithmetic Springer Theorem have been gathered. 
     In this paper, we introduce a new way to tackle the difficulty in applying Beli's formulas so that we may treat different problems involving spinor norm groups in general dyadic fields. To be precise, we show the norm principle for different ``pieces" (Lemma \ref{lem:norm-principle-Ga} for $G(a)$), and then prove the norm principle for integral spinor norms (Theorem \ref{thm:norm-principle-M}) and relative integral spinor norms  (Theorem \ref{thm:norm-principle-M-N}) piece by piece, respectively. Finally, following the work of Beli \cite[Theorems 3.2 and 4.5]{beli_representations_2006}, Earnest and Hsia \cite[\S 1]{earnest-hisa_spinorgeneraI-1977} and Xu \cite[\S 1]{xu_springer_1999}, we prove Springer-type theorems for genera and proper spinor genera (Theorems \ref{thm:AST-genus} and \ref{thm:AST-spn}) in general algebraic number fields, thereby confirming Xu's conjecture, i.e.,
      \begin{thm}[Arithmetic Springer Theorem]\label{thm:AST}
     	Let $L$ and $N$ be indefinite $\mathcal{O}_{F}$-lattices and $ \ell=\rank L\ge \rank N$. Suppose that $[E:F]$ is odd and $\ell\ge 3$.
      \begin{enumerate}[itemindent=-0.5em,label=\rm (\roman*)]
     		 	\item If $L\otimes_{\mathcal{O}_{F}} \mathcal{O}_{E}$ is isometric to $ N\otimes_{\mathcal{O}_{F}} \mathcal{O}_{E}$, then $L$ is isometric to $N$.
     		
     		\item If $L\otimes_{\mathcal{O}_{F}} \mathcal{O}_{E}$ represents $N\otimes_{\mathcal{O}_{F}} \mathcal{O}_{E}$, then $L$ represents $N$.
     	\end{enumerate} 
     \end{thm}
     \begin{re}
     	  (i) Theorem \ref{thm:AST} for quadratic forms over general Dedekind domains was recently proved by Hu, Liu and Xu \cite{hlx_springer_2024} under some mild assumptions. 
     	  
     	  (ii) Theorem \ref{thm:AST} is not true in general when $\ell=2$, as shown in Example \ref{ex:binary}. Also, see \cite[Example 5.10]{hlx_springer_2024} for counterexamples in function field cases.
 \end{re}
Recently, there has been a lot of concern regarding the lifting problem for universal positive definite quadratic forms, as discussed in \cite{ kala-number-field-2021, kala-universal-2023,kala-yatsyna-lifting-2021, kala-yatsyna-Kitaoka-lifting-2023, xu_indefinite_2020}, and so on. It is natural to study the behaviour of $n$-universal indefinite quadratic forms or lattices over $F$ upon inflation to the extension field $E$. An $\mathcal{O}_{F}$-lattice $L$ is called \textit{integral} if its norm is contained in $\mathcal{O}_{F}$.

\begin{defn}\label{defn:n-uni}
	  Let $n$ be a positive integer and $L$ be an integral $\mathcal{O}_{F}$-lattice. 
	  
	    \begin{enumerate}[itemindent=-0.5em,label=\rm (\roman*)] 
	  	\item For non-archimedean primes $\pmid $, we say that $L_{\mathfrak{p}}$ is $n$-universal over $E_{\mathfrak{P}}$ if $L_{\mathfrak{p}}\otimes_{\mathcal{O}_{F_{\mathfrak{p}}}} \mathcal{O}_{E_{\mathfrak{P}}}$ represents every integral $\mathcal{O}_{E_{\mathfrak{P}}}$-lattice $N$ of rank $n$.
	  	
	  	\item We say that $L$ is $n$-universal over $E$ if $L$ is indefinite and $L\otimes_{\mathcal{O}_{F}} \mathcal{O}_{E}$ represents every integral $\mathcal{O}_{E}$-lattice $N$ of rank $n$ for which $L_{\mathfrak{P}}$ represents $N_{\mathfrak{P}} $ at all real primes $\mathfrak{P}$ of $E$.
	  \end{enumerate}

\end{defn}

As an immediate application of Theorem \ref{thm:AST}, we have the following corollary. 
  \begin{cor}\label{cor:globally-n-universal-extension-E-odd}
   	Let $L$ be an integral indefinite $\mathcal{O}_{F}$-lattice of $\rank \ge 3$. Suppose that $[E:F]$ is odd and $n\ge 1$. If $L$ is $n$-universal over $E$, then it is $n$-universal over $F$.
   \end{cor}
   \begin{proof}
   	Suppose that $L$ is $n$-universal over $E$. Then $L\otimes_{\mathcal{O}_{F}} \mathcal{O}_{E}$ represents the $\mathcal{O}_{E}$-lattice $N\otimes_{\mathcal{O}_{F}} \mathcal{O}_{E}$ for any $\mathcal{O}_{F}$-lattice $N$. Since $[E:F]$ is odd, by Theorem \ref{thm:AST}(ii), $L$ represents $N$.
   \end{proof}

 We also consider lifting problems for $n$-universal lattices over non-archimedean local fields. Based on characterization of $n$-universal $\mathcal{O}_{F_{\mathfrak{p}}}$-lattices in \cite{beli_universal_2020,HeHu2,hhx_indefinite_2021,xu_indefinite_2020}, we provide the necessary and sufficient conditions for an $\mathcal{O}_{F_{\mathfrak{p}}}$-lattice to be $n$-universal over a finite extension $E_{\mathfrak{P}}$ of $F_{\mathfrak{p}}$ (Theorems \ref{thm:n-universal-nondyadic-overE}, \ref{thm:even-n-universaldyadic-overE}, \ref{thm:odd-n-universaldyadic-overE} and \ref{thm:1-universaldyadic-overE}). Interestingly, the oddness of $[E:F]$ is necessary for Arithmetic Springer Theorem (cf. \cite[Example]{hisa_spinorormI-1975} or \cite[Appendix A]{earnest-hisa_spinorgeneraI-1977}), but not for $n$-universality to be stable under field extensions, as seen from these conditions. Using those equivalent conditions, we weaken the oddness assumption of Corollary \ref{cor:globally-n-universal-extension-E-odd}:
 

	\begin{thm}\label{thm:globally-n-universal-extension-E}
			Let $L$ be an integral indefinite $\mathcal{O}_{F}$-lattice of rank $\ell$. Suppose that either $n\ge 3$, or $\ell\ge n+3=5$ and the class number of $F$ is odd. If $L$ is $n$-universal over $E$, then it is $n$-universal over $F$.
	\end{thm}

%
%
%
%

		%
		%
	
		%

 
   The remaining sections are organized as follows. In Section \ref{sec:BONGs}, we review Beli's BONGs theory and study the spinor norms and invariants of quadratic lattices in dyadic local fields. In Section \ref{sec:lifting}, we analyze the invariants and the equivalent conditions on the representation of quadratic lattices between ground fields and extension fields. Then we prove the norm principles for spinor norms in Section \ref{sec:norm-principle} and Theorem \ref{thm:AST} in Section \ref{sec:AST}. At last, we study $n$-universal lattices over local fields under field extensions and show Theorem \ref{thm:globally-n-universal-extension-E}.
   
    \medskip
    
    Before proceeding further, we introduce some settings for discussion in various sections. Any unexplained notation or definition can be found in \cite{omeara_quadratic_1963}.
    
    \noindent \textbf{General settings} 
  
    Let $ F $ be an algebraic number field or a non-archimedean local field with $  \chartic\, F\not=2 $, $ \mathcal{O}_{F} $ the ring of integers of $ F $ and $ \mathcal{O}_{F}^{\times}$ the group of units. For a non-degenerate quadratic space $V$ over $F$, let $L$ be an $\mathcal{O}_{F}$-lattice in $V$ and write $FL$ for the subspace of $V$ spanned by $L$ over $F$. For an $ \mathcal{O}_{F} $-lattice $ L $, we denote the scale, norm and volume of $L$ by $ \mathfrak{s}(L) $, $ \mathfrak{n}(L) $ and $\mathfrak{v}(L)$, respectively.
    We also denote by $O^{+}(V)$  the special orthogonal group of $V$. If $V=FL$, write
    \begin{align*}
    	 O^{+}(L)\coloneqq&\{\sigma\in O^{+}(V)\mid \sigma(L)=L\},\\
    	 O^{\prime}(V)\coloneqq&\{\sigma\in O^{+}(V)\mid \theta (\sigma) =1\},
    \end{align*}
     where $\theta$ denotes the spinor norm (cf. \cite[\S 55]{omeara_quadratic_1963}). Furthermore, the notations $\lfloor \cdot\rfloor$ and $\lceil \cdot \rceil$ denote the usual floor and ceiling function, respectively.
     
     Let $E$ be a finite extension of $F$. We write $N_{E/F}$ for the norm from $E$ to $F$. We usually add the tilde symbol ``$\;\widetilde{}\;$" for the object (e.g. a map, a function, a group or a lattice) of $F$ considered in $E$. For $c\in F^{\times}$, we also use $\widetilde{c}$ to emphasize that it is an element of $E^{\times}$, but omit the tilde symbol, when our context is clear. For instance, a map or a function $\widetilde{\mathcal{F}}$ defined on a subset of $E^{\times}$ acts on $\widetilde{c}$, we use $\widetilde{\mathcal{F}}(c)$ instead of $\widetilde{\mathcal{F}}(\tilde{c})$ for convenience.
    \medskip
    
   \noindent	\textbf{Local settings}
    
    When $ F $ is a non-archimedean local field, write $ \mathfrak{p} $ for the maximal ideal of $\mathcal{O}_{F}$ and $ \pi \in\mathfrak{p}$ for a uniformizer. For $c\in F^{\times}\coloneqq F\backslash\{0\}$, we denote by $\ord(c)=\ord_{\mathfrak{p}}(c)$ the \textit{order} of $c$ and put $\ord(0)=\infty$. Set $ e=e_{F}\coloneqq\ord(2)$. For a fractional or zero ideal $\mathfrak{c}$ of $F$, we put $\ord(\mathfrak{c})=\min\{\ord(c)\,|\,c\in \mathfrak{c}\}$. As usual, we denote by $\Delta=\Delta_{F}$ a fixed unit such that $F(\sqrt{\Delta_{F}})$ is an unramified quadratic extension of $F$. If $F$ is non-dyadic, then $\Delta_{F}$ is an arbitrary non-square unit. If $F$ is dyadic, then $\Delta_{F}$ is a non-square unit of the form $\Delta_{F}=1-4\rho_{F}$.
    
    When $F$ is dyadic, we define the \textit{quadratic defect} of $ c $ by $ \mathfrak{d}(c)=\mathfrak{d}_{F}(c)\coloneqq\bigcap_{x\in F}(c-x^{2})\mathcal{O}_{F}$ and the \textit{order of relative quadratic defect} by the map 
    \begin{align*}
    	d=d_{F}\colon F^{\times}/F^{\times 2} \to \mathbb{N}\cup \{\infty\},\quad c\mapsto\ord (c^{-1}\mathfrak{d}(c)).
    \end{align*}

    For $c\in F^{\times}$ and $a,b\in F$, we write $c A(a,b)$ for the binary lattice with the Gram matrix $c\begin{pmatrix}
    	a  &  1\\
    	1  &  b
    \end{pmatrix}$, and then write $\mathbf{H}_{F}$ and $\mathbf{A}_{F}$ for the binary $\mathcal{O}_{F}$-lattices
   \begin{align*}
   	\mathbf{H}_{F}= 
   	\begin{cases}
   		\langle 1,-1\rangle &\text{if $F$ is non-dyadic}, \\
   		2^{-1}A(0,0)         &\text{if $F$ is  dyadic},
   	\end{cases}\quad\text{and}\quad
   	\mathbf{A}_{F}=
   	\begin{cases}
   		\langle 1,-\Delta_{F}\rangle  &\text{if $F$ is non-dyadic},\\
   		2^{-1}A(2,2\rho_{F})   &\text{if $F$ is dyadic}.
   	\end{cases}
   \end{align*}  
   And denote by $\mathbb{H}_{F}$ (resp. $\mathbb{A}_{F}$) the quadratic space spanned by $\mathbf{H}_{F}$ (resp. $\mathbf{A}_{F}$) over $F$. 
	
	\medskip
 \noindent	\textbf{Global settings}
	
	When $F$ is an algebraic number field, we denote by $\Omega_{F}$ the set of primes (or places) of $F$ and by $\infty_{F}$ the set of infinite primes. Let $L$ be an $\mathcal{O}_{F}$-lattice on a quadratic space over $F$. For $\mathfrak{p}\in \Omega_{F}$, let $F_{\mathfrak{p}}$ be the completion of $F$ at $\mathfrak{p}$ and $V_{\mathfrak{p}}=V\otimes F_{\mathfrak{p}}$. Then put $L_{\mathfrak{p}}= L\otimes\mathcal{O}_{F_{\mathfrak{p}}}$ for $\mathfrak{p}\in \Omega_{F}\backslash \infty_{F}$ and $L_{\mathfrak{p}}=V_{\mathfrak{p}}$, otherwise.
	
	Let $V=FL$. We denote by $O_{\mathbb{A}} (V)$ the adelic group of  $O^{+}(V)$ and write
	\begin{align*}
				O_{\mathbb{A}}^{+}(L)&\coloneqq\{\sigma_{\mathbb{A}}\in O_{\mathbb{A}} (V)\mid \sigma(L)=L\}, \\
		O_{\mathbb{A}}^{\prime}(V)&\coloneqq\{\sigma_{\mathbb{A}}\in O_{\mathbb{A}}(V)\mid \sigma_{\mathfrak{p}}\in O^{\prime}(V_{\mathfrak{p}})\;\text{for each $\mathfrak{p}\in \Omega_{F}$}\}.
	\end{align*}
	For an $\mathcal{O}_{F}$-lattice  $N\subseteq L$, we put $
		X_{\mathbb{A}}(L/N)\coloneqq\{\sigma_{\mathbb{A}}\in O_{\mathbb{A}} (V)\mid N\subseteq \sigma(L)\}$.
	
	 As \cite[\S 102]{omeara_quadratic_1963}, we write $\gen(L)$, $\spn^{+}(L)$ and $\cls^{+}(L)$ for the genus, proper genus and proper class of $L$, respectively. And we say that $\gen(L)$ \textit{represents} an $\mathcal{O}_{F}$-lattice $N$ if there is a lattice $M\in \gen(L)$ such that $N\subseteq M$. Similarly for the definition of the representation of an $\mathcal{O}_{F}$-lattice by $\spn^{+}(L)$ or $\cls^{+}(L)$.
	
\section{Spinor norms and invariants in terms of BONGs}\label{sec:BONGs}
In this section, we summarize and analyze the key ingredients of the BONGs theory of quadratic forms in dyadic local fields, established by Beli in a series of papers \cite{beli_thesis_2001,beli_integral_2003,beli_representations_2006,beli_Anew_2010,beli_representations_2019,beli_universal_2020}, where any unexplained notation and definition can be found.

Throughout the section, we assume $F$ to be dyadic, i.e., $e\ge 1$. Let $f=[\mathcal{O}_{F}/\mathfrak{p}: \mathbb{Z}_{2}/(p)]$ and $N\mathfrak{p}=|\mathcal{O}_{F}/\mathfrak{p}|$. Let $L$ be a binary $\mathcal{O}_{F}$-lattice, with $\mathfrak{n}(L)=a\mathcal{O}_{F}$ and $a\in F^{\times}$. Then we denote $a(L)\in F^{\times}/\mathcal{O}_{F}^{\times 2}$ by $a(L)\coloneqq a^{-2}\det L$ and denote  $\mathscr{A}=\mathscr{A}_{F}\subseteq F^{\times}/\mathcal{O}_{F}^{\times 2}$ by the set of all possible values of $a(L)$. From \cite[Lemma 3.5]{beli_integral_2003}, we have
\begin{align*}
	\mathscr{A}=\{a\in F^{\times}\mid 4a\in  \mathcal{O}_{F}\;\mbox{and}\;\mathfrak{d}(-a)\subseteq \mathcal{O}_{F}\},
\end{align*}
which implies that
\begin{align}\label{a-in-mathscrA}
	a\in \mathscr{A} \quad\Longleftrightarrow\quad 
	\begin{cases}
	    \ord(a)+2e\ge 0,  \\
	    \ord(a)+d(-a)\ge 0.
	\end{cases}
\end{align}
If $a\in \mathscr{A}$, we further define $\alpha(a)\coloneqq\min\{\ord(a)/2+e,\ord(a)+d(-a)\}$. Thus, from \eqref{a-in-mathscrA}, 
\begin{align}\label{alpha-A-equiv}
	a\in \mathscr{A} \quad\text{if and only if}\quad\alpha(a)\ge 0.
\end{align}
 We denote by $\mathscr{H}=\mathscr{H}_{F}\coloneqq(-1/4)\mathcal{O}_{F}^{\times 2}$, which is an element of $\mathscr{A}$. Clearly, $a\in \mathscr{H}$ if and only if  $\ord(a)=-2e$ and $d(-a)=\infty$.

For abelian groups $H_{1}$ and $H_{2}$ with $H_{1}\subseteq H_{2}$, we denote by $\sgp(H_{2}/H_{1})$ the set of all subgroups $H$ of $H_{2}/H_{1}$. Note that there is one-to-one correspondence of subgroups between the sets $\{H\mid H\subseteq H_{2}/H_{1}\}$ and $\{H\mid H_{1}\subseteq H\subseteq H_{2}\}$. Thus, in the sequel, we also identify $\sgp(H_{2}/H_{1})$ as the set of all subgroups $H$ such that $H_{1}\subseteq H\subseteq H_{2}$. Integral spinor norm groups of binary lattices have been determined in \cite{earnest-hisa_spinorgeneraII-1978,hisa_spinorormI-1975,xu-remark-1989,xu_integral-spinor-norm_1993}. To unify various cases, Beli introduced the functions $G:F^{\times}/\mathcal{O}_{F}^{\times 2}\to \sgp(F^{\times}/F^{\times 2})$ and $g:\mathscr{A}_{F}\to \sgp(\mathcal{O}_{F}^{\times}/\mathcal{O}_{F}^{\times 2})$ (cf. \cite[Definitions 4 and 6]{beli_integral_2003}). In fact, he showed that  $\theta(L)=G(a(L))$ for any binary lattice $L$, where $\theta(L)$ denotes the integral spinor norm group of $L$ (cf. \cite[Lemma 3.7]{beli_integral_2003}), and provided a concise formula for $g$ (cf. \cite[Lemma 5.1]{beli_Anew_2010}).
\begin{lem}\label{lem:ga-formula}
	If $a\in \mathscr{A}$, then $g(a)=(1+\mathfrak{p}^{\alpha(a)})\mathcal{O}_{F}^{\times 2}\cap N(-a)$.
\end{lem}
To clarify the formula in Lemma \ref{lem:ga-formula}, we follow the setting in \cite[\S 1]{beli_integral_2003}. Recall that the Hilbert symbol $(\;,\;)_{\mathfrak{p}}:F^{\times}/F^{\times 2}\times F^{\times}/F^{\times 2}$ is a non-degenerate symmetric bilinear form. Then, for any $H\in \sgp(F^{\times}/F^{\times 2})$, we denote by $H^{\perp}$ its orthogonal complement of $H$ with respect to $(\;,\;)_{\mathfrak{p}}$. Thus, $H^{\perp}=\{c\in F^{\times}\mid (c,a)_{\mathfrak{p}}=1\;\text{for any $a\in H$}\}$. For $a\in F^{\times}$, put 
\begin{align}\label{Na}
	N(a)\coloneqq N(F(\sqrt{a})/F)=\{c\in F^{\times}\mid (c,a)_{\mathfrak{p}}=1\}=(\langle a\rangle F^{\times 2})^{\perp},
\end{align}
where $\langle a\rangle F^{\times 2}\coloneqq\cup_{k\in \mathbb{Z}} a^{k}F^{\times 2}=aF^{\times 2}\cup F^{\times 2}$.
For $h\in \mathbb{R}\cup \{\infty\}$, also put
\begin{align}\label{eq:1+pk}
	(1+\mathfrak{p}^{h})F^{\times 2}\coloneqq\{a\in F^{\times}\mid d(a)\ge h\} 
\end{align}
and $	(1+\mathfrak{p}^{h})\mathcal{O}_{F}^{\times 2}\coloneqq(1+\mathfrak{p}^{h})F^{\times 2}\cap \mathcal{O}_{F}^{\times}$. Furthermore, formally put $(1+\mathfrak{p}^{0})F^{\times 2}=F^{\times}$, $(1+\mathfrak{p}^{\infty})F^{\times 2}=F^{\times 2}$, $(1+\mathfrak{p}^{0})\mathcal{O}_{F}^{\times 2}=\mathcal{O}_{F}^{\times}$ and $(1+\mathfrak{p}^{\infty})\mathcal{O}_{F}^{\times 2}=\mathcal{O}_{F}^{\times 2}$.

\begin{prop}\label{prop:1+p}
	Let $h\in \mathbb{R}$.
	\begin{enumerate}[itemindent=-0.5em,label=\rm (\roman*)]
		\item  If $h\in \mathbb{Z}$ and $0<h<2e$ is even, then $(1+\mathfrak{p}^{h})F^{\times 2}=(1+\mathfrak{p}^{h+1})F^{\times 2}$.
		
		\item If  $h\in \mathbb{Z}$ and $1<h<2e+1$ is odd, then $((1+\mathfrak{p}^{h})F^{\times 2})^{\perp}=(1+\mathfrak{p}^{2e+2-h})F^{\times 2}$.
		
		\item If  $h<0  $, then $(1+\mathfrak{p}^{h})F^{\times 2}=(1+\mathfrak{p}^{0})F^{\times 2}=F^{\times}$.

		\item Let $s,t$ be two consecutive elements in $d(F^{\times})=\{0,1,\ldots,2e-1,2e,\infty\}$. If $s<h\le t$, then $(1+\mathfrak{p}^{h})F^{\times 2}=(1+\mathfrak{p}^{t})F^{\times 2}$.
\end{enumerate}
\end{prop}
\begin{proof}
	Assertions (i), (iii) and (iv)  are straightforward from \eqref{eq:1+pk}. Assertion (ii) is the third assertion in \cite[Lemma 1.2(i)]{beli_integral_2003}.
\end{proof}
\begin{prop}\label{prop:hsharp}
	Let $h\in \mathbb{Z}$ and $0\le h\le 2e$, or $h=\infty$. Set
	\begin{align*}
		h^{\#}\coloneqq\begin{cases}
			  \infty    &\text{if $h=0$}, \\
			 2e    &\text{if $h=1$}, \\
			 2e+2-h  &\text{if $1<h<2e$ is odd}, \\
			 2e+1-h    &\text{if $1<h<2e$ is even}, \\
			 1   &\text{if $h=2e$},  \\
			 0   &\text{if $h=\infty$}.
		\end{cases}
	\end{align*}
	Then $h^{\#}\in d(F^{\times})$, $h+h^{\#}\ge 2e+1$ and $((1+\mathfrak{p}^{h^{\#}})F^{\times 2})^{\perp}=(1+\mathfrak{p}^{h})F^{\times 2}$.
	 \end{prop}
	\begin{proof}
   	It is clear from definition that $h^{\#}\in d(F^{\times})$ and $h+h^{\#}\ge 2e+1$. 
	
	For $h \in \{0,1,2e,\infty\}$, the equality holds by the first two assertions in \cite[Lemma 1.2(i)]{beli_integral_2003}. 
	
	Suppose that $1<h<2e$. Then $1<h^{\#}<2e$ is odd. If $ h $ is odd, then the equality holds by Proposition \ref{prop:1+p}(ii).  If $ h  $ is even, then the equality holds by Proposition \ref{prop:1+p}(ii) and (i).
\end{proof}
\begin{prop}\label{prop:ga-formula-F2}
	Let $h\in \mathbb{R}$ and $c\in F^{\times}$. Then
	\begin{enumerate}[itemindent=-0.5em,label=\rm (\roman*)]
		\item 
		\begin{align*}
		(1+\mathfrak{p}^{h})\mathcal{O}_{F}^{\times 2}F^{\times 2}=\mathcal{O}_{F}^{\times}F^{\times 2}\cap (1+\mathfrak{p}^{h})F^{\times 2}=\begin{cases}
			(1+\mathfrak{p}^{h})F^{\times 2}  &\text{$h>0$},\\
			\mathcal{O}_{F}^{\times}F^{\times 2} &\text{$h\le 0$}.
		\end{cases}
		\end{align*}
		
		 \item $((1+\mathfrak{p}^{h})\mathcal{O}_{F}^{\times 2}\cap N(c) )F^{\times 2}=\mathcal{O}_{F}^{\times}F^{\times 2}\cap (1+\mathfrak{p}^{h})F^{\times 2}\cap N(c)$.
	\end{enumerate}
\end{prop}
\begin{proof}
	 (i) Let $h>0$. First, if $h\in \mathbb{Z}$, then, clearly, $((1+\mathfrak{p}^{h})\mathcal{O}_{F}^{\times 2})F^{\times 2}=(1+\mathfrak{p}^{h})F^{\times 2}$. Then for $h\in \mathbb{R}$, since $ (1+\mathfrak{p}^{h})\mathcal{O}_{F}^{\times 2} =(1+\mathfrak{p}^{\lceil h\rceil})\mathcal{O}_{F}^{\times 2}$ and $ (1+\mathfrak{p}^{h})F^{\times 2} =(1+\mathfrak{p}^{\lceil h\rceil})F^{\times 2}$, the equality follows by the case $h\in \mathbb{Z}$. We also have $(1+\mathfrak{p}^{h})F^{\times 2}\subseteq \mathcal{O}_{F}^{\times}F^{\times 2}$, so $(1+\mathfrak{p}^{h})F^{\times 2}=\mathcal{O}_{F}^{\times}F^{\times 2}\cap (1+\mathfrak{p}^{h})F^{\times 2}$.
	 
	 If $h\le 0$, then $(1+\mathfrak{p}^{h})\mathcal{O}_{F}^{\times 2}=\mathcal{O}_{F}^{\times}$ and $(1+\mathfrak{p}^{h})F^{\times 2}=F^{\times}$, so $(1+\mathfrak{p}^{h})\mathcal{O}_{F}^{\times 2}F^{\times 2}=\mathcal{O}_{F}^{\times}F^{\times 2}=\mathcal{O}_{F}^{\times}F^{\times 2}\cap (1+\mathfrak{p}^{h})F^{\times 2}$. 
	 
	 (ii) By an elementary argument, if $H_{1},H_{2},H_{3}$ are the subgroups of an abelian group $ G$ and $H_{3}\subseteq H_{2}$, then $(H_{1}\cap H_{2})H_{3}=H_{1}H_{3}\cap H_{2}$. Then the assertion follows by (i) and taking $G=F^{\times}$, $H_{1}=(1+\mathfrak{p}^{h})\mathcal{O}_{F}^{\times 2}$, $H_{2}=N(c)$ and $H_{3}=F^{\times 2}$.
	 \end{proof}

We need a modified version of Lemma \ref{lem:ga-formula}.
\begin{lem}\label{lem:ga-formula-F2}
	Let $a\in \mathscr{A}$. Then $g(a)F^{\times 2}=  \mathcal{O}_{F}^{\times}F^{\times 2}\cap (1+\mathfrak{p}^{\alpha(a)})F^{\times 2}\cap N(-a)$. Precisely,
	\begin{enumerate}[itemindent=-0.5em,label=\rm (\roman*)]
		\item If $a\in \mathscr{H}$, then $g(a)F^{\times 2}=\mathcal{O}_{F}^{\times}F^{\times 2}$.
		
		\item If $a\not\in \mathscr{H}$, then $g(a)F^{\times 2}=(1+\mathfrak{p}^{\alpha(a)})F^{\times 2}\cap N(-a)$.
	\end{enumerate}
\end{lem}
\begin{proof}
	The equality follows by Lemma \ref{lem:ga-formula} and Proposition \ref{prop:ga-formula-F2}(ii).
	
	(i) If $a\in \mathscr{H}$, then $\ord(a)=-2e$ and $d(-a)=\infty$. Hence $\alpha(a)=0$. So $g(a)F^{\times 2}=((1+\mathfrak{p}^{0})\mathcal{O}_{F}^{\times 2}\cap N(1))F^{\times 2}=\mathcal{O}_{F}^{\times}F^{\times 2} $.
	
	(ii) If $a\not\in \mathscr{H}$, since $a\in \mathscr{A}$, $\alpha(a)\ge 0$. Hence either $\alpha(a)\ge 1$, or $\alpha(a)=0$ and $d(-a)=2e$. In the former case, $\mathcal{O}_{F}^{\times}F^{\times 2}\cap(1+\mathfrak{p}^{\alpha(a)})F^{\times 2}=(1+\mathfrak{p}^{\alpha(a)})F^{\times 2} $; in the latter case, $N(-a)=N(\Delta) =\mathcal{O}_{F}^{\times}F^{\times 2}$. Hence $  \mathcal{O}_{F}^{\times}F^{\times 2}\cap (1+\mathfrak{p}^{\alpha(a)})F^{\times 2}\cap N(-a)=  (1+\mathfrak{p}^{\alpha(a)})F^{\times 2}\cap N(-a)$.
\end{proof}

In the remainder of this section, we let $a=\pi^{R}\varepsilon$ with $R\in \mathbb{Z}$ and $\varepsilon\in \mathcal{O}_{F}^{\times}$. Put $S\coloneqq e-R/2$,
\begin{align*}
	\mathscr{S}&=\mathscr{S}_{F}\coloneqq\{a\in \mathscr{A}\mid d(-a)>S\},\\
	\mathscr{S}^{i}&=\mathscr{S}_{F}^{i}\coloneqq\{a\in \mathscr{S}\mid R\le 2e\;\text{and}\;S\equiv i\pmod{2}\}.
\end{align*}
If $a\in \mathscr{H}$, then $R=\ord(a)=-2e$ and so $S=2e$. Hence $d(-a)=\infty>S$ and thus $a\in \mathscr{S}$. For $a\in \mathscr{S}$, we define
\begin{align}\label{phi-a}
	\phi_{a}=\phi(a)\coloneqq\begin{cases}
		a\pi^{-2e}  &\text{if $R>2e$},\\
		a\pi^{-R-2\lfloor S/2\rfloor}  &\text{if $R\le 2e$}, 
	\end{cases}
\end{align}
which is a map from $\mathscr{S}$ to $\mathscr{A}$, by Proposition \ref{prop:phia-a} below.
\begin{prop}\label{prop:phia-a}
	Let $a \in \mathscr{S}$.
	\begin{enumerate}[itemindent=-0.5em,label=\rm (\roman*)]
		\item  We have $\phi_{a}\in aF^{\times 2}$, $\phi_{a}\in \mathscr{A}$, $d(-\phi_{a})=d(-a)$ and $N(-\phi_{a})=N(-a)$.
		
		\item 	If $R>2e$, then $\alpha(\phi_{a})=\min\{R/2,R-2e+d(-a)\}$.
		
		\item If $R\le 2e$, then  $R$ is even, $S\in \mathbb{Z}\cap [0,2e]$ and $\alpha(\phi_{a})=\min\{e-\lfloor S/2\rfloor, d(-a)-2\lfloor S/2\rfloor\}$. 
	\end{enumerate}
\end{prop}

\begin{proof}
	If $R>2e$, then $b\coloneqq\phi_{a}=\pi^{R-2e}\varepsilon$. It follows that $b\in aF^{\times 2}$ and thus $d(-b)=d(-a)$ and $N(-b)=N(-a)$. From definition, $\alpha(b)=\min\{R/2,R-2e+d(-a)\}$. Since $R>2e$, $\alpha(b)\ge 0$. So, by \eqref{alpha-A-equiv}, $b\in \mathscr{A}$.

	If $R\le 2e$, then $c\coloneqq\phi_{a}=\pi^{-2\lfloor S/2\rfloor}\varepsilon$.  Since $a\in \mathscr{A}$, $R\ge -2e$. This combined with $R\le 2e$ shows that $0\le S\le 2e$. Hence $d(-a)>S\ge 0$ and thus $d(-a)\ge 1$. So $R=\ord(a)$ must be even and thus $S=e-R/2\in \mathbb{Z}$, as desired.
	
	From the parity of $R$, we have $c\in aF^{\times 2}$. Thus $d(-c)=d(-a)$ and $N(-c)=N(-a)$. 
	
	From definition, $\alpha(c)=\min\{e- \lfloor S/2\rfloor, d(-a)-2\lfloor S/2\rfloor\}$ is clearly an integer. Since $ 0\le S\le 2e$ and  $d(-a)>S$, we see that
	\begin{align*}
		e-\lfloor S/2\rfloor&\ge e-(2e/2)=0, \\
		d(-a)-2\lfloor S/2\rfloor&\ge S-2(S/2)=0.
	\end{align*}
	By \eqref{alpha-A-equiv},  $c\in \mathscr{A}$. Now, we have shown (i), (ii) and (iii).
\end{proof}
 From Proposition \ref{prop:phia-a}(i), if $a\in \mathscr{S}$, then $\phi_{a}\in \mathscr{A}$ and thus $g(\phi_{a}) $ is defined.
\begin{lem}\label{lem:a-in-gphia}
	Let $a\in \mathscr{S}$. Then
	 \begin{enumerate}[itemindent=-0.5em,label=\rm (\roman*)]
	 	\item  If $\ord(a)=-2e$ or $\ord(\phi_{a})=-2e$, then $\phi_{a}=a$. Thus $a\in \mathscr{H}$ if and only if $ \phi_{a}\in \mathscr{H}$.
	 	
	 	\item 	If $R\le 2e$, then $a\in g(\phi_{a})$. Thus $\langle a\rangle g(\phi_{a})F^{\times 2}=g(\phi_{a})F^{\times 2}$.
	 	
	 	\item If $a\in \mathscr{H}$, then $\langle a\rangle g(\phi_{a})F^{\times 2}=g(\phi_{a})F^{\times 2}= \mathcal{O}_{F}^{\times}F^{\times 2} $.
	 \end{enumerate}
\end{lem}
\begin{proof}
	(i) If $R=\ord(a)=-2e<2e$, then $S=e-R/2=2e$ and thus $\phi_{a}=a\pi^{-R-2\lfloor S/2\rfloor}=a\in \mathscr{H} $.  If $ \ord(\phi_{a})=-2e$, then $R-2e=-2e$ or $-2\lfloor S/2 \rfloor=-2e$, according as $R>2e$ or not. In the first case, we have $0=R>2e$, which is impossible. Thus, the second case must hold, implying $R\le 2e$ and $ S\ge 2\lfloor S/2\rfloor=2e$. By Proposition \ref{prop:phia-a}(iii), we also have $S\le 2e$. Hence $S=2e$, which yields $ R=-2e$. Therefore, $a=\phi_{a}\pi^{R+2\lfloor S/2\rfloor}=\phi_{a}$.
	
	 The equivalence follows immediately from this equality and the definition of $\mathscr{H}$.
	
	 (ii) Clearly, $a\in  N(-a)=N(-\phi_{a})$. By Lemma \ref{lem:ga-formula}, it suffices to show that $d(a)\ge \alpha(\phi_{a}) $. First, by Proposition \ref{prop:phia-a}(iii), we have $S\ge 0$. Also, note that $-1=1-2=1+\delta\pi^{e}$ for some $\delta\in \mathcal{O}_{F}^{\times}$ and thus $ d(-1)\ge e$. Hence
	 \begin{align*}
	 	d(-a)\ge d(-a)-S\ge \alpha(\phi_{a}) \quad \text{and}\quad d(-1)\ge e\ge e-\lfloor S/2\rfloor \ge \alpha(\phi_{a}).
	 \end{align*}
	 So, by the domination principle, $d(a)\ge \alpha(\phi_{a})$, as required.
	
	 (iii) If $a\in \mathscr{H}$, then $R=-2e<2e$. So, by (ii), (i) and Lemma \ref{lem:ga-formula-F2}(i), 
	 \begin{align*}
	 	\langle a\rangle g(\phi_{a})F^{\times 2}=g(\phi_{a})F^{\times 2}=g(a)F^{\times 2}=\mathcal{O}_{F}^{\times}F^{\times 2}.
	 \end{align*}
\end{proof}
To better compute the function $G$, we collect some results linking the functions $G$ and $g$, instead of clarifying its complicated definition.

\begin{prop}\label{prop:G-g}
	Let $a\in F^{\times}$.
	\begin{enumerate}[itemindent=-0.5em,label=\rm (\roman*)]
		\item  If $a\in \mathscr{S}$, then $G(a)=\langle a\rangle g(\phi_{a})F^{\times 2}$.
		In particular, 		
		\begin{enumerate}[itemindent=-0.5em,label=\rm (\roman*)]
			\item[(a)] if $R>4e$, then $G(a)=\langle a\rangle F^{\times 2}$;
			
			\item[(b)] if $R\le 2e$, then $G(a)=g(\phi_{a})F^{\times 2}$;
			
			\item[(c)] if $a\in \mathscr{H}$, then $G(a)=\mathcal{O}_{F}^{\times}F^{\times 2}$.
	 	\end{enumerate}  
	 	\item If $a\not\in \mathscr{S}$, then $G(a)=N(-a)$. 
	 	
	 	\item In all cases, $G(a)\subseteq N(-a)$.
	\end{enumerate} 
\end{prop}
\begin{proof}
	(i) If $a\in \mathscr{H}$, from \cite[Definition 4]{beli_integral_2003} and Lemma \ref{lem:a-in-gphia}(iii), $G(a)=\mathcal{O}_{F}^{\times}F^{\times 2}=\langle a\rangle g(\phi_{a})F^{\times 2}$.
	
	Assume that $a\in \mathscr{S}\backslash \mathscr{H}$. If $R>4e$, then $\phi_{a}=a\pi^{-2e}$ and thus $\ord(\phi_{a})=R-2e>2e$. Hence, from \cite[Definition 4(I)]{beli_integral_2003} and  \cite[Definition 6(I)]{beli_integral_2003}, $G(a)=\langle a\rangle F^{\times 2}=\langle a\rangle \mathcal{O}_{F}^{\times 2}F^{\times 2}=\langle a\rangle g(\phi_{a})F^{\times 2}$. If $2e<R\le 4e$, by \cite[p.\hskip 0.1cm 137 and Lemma 3.13(i)]{beli_integral_2003}, $G(a)=\langle a\rangle G^{'}(a)=\langle a\rangle g(\phi_{a})F^{\times 2}$. If $R\le 2e$,  by \cite[Lemma 3.13(ii)]{beli_integral_2003} and Lemma \ref{lem:a-in-gphia}(ii), $G(a)=g(\phi_{a})F^{\times 2}=\langle a\rangle g(\phi_{a})F^{\times 2}$. The proof of (i) is completed.
	
	 (ii) For $a\in F^{\times}\backslash \mathscr{A}$, see the first definition in \cite[Definition 4]{beli_integral_2003}; for $a\in \mathscr{A}\backslash \mathscr{S}$, see \cite[Definition 4(I) and (III)(iv)]{beli_integral_2003}. 
	 
	 (iii) By (ii), we may let $a\in \mathscr{S}$. Then, by (i), $G(a)= \langle a\rangle g(\phi_{a})F^{\times 2}$.
	 
	 For any $c\in  N(-a)$, we have  $(ac,-a)_{\mathfrak{p}}=(c,-a)_{\mathfrak{p}}=1$, so $ac\in N(-a)$. Thus $\langle a\rangle N(-a)\subseteq N(-a)$. Hence, by Lemma \ref{lem:ga-formula-F2} and Proposition \ref{prop:phia-a}(i), 
	\begin{align*}
		 \langle a\rangle g(\phi_{a})F^{\times 2}\subseteq \langle a\rangle N(-\phi_{a})=\langle a\rangle N(-a)\subseteq N(-a),
	\end{align*}
     as desired.
\end{proof}
 
We present two propositions to emphasize that the piece-by-piece approach to spinor norms would also be effective for other related topics, such as spinor exceptions and the class number of lattices, although these are beyond the topic of this paper.
 
\begin{prop}\label{prop:Ga-Norm}
	Let  $a,c\in F^{\times}$. 
	Then $G(a)\subseteq N(-c)$ if and only if the following conditions hold:
	\begin{enumerate}[itemindent=-0.5em,label=\rm (\roman*)]
		 \item If $a\in \mathscr{S}$, then 	$(a,-c)_{\mathfrak{p}}=1$ and one of the following conditions holds:
		 
		 	\begin{enumerate}[itemindent=-0.5em,label=\rm (\roman*)]
		 
		 \item[(a)] $a\in \mathscr{H}$ and $d(-c)\ge 2e$; 
		 
		 \item[(b)] $\max\{d(-c),d(ca)\}>2e-\alpha(\phi_{a})$.
					\end{enumerate}	
			\item If $a\not\in \mathscr{S}$, then $-c\in  F^{\times 2}\cup (-a)F^{\times 2}$.			
		\end{enumerate}	
\end{prop}
\begin{proof}
 	 (i) First, by Proposition \ref{prop:G-g}(i), $G(a)=\langle a\rangle g(\phi_{a})F^{\times 2}$. By dualization and \eqref{Na},  
 	 \begin{align}\label{equiv-a}
 	 	\langle a \rangle F^{\times 2} \subseteq N(-c) \Longleftrightarrow  -c\in (\langle a\rangle F^{\times 2} )^{\perp}=N(a) \Longleftrightarrow (-c,a)_{\mathfrak{p}}=1.
 	 \end{align}
 	 
 	 For (a), if $a\in \mathscr{H}$, by Lemma \ref{lem:a-in-gphia}(i), $\phi_{a}\in \mathscr{H}$. Hence, by Lemma \ref{lem:ga-formula-F2}(i), $g(\phi_{a})F^{\times 2}=\mathcal{O}_{F}^{\times}F^{\times 2}$. By \cite[Lemma 1.2(i) and (iii)]{beli_integral_2003} with $k=1$,
 	 \begin{align}\label{N-aOF}
 	 	\mathcal{O}_{F}^{\times }F^{\times 2}= (1+\mathfrak{p})F^{\times 2}\subseteq N(-c) \Longleftrightarrow d(-c)\ge 2e. 
 	 \end{align} 
 	 So the assertion follows by \eqref{equiv-a} and \eqref{N-aOF}. 
 	 
 	  For (b), if $a\not\in \mathscr{H}$, again by Lemma \ref{lem:a-in-gphia}(i), $\phi_{a}\not\in \mathscr{H}$. Hence, by Lemma \ref{lem:ga-formula-F2}(ii) and Proposition \ref{prop:phia-a}(i), 
 	 \begin{align*}
 	 	g(\phi_{a})F^{\times 2}=(1+\mathfrak{p}^{\alpha(\phi_{a})})F^{\times 2}\cap N(-\phi_{a})=(1+\mathfrak{p}^{\alpha(\phi_{a})})F^{\times 2}\cap N(-a).
 	 \end{align*} 
 	 So, by \cite[Lemma 1.3(i)]{beli_integral_2003} with $H=(1+\mathfrak{p}^{\alpha(\phi_{a})})F^{\times 2}$ and \cite[Lemma 1.2(iii)]{beli_integral_2003} with $k=\alpha(\phi_{a})$, 
	 \begin{align}\label{equiv-chia}
	 	\begin{split}
	 		g(\phi_{a})F^{\times 2}\subseteq N(-c) &\Longleftrightarrow   (1+\mathfrak{p}^{\alpha(\phi_{a})})F^{\times 2}\cap N(-a)\subseteq N(-c)\\			
	 		&\Longleftrightarrow  (1+\mathfrak{p}^{\alpha(\phi_{a})})F^{\times 2}\subseteq N(-c)\quad\text{or}\quad(1+\mathfrak{p}^{\alpha(\phi_{a})})F^{\times 2}\subseteq N(ca)  \\
	 		&\Longleftrightarrow   d(-c)+\alpha(\phi_{a})>2e\quad\text{or}\quad d(ca)+\alpha(\phi_{a})>2e \\
	 		&\Longleftrightarrow  \max\{d(-c),d(ca)\}>2e-\alpha(\phi_{a}).
	 	\end{split}
	 \end{align}
  So the assertion follows by \eqref{equiv-a} and \eqref{equiv-chia}.
  
  (ii) By Proposition \ref{prop:G-g}(ii) and \eqref{Na}, $G(a)=N(-a)=(\langle -a\rangle F^{\times 2} )^{\perp}$. Then, by dualization,
	 \begin{align*}
	 	(\langle -a\rangle F^{\times 2} )^{\perp}\subseteq  N(-c)\Longleftrightarrow -c\in \langle -a\rangle F^{\times 2}=F^{\times 2}\cup (-a)F^{\times 2},
	 \end{align*}
	  as desired.
\end{proof}
 
 \begin{prop}\label{prop:Ga-OF}
 	Let $a\in F^{\times}$. Then $G(a)\supseteq \mathcal{O}_{F}^{\times}F^{\times 2}$ if and only if $d(-a)\ge 2e$ and either $R\le 4-2e$, or $e=f=1$, i.e., $F=\mathbb{Q}_{2}$, and $R=4$.
 \end{prop}
 \begin{proof}
   \textbf{Necessity.} By Proposition \ref{prop:G-g}(iii) and the hypothesis, $N(-a)\supseteq G(a)\supseteq \mathcal{O}_{F}^{\times}F^{\times 2}$. So, by \eqref{N-aOF}, $d(-a)\ge 2e$ and thus $R$ is even.
 	
 	 If $a\not\in\mathscr{S}$, then $2e\le d(-a)\le S=e-R/2$, so $R\le -2e<4-2e$.
 	 
 	Assume that $a\in \mathscr{S}$. If $R>4e$, by Proposition \ref{prop:G-g}(i)(a), $G(a)=\langle a\rangle F^{\times 2}$. But note that $\langle a\rangle F^{\times 2}=aF^{\times 2}\cup F^{\times 2}\nsupseteq\mathcal{O}_{F}^{\times}F^{\times 2}$, a contradiction. Therefore, $R\le 4e$. 
 	 
 	 If $R\le 2e$, by Proposition \ref{prop:G-g}(i)(b) and Lemma \ref{lem:ga-formula-F2}, $G(a)=g(\phi_{a})F^{\times 2}=\mathcal{O}_{F}^{\times}F^{\times 2}\cap (1+\mathfrak{p}^{\alpha(\phi_{a})})F^{\times 2}\cap N(-\phi_{a})$. Hence, by  \eqref{N-aOF},
 	 \begin{align*}
 	 	g(\phi_{a})F^{\times 2} \supseteq \mathcal{O}_{F}^{\times }F^{\times 2}&\Longleftrightarrow \begin{cases}
 	 		N(-\phi_{a})\supseteq \mathcal{O}_{F}^{\times}F^{\times 2},  \\
 	 		 (1+\mathfrak{p}^{\alpha(\phi_{a})})F^{\times 2} \supseteq \mathcal{O}_{F}^{\times}F^{\times 2}= (1+\mathfrak{p})F^{\times 2}.
 	 	\end{cases}\\
 	 	&\Longleftrightarrow \begin{cases}
 	 		d(-a)=d(-\phi_{a})\ge 2e,  \\
 	 		\alpha(\phi_{a})\le 1.
 	 	\end{cases}
 	 \end{align*} 
 	 By Proposition \ref{prop:phia-a}(iii), $S\le 2e$ and thus $\lfloor S/2\rfloor\le e$. So
 	 \begin{align*}
 	 	  e-\lfloor S/2\rfloor \le 2e-2\lfloor S/2\rfloor \le d(-a)-2\lfloor S/2\rfloor.
 	 \end{align*}
    It follows that $e-S/2\le e-\lfloor S/2\rfloor=\alpha(\phi_{a})\le 1 $. Recall that $S=e-R/2$ and hence $R\le 4-2e$.
    
    Suppose that $2e<R\le 4e$. If $R>4-2e$, recall that $R$ is even, so $R\ge 6-2e$. By \cite[Lemma 3.8]{beli_integral_2003}, $\pi^{6-2e}\varepsilon\in \mathscr{A}$ and $ G(\pi^{6-2e}\varepsilon)\supseteq G(\pi^{R}\varepsilon)\supseteq \mathcal{O}_{F}^{\times }F^{\times 2}$. Assume that $e\ge 2$. Then $4-2e<6-2e<2e$. In the case $R\le 2e$, as previously shown, $G(\pi^{6-2e}\varepsilon)\supsetneq \mathcal{O}_{F}^{\times}F^{\times 2}$, a contradiction. So $e=1$.

    Now, we have $2<R\le 4$ and thus $R=4$ from the parity of $R$. Since $d(-a)\ge 2=2e$, by \eqref{N-aOF},  
    \begin{align}\label{1+p2}
    	 (1+\mathfrak{p}^{2})F^{\times 2}\subseteq (1+\mathfrak{p} )F^{\times 2}= \mathcal{O}_{F}^{\times}F^{\times 2}\subseteq N(-a)=N(-\phi_{a}).
    \end{align}
    By Proposition \ref{prop:phia-a}(ii), $\alpha(\phi_{a})=2$. Hence, by Lemma \ref{lem:ga-formula-F2} \eqref{1+p2} and \cite[Lemma 1.2(i)]{beli_integral_2003},
    \begin{align*}
    	 g(\phi_{a})F^{\times 2}=\mathcal{O}_{F}^{\times}F^{\times 2}\cap (1+\mathfrak{p}^{\alpha(\phi_{a})})F^{\times 2}\cap N(-\phi_{a})=(1+\mathfrak{p}^{2})F^{\times 2}= \langle\Delta\rangle F^{\times 2}.
    \end{align*}
   So, by Proposition \ref{prop:G-g}(i) and the hypothesis,
    \begin{align*}
    	G(a)=\langle a\rangle g(\phi_{a})F^{\times 2}=\langle a\rangle \langle\Delta\rangle F^{\times 2}\supseteq \mathcal{O}_{F}^{\times}F^{\times 2}.
    \end{align*}
     Hence, by \cite[63:9]{omeara_quadratic_1963},
    \begin{align*}
    	   2N\mathfrak{p}=|\mathcal{O}_{F}^{\times}F^{\times 2}/F^{\times 2}|\le |\langle a\rangle \langle \Delta\rangle F^{\times 2 }/F^{\times 2}|\le 4.
    \end{align*}
     This implies that $2^{f}=N\mathfrak{p}\le 2$ and thus $f=1$, as desired.
    
     \textbf{Sufficiency.} If $e=f=1$ and $R=4$, then $F=\mathbb{Q}_{2}$, $\Delta\in 5\mathbb{Q}_{2}^{\times 2}$ and $-a\in \mathbb{Q}_{2}^{\times 2}\cup 5\mathbb{Q}_{2}^{\times 2}$ (as $d(-a)\ge 2e$). So $
     	 (1+\mathfrak{p}^{2})\mathbb{Q}_{2}^{\times 2}=(1+4\mathbb{Z}_{2})\mathbb{Q}_{2}^{\times 2}=\langle 5\rangle \mathbb{Q}_{2}^{\times 2}$ and
     \begin{align*}
     	\langle a\rangle (1+\mathfrak{p}^{2})F^{\times 2}=\langle a,5\rangle \mathbb{Q}_{2}^{\times 2}=\langle -1,5\rangle \mathbb{Q}_{2}^{\times 2} \quad\text{or}\quad \langle -5,5\rangle \mathbb{Q}_{2}^{\times 2},
     \end{align*}
      which equals to $\langle -1,5\rangle \mathbb{Q}_{2}^{\times 2} $ in both cases. Hence, from \cite[Definition 4(II)(iii)]{beli_integral_2003},
    \begin{align*}
    	G(a)=\langle a\rangle (1+\mathfrak{p}^{2})F^{\times 2}=\langle -1,5\rangle   \mathbb{Q}_{2}^{\times 2}= \mathbb{Z}_{2}^{\times}\mathbb{Q}_{2}^{\times 2}.
    \end{align*}
    Assume that $R\le 4-2e$. If $R<-2e$, then $a\not\in \mathscr{A}$. Then, by Proposition \ref{prop:G-g}(ii), $G(a)=N(-a)$. Since $d(-a)\ge 2e$, \eqref{N-aOF} implies $G(a)\supseteq \mathcal{O}_{F}^{\times}F^{\times 2}$; if $-2e\le R\le 4-2e$, then $S\ge e-(4-2e)/2=2e-2$, so $e-\lfloor S/2\rfloor \le e-\lfloor (2e-2)/2\rfloor=1$. Hence, from \cite[Definition 4(III)(vi)]{beli_integral_2003}, $G(a)=(1+\mathfrak{p}^{e-\lfloor S/2\rfloor})F^{\times 2}\supseteq (1+\mathfrak{p})F^{\times 2}= \mathcal{O}_{F}^{\times}F^{\times 2}$.
\end{proof}
From \cite[Lemma 7.2(i)]{beli_integral_2003} and \cite[Lemma 2.3.2(ii)]{beli_thesis_2001}, we also have equivalent conditions on the reverse inclusion.
\begin{prop}\label{prop:GaOF-equiv}
	Let $a\in F^{\times}$. Then $G(a)\subseteq \mathcal{O}_{F}^{\times }F^{\times 2}$ if and only if $R$ is even and the following conditions holds:
	\begin{enumerate}[itemindent=-1em,label=\rm (\roman*)]
		\item If $a\in \mathscr{A}$, then $d(-a)=2e=-R  $ or $d(-a)>S$, i.e., $a\in \mathscr{S}$.

		\item If $a\not\in \mathscr{A}$, then 	$d(-a)=2e<-R$.
	\end{enumerate}
	Thus either $a\in \mathscr{A}$ or $d(-\varepsilon)=d(-a)=2e$ holds.
\end{prop}

In what follows, we study the invariants in BONGs theory. First, let us recall the definition of a BONG and the equivalent conditions for a sequence of vectors to be a good BONG (cf. \cite[Definition 2]{beli_integral_2003} and \cite[Lemma 2.2]{HeHu2}).

 \begin{defn}\label{defn:bong}
	Let $ M $ be an $ \mathcal{O}_{F} $-lattice on a quadratic space $V$. The elements $ x_{1},\ldots,x_{m}\in V$ is called a basis of norm generators (abbr. BONG) for $ M $ if
	\begin{enumerate}[itemindent=-1em,label=\rm (\roman*)]
		\item  $ \mathfrak{n}(M)=Q(x_{1})\mathcal{O}_{F} $ and
		
		\item $ x_{2},\ldots,x_{m} $ is a BONG for $ \pr_{x_{1}^{\perp}}M $,  where $ \pr_{x_{1}^{\perp}} :FM\to (Fx_{1})^{\perp} $ denotes the projection map.
	\end{enumerate}
 	If moreover $ \ord (Q(x_{i}))\le \ord (Q(x_{i+2})) $ for all $ 1\le i\le m-2 $, then the BONG $ x_{1},\ldots,x_{m} $ is called good. 
\end{defn}

Let $a_{i}=Q(x_{i})$. We write $V\cong [a_{1},\ldots,a_{m}]$ if $V=Fx_{1}\perp\ldots \perp Fx_{m}$, and write $M\cong \prec a_{1},\ldots,a_{m}\succ$ if $x_{1},\ldots, x_{m}$ is a BONG for $M$. Note from Definition \ref{defn:bong} that $[a_{1},\ldots,a_{m}]=F\prec a_{1},\ldots,a_{m}\succ$.
 \begin{lem}\label{lem:2.2}
		Let $ x_{1},\ldots, x_{m} $ be pairwise orthogonal vectors of a quadratic space with $ Q(x_{i})=a_{i} $ and $ R_{i}=\ord (a_{i})$. Then $ x_{1},\ldots,x_{m} $ is a good BONG for some lattice if and only if
		\begin{equation}\label{eq:GoodBONGs}
			R_{i}\le R_{i+2} \quad \text{for all }\;  1\le i\le m-2,
		\end{equation}
		and $a_{i+1}/a_{i}\in \mathscr{A}$ for all $1\le i\le m-1$, i.e.,
		\begin{equation}\label{eq:BONGs}
			R_{i+1}-R_{i}+2e\ge 0 \quad\text{and}\quad  R_{i+1}-R_{i}+d(-a_{i}a_{i+1})\ge  0 \quad \text{for all }\; 1\le i\le m-1.
		\end{equation}
	\end{lem}
\begin{re}\label{re:Ri+Ri+2-increasing}
	 From \eqref{eq:GoodBONGs}, if $1\le i\le j\le m-1$, then $R_{i}+R_{i+1}\le  R_{j}+R_{j+1} $.
\end{re}

\medskip

	In the remainder of this section, let $M\cong \prec a_{1},\ldots, a_{m} \succ $ be an $ \mathcal{O}_{F} $-lattice relative to some good BONG.

\begin{defn}\label{defn:alpha}
		For $1\le i\le m$, we define the $R_{i}$-invariant $R_{i}=R_{i}(M)\coloneqq \ord(a_{i})$.

		For $ 1\le i\le m-1 $, we put
		\begin{align}\label{T}
		 T_{j}^{(i)}=T_{j}^{(i)}(M)\coloneqq
			\begin{cases}
				(R_{i+1}-R_{i})/2+e  &\text{if $ j=0$}, \\
				R_{i+1}-R_{j}+d(-a_{j}a_{j+1}) &\text{if $ 1\le j\le i $}, \\
			 R_{j+1}-R_{i}+d(-a_{j}a_{j+1})  &\text{if $ i\le j\le m-1 $},
			\end{cases}
		\end{align}
		and define the $\alpha_{i}$-invariant $ \alpha_{i}=\alpha_{i}(M)\coloneqq\min\{T_{0}^{(i)},\ldots, T_{m-1}^{(i)}\} $.
\end{defn}
\begin{re}\label{re:alpha-binary}
		In particular, if $L$ is a binary lattice, then $R_{2}(L)-R_{1}(L)=\ord(a(L))$ and $\alpha_{1}(L)=\alpha(a(L))$.
\end{re}
Let $  c_{1},c_{2},\ldots \in F^{\times} $. For $ 1\le i\le j+1 $, we write $ c_{i,j}=c_{i}\cdots c_{j} $ for short and set $ c_{i,i-1}=1 $. For $ 0\le i-1\le j\le m $ and $c\in F^{\times}$, we define
\begin{equation}\label{defn:d[]}
  d[ca_{i,j}]\coloneq  \min\{d(ca_{i,j}),\alpha_{i-1},\alpha_{j}\}.
\end{equation}	Here, if $ i-1\in \{0,m\}$, $ \alpha_{i-1} $ is ignored; if $ j\in\{0,m\}$, $ \alpha_{j} $ is ignored. 
 By \cite[Corollary\;2.5(i)]{beli_Anew_2010}, the invariants $\alpha_{i}$ can be reformulated as
\begin{align}\label{eq:alpha-defn}
	\alpha_{i}=\min\{(R_{i+1}-R_{i})/2+e,R_{i+1}-R_{i}+d[-a_{i,i+1}]\}\,.
\end{align}

Let us recall some useful properties of the invariants  $ R_{i} $ and $ \alpha_{i} $ \cite[Propositions 2.2 and 2.3]{He22}.  
\begin{prop}\label{prop:Rproperty}
	 Let $ 1\le i\le m-1 $. 		
		\begin{enumerate}[itemindent=-0.5em,label=\rm (\roman*)]
			\item  $ R_{i+1}-R_{i}>2e$ (resp. $ =2e $, $ <2e $) if and only if $ \alpha_{i}>2e $ (resp. $ =2e $, $ <2e $).
			
			\item   If $ R_{i+1}-R_{i}\ge 2e$ or $ R_{i+1}-R_{i}\in \{-2e,2-2e,2e-2\} $, then $ \alpha_{i}=(R_{i+1}-R_{i})/2+e $.
			
			\item If $ R_{i+1}-R_{i}\le 2e $, then $ \alpha_{i}\ge R_{i+1}-R_{i} $, and the equality holds if and only if $ R_{i+1}-R_{i}=2e $ or $ R_{i+1}-R_{i} $ is odd. 
		\end{enumerate} 
	\end{prop}
	\begin{prop}\label{prop:alphaproperty}
		Let $ 1\le i\le m-1 $.
			\begin{enumerate}[itemindent=-0.5em,label=\rm (\roman*)]
			\item  Either $ 0\le \alpha_{i}\le 2e $ and $ \alpha_{i}\in \mathbb{Z} $, or $ 2e<\alpha_{i}<\infty $ and $ 2\alpha_{i}\in \mathbb{Z} $; thus $ \alpha_{i}\ge 0 $.
			
			\item   $ \alpha_{i}=0 $ if and only if $ R_{i+1}-R_{i}=-2e $.
			
			\item $ \alpha_{i}=1 $ if and only if either $ R_{i+1}-R_{i}\in \{2-2e,1\} $, or $ R_{i+1}-R_{i}\in [4-2e,0]^{E} $ and $ d[-a_{i}a_{i+1}]=R_{i}-R_{i+1}+1 $.
			
			\item  If $ \alpha_{i}=0 $, i.e., $ R_{i+1}-R_{i}=-2e $, then $ d[-a_{i}a_{i+1}]\ge 2e $.
			
			\item If $ \alpha_{i}=1 $, then $ d[-a_{i}a_{i+1}]\ge R_{i}-R_{i+1}+1 $, and the equality holds if $ R_{i+1}-R_{i}\not=2-2e $.
		\end{enumerate} 
  
	\end{prop}

We present a simple but new proposition that is the key to proving the stability of the invariant $\alpha_{i}$ under unramified extensions (cf. Proposition \ref{prop:alpha-extension}(iii)).

\begin{prop}\label{prop:T}
	Let $1\le i,j\le m-1$. If $d(-a_{j}a_{j+1})\ge 2e$, then $T_{j}^{(i)}\ge T_{0}^{(i)}$.
\end{prop}
\begin{proof}
	 If $j\le i$, then, by \eqref{eq:BONGs} and Remark \ref{re:Ri+Ri+2-increasing},
	\begin{align*}
		-R_{j}+d(-a_{j}a_{j+1})\ge -R_{j}+2e\ge -(R_{j}+R_{j+1})/2+e\ge -(R_{i}+R_{i+1})/2+e.
	\end{align*}
	 So
	\begin{align*}
		 T_{j}^{(i)}&=R_{i+1}-R_{j}+d(-a_{j}a_{j+1}) \ge R_{i+1}-(R_{i}+R_{i+1})/2+e=T_{0}^{(i)}.
	\end{align*}
	If $j\ge i$, then, again by \eqref{eq:BONGs} and Remark \ref{re:Ri+Ri+2-increasing},
	\begin{align*}
		R_{j+1}+d(-a_{j}a_{j+1})\ge R_{j+1}+2e\ge (R_{j}+R_{j+1})/2+e\ge (R_{i}+R_{i+1})/2+e.
	\end{align*}
	 So
	\begin{align*}
		  T_{j}^{(i)}&=R_{j+1}-R_{i}+d(-a_{j}a_{j+1})\ge (R_{i}+R_{i+1})/2+e-R_{i}=T_{0}^{(i)}.
	\end{align*}
\end{proof}
 
	Let $M\cong \prec a_{1},\ldots,a_{m}\succ$ and $N\cong \prec b_{1},\ldots, b_{n}\succ$ relative to some good BONGs $x_{1},\ldots,x_{m}$ and $y_{1},\ldots,y_{n}$ with $a_{i}=\pi^{R_{i}}\varepsilon_{i}$ and $b_{i}=\pi^{S_{i}}\eta_{i}$, where $R_{i},S_{i}\in \mathbb{Z}$ and $\varepsilon_{i},\eta_{i}\in \mathcal{O}_{F}^{\times}$. Also, let $ \alpha_{i}=\alpha_{i}(M) $ and $ \beta_{i}=\alpha_{i}(N) $.
	For $ 0\le i,j\le m $ and $c\in F^{\times}$, we define
	\begin{equation}\label{defn:d[ab]}
		d[ca_{1,i}b_{1,j}]=\min\{d(ca_{1,i}b_{1,j}),\alpha_{i},\beta_{j}\}.
	\end{equation}
	Here if $ i\in \{0,m\} $, then $ \alpha_{i} $ is ignored; if $ j\in \{0,n\} $, $ \beta_{j} $ is ignored.

Write $P_{i,k}=P_{i,k}(M,N)\coloneqq R_{i+k}-S_{i}$. Define
\begin{equation}
	\begin{aligned}\label{A-i-M-N}
		A_{i}\coloneqq&\;A_{i}(M,N)\\
		 =&\;\begin{cases}
			\min\{P_{i,1}/2+e,P_{i,1}+d[-a_{1,i+1}b_{1,i-1}], P_{i-1,3}+P_{i,1}+d[a_{1,i+2}b_{1,i-2}] \}  \\
			\hskip 8cm\text{if $1\le i\le \min\{m-1,n\}$},\\
			\min\{P_{n+1,1}+d[-a_{1,n+2}b_{1,n}],P_{n+1,1}+P_{n,3}+d[a_{1,n+3}b_{1,n-1}]\}   \\
			\hskip 8cm\text{if $n\le m-2 $ and $i=n+1$},
		\end{cases}
	\end{aligned} 
\end{equation}
 where  the term $ P_{i-1,3}+P_{i,1}+d[a_{1,i+2}b_{1,i-2}] $ is ignored if $ i\in \{1,m-1\} $.

Now, we introduce Beli's representation theorem on lattices in dyadic local fields  (cf. \cite[Theorem 4.5]{beli_representations_2006}).

	\begin{thm}\label{thm:beligeneral}
	Suppose $  n\le m$. Then $ N\rep M $ if and only if  $ FN\rep FM $ and the following conditions hold:
			\begin{enumerate}[itemindent=-0.5em,label=\rm (\roman*)]
		\item For any $ 1\le i\le n $, we have either $ R_{i}\le S_{i} $, or $ 1<i<m $ and $ R_{i}+R_{i+1}\le S_{i-1}+S_{i} $.
		
		\item  For any $ 1\le i\le \min\{m-1,n\} $, we have $ d[a_{1,i}b_{1,i}]\ge A_{i} $.
		
		\item  For any $ 1<i\le \min\{m-1,n+1\} $ such that $R_{i+1}>S_{i-1}$ and $A_{i-1}+A_{i}>2e+R_{i}-S_{i}$, we have $ [b_{1},\ldots, b_{i-1}]\rep [a_{1},\ldots,a_{i}] $.
		
		\item For any $ 1<i\le \min\{m-2,n+1\} $ such that $ S_{i}\ge R_{i+2}>S_{i-1}+2e\ge R_{i+1}+2e$, we have $ [b_{1},\ldots,b_{i-1}]\rep [a_{1},\ldots,a_{i+1}] $. (If $ i=n+1 $, the condition $ S_{i}\ge R_{i+2} $ is ignored.)
	\end{enumerate} 
 
\end{thm}
\begin{re}
	If $n\le m-2$ and $i=n+1$, although $S_{n+1}$ is undefined, the inequality  $A_{n}+A_{n+1}>2e+R_{n+1}-S_{n+1}$ still makes sense, because it contains $-S_{n+1}$ on both sides, which will be canceled naturally. 
\end{re}

Suppose that $N\subseteq M$. Put $O^{+}(M)\coloneqq\{\sigma\in O^{+}(FM)\mid \sigma(M)=M\}$ and
\begin{align}\label{local-XMN}
 X(M/N)\coloneqq\{\sigma\in O^{+}(FM)\mid N\subseteq \sigma(M)\}.
\end{align}
 For short, we denote by $\theta$ the usual spinor norm and write $\theta(M/N)$ (resp. $\theta(M)$) for $\theta(X(M/N))$ (resp. $\theta(O^{+}(M))$).

By virtue of the properties introduced in \cite[\S 4]{beli_integral_2003}, one can separate the calculation of integral spinor norms into several cases and then use the corresponding formulas.
\begin{defn}\label{defn:A}
	 Suppose that $M$ has a Jordan splitting $M=M_{1}\perp \cdots \perp M_{t}$. We say that $M$ has property A if $\rank M_{i}\le 2$ for any $i$ and 
	\begin{align*} 
		0<\ord\, (\mathfrak{n}(M_{j}))-\ord\,(\mathfrak{n}(M_{i}))<2(\ord\, (\mathfrak{s}(M_{j}))-\ord\,(\mathfrak{s}(M_{i})))
	\end{align*}
	for any $i<j$.
\end{defn}
\begin{re}\label{re:property-A-B}
		\begin{enumerate}[itemindent=-0.5em,label=\rm (\roman*)]		
		\item  From \cite[\S 7]{beli_integral_2003}, if $M$ does not have property A, then $\theta(M)=F^{\times}$ or $\mathcal{O}_{F}^{\times}F^{\times 2}$.

		\item From \cite[Lemma 4.3(i)]{beli_integral_2003}, Property A is equivalent to $R_{i}<R_{i+2}$ for all $1\le i\le m-2$.
			\end{enumerate} 
\end{re}

Before giving Beli's formulas for $\theta(M)$ (\cite[Theorems 1 and 3]{beli_integral_2003}) and $\theta(M/N)$ (\cite[Theorem II.2 and II.3]{beli_thesis_2001}), we introduce some notations. Put
\begin{align}
	\gamma(M,N)&\coloneqq\min_{1\le i\le m-2}\{\lfloor (R_{i+2}-S_{i})/2\rfloor\}, \label{gammaMN}\\
	\intertext{and when $M=N$, put}
	\gamma(M)&\coloneqq\gamma(M,M)=\min_{1\le i\le m-2}\{\lfloor (R_{i+2}-R_{i})/2\rfloor\}.\label{gammaM}
\end{align}

\begin{thm}\label{thm:integralspinornorm-A}
	If $M$ has property A, then
	\begin{align*}
		\theta(M)=G(a_{2}/a_{1})G(a_{3}/a_{2})\cdots G(a_{m}/a_{m-1})(1+\mathfrak{p}^{\gamma})F^{\times 2},
	\end{align*}
	where $\gamma=\gamma(M)$.
\end{thm}
\begin{thm} \label{thm:integralspinornorm}
	$\theta(M)\subseteq \mathcal{O}_{F}^{\times}F^{\times 2}$ if and only if  the following conditions hold:
		\begin{enumerate}[itemindent=-0.5em,label=\rm (\roman*)]
		\item  for $1\le i\le m-1$, we have $G(a_{i+1}/a_{i})\subseteq \mathcal{O}_{F}^{\times}F^{\times 2}$;
		
		\item  for $1\le i\le m-2$ such that $R_{i}=R_{i+2}$, we have $(R_{i+1}-R_{i})/2\equiv e\pmod{2}$.
	\end{enumerate} 
 
\end{thm}

For $1\le i\le m-1$, let $\xi_{i}= \varepsilon_{1,i+1}\eta_{1,i-1}\in \mathcal{O}_{F}^{\times }$.

 \begin{thm}\label{thm:Beli-II2}
	Assume that $m-n\le 2$.
	\begin{enumerate}[itemindent=-0.5em,label=\rm (\roman*)]
		\item  	If $R_{i+2}>S_{i}$ for each $1\le i\le  m-2$, put
		\begin{align}\label{Gi}
			G_{i}=\begin{cases}
				G(\pi^{R_{i+1}-S_{i}}\xi_{i}) &\text{if $i\le n$ and $\sum_{k=1}^{i}(R_{k}-S_{k})\equiv 0\pmod{2}$}, \\
				N(-a_{1,i+1}b_{1,i-1})  &\text{otherwise}.
			\end{cases}
		\end{align}
		Then 
		\begin{align*}
			\theta(M/N)=\theta(M)G_{1}\cdots G_{m-1}(1+\mathfrak{p}^{\gamma})F^{\times 2},
		\end{align*}
		where $\gamma=\gamma(M,N)$ and we ignore the term $(1+\mathfrak{p}^{\gamma})F^{\times 2} $ if $m\le 2$.
		
		\item If $R_{j+2}\le S_{j}$ for some $1\le j\le  m-2$, then $\theta(M/N)\supseteq \mathcal{O}_{F}^{\times} F^{\times 2}$.
		
	\end{enumerate}
\end{thm}
\begin{re}\label{re:m-n3}
	When $m-n\ge 3$, $\theta(M/N)=F^{\times}$ (cf. \cite[\S 4]{hsia_indefinite_1998}).
\end{re}
For $1\le i\le \min\{m-1,n\}$, let $T_{i}=\max\{S_{i-1},R_{i+1}\}-\min\{S_{i},R_{i+2}\}$, where we ignore $S_{i-1}$ in the maximum if $i=1$ and we ignore $R_{i+2}$ in the minimum if $i=m-1$.
\begin{thm}\label{thm:Beli-II3}
 	$\theta(M/N)\subseteq \mathcal{O}_{F}^{\times}F^{\times 2}$ if and only if the following conditions hold:
	\begin{enumerate}[itemindent=-0.5em,label=\rm (\roman*)]
		 \item $R_{1}\equiv\cdots\equiv R_{m}\equiv S_{1}\equiv \cdots \equiv S_{n}\pmod{2}$.
		
		\item For any $1\le i\le n$ such that $R_{i+2}\le S_{i} $, one of conditions (a) and (b) holds:
		\begin{enumerate}[itemindent=-1em,label=\rm (\roman*)]
			\item[(a)] $R_{i+1}+R_{i+2}=S_{i}+S_{i+1} $ and $(R_{i+2}-R_{i+1})/2\equiv (S_{i+1}-S_{i})/2 \equiv e\pmod{2}$;
			
			\item[(b)] $R_{i+2}=S_{i}$ and either $R_{i+2}-R_{i+1}=-2e$ or $S_{i+1}-S_{i}=-2e$.
			
		\end{enumerate}
		\item For any $1\le i\le m-1$, we have either $d(-\xi_{i})=2e$ or $i\le n$ and $G(\pi^{T_{i}}\xi_{i})\subseteq \mathcal{O}_{F}^{\times }F^{\times 2}$.
		
		\item $\theta(M)\subseteq \mathcal{O}_{F}^{\times }F^{\times 2}$ and  $\theta(N)\subseteq \mathcal{O}_{F}^{\times }F^{\times 2}$.
	\end{enumerate}
\end{thm}
 
\section{Liftings over local fields}\label{sec:lifting}
 
 In this section, we assume that $F$ is a non-archimedean local field and $E$ is a finite extension of $F$ at the primes $\mathfrak{P}|\mathfrak{p}$. Denote by $e_{\mathfrak{P}}=\ord_{\mathfrak{P}}(2)$, $e_{\mathfrak{p}}=\ord_{\mathfrak{p}}(2)$, $e_{\pmid}=e(\mathfrak{P}|\mathfrak{p})$ the ramification index and  $f_{\pmid}=f(\mathfrak{P}|\mathfrak{p})$ the inertia degree. Clearly, $e_{\mathfrak{P}}=e_{\mathfrak{p}}e_{\pmid}$.
 
Write $ \widetilde{d}: E^{\times}/E^{\times 2} \to \mathbb{N}\cup \{\infty\} $, $ \widetilde{d}(c)\coloneqq\ord_{\mathfrak{P}}(c^{-1}\mathfrak{d}_{E}(c)) $ for the lifting over $E$ of the function $ d$. For short, also write $\ord_{\mathfrak{P}}(c)$ instead of $\ord_{\mathfrak{P}}(\tilde{c})$ for $c\in F^{\times}$, and $\ord_{\mathfrak{P}}(\mathfrak{c})$ instead of $\ord_{\mathfrak{P}}(\tilde{\mathfrak{c}})$ for any fractional ideal $\mathfrak{c}$ of $F$.

Let $M$ be an integral $\mathcal{O}_{F}$-lattice of rank $m$ and $\widetilde{M}=M\otimes_{\mathcal{O}_{F}} \mathcal{O}_{E}$. When $F$ is non-dyadic, we denote by $ J_{i}(M) $ the Jordan component of $M$, with possible zero rank and $ \mathfrak{s}(J_{i}(M))=\mathfrak{p}^{i} $. Note that $J_{i}(M)$ is unique when $F$ is non-dyadic. When $F$ is dyadic, we let $R_{i}=R_{i}(M)$, $T_{j}^{(i)}=T_{j}^{(i)}(M)$ and $\alpha_{i}=\alpha_{i}(M)$ (cf. Definition \ref{defn:alpha}).
 
\begin{lem}\label{lem:defectP-p}
	Suppose $F$ to be non-dyadic or dyadic. Let $c\in  F^{\times}$ and $\mathfrak{c}$ be a fractional ideal of $F$. Then 
	\begin{enumerate}[itemindent=-0.5em,label=\rm (\roman*)]
		\item $\ord_{\mathfrak{P}}(c)=\ord_{\mathfrak{p}}(c)e_{\pmid}$ and	$\ord_{\mathfrak{P}}(\mathfrak{c})=\ord_{\mathfrak{p}}(\mathfrak{c})e_{\pmid}$.
		
		\item  When $F$ is dyadic, $ \widetilde{d}(c)\ge d(c)e_{\pmid}$, and the equality holds if and only if one of the following conditions holds:
		\begin{enumerate}[itemindent=-1em,label=\rm (\roman*)]
			\item[(a)]  $d(c)<2e_{\mathfrak{p}}$ and $e_{\pmid}$ is odd;
			
			\item[(b)]  $d(c)=2e_{\mathfrak{p}}$ and $f_{\pmid}$ is odd, i.e., $c\not\in \mathcal{O}_{E}^{\times 2}$;
			
			\item[(c)]  $d(c)=\infty$.
		\end{enumerate} 
		In particular, if $[E:F]$ is odd, then $\widetilde{d}(c)=d(c)e_{\pmid}$.
	\end{enumerate} 
\end{lem}	
 	\begin{proof}
 	(i) This is clear from \cite[16:2]{omeara_quadratic_1963} and the definitions of $\ord_{\mathfrak{p}}$ and $\ord_{\mathfrak{P}}$.
 	
 	(ii)  Let $c=\eta^{2}+a$ with $d(c)=\ord_{\mathfrak{p}}(a)$. Then, by (i), $\widetilde{d}(c)\ge \ord_{\mathfrak{P}}(a)= \ord_{\mathfrak{p}}(a)e_{\pmid}=d(c)e_{\pmid}$.
 	
 	(a) If $d(c)=0$, then $\ord_{\mathfrak{p}}(c)$ is odd. Then $\widetilde{d}(c)=0$ if and only if $\ord_{\mathfrak{P}}$ is odd, i.e., $e_{\pmid}$ is odd; if $1\le d(c)<2e_{\mathfrak{p}}$, then $d(c)$ is odd and $1\le \widetilde{d}(c)<2e_{\mathfrak{P}}$. By \cite[63:5]{omeara_quadratic_1963},  $\widetilde{d}(c)=d(c)e_{\pmid}$ if and only if  $d(c)e_{\pmid}$ is odd, i.e., $e_{\pmid}$ is odd.
 	
 	(b) If $d(c)=2e_{\mathfrak{p}}$, i.e., $c\in \Delta_{F}\mathcal{O}_{F}^{\times 2}$, then $\widetilde{d}(c)\ge d(c)e_{\pmid}=2e_{\mathfrak{P}}$ and thus $c\in \Delta_{E}\mathcal{O}_{E}^{\times 2}\cup \mathcal{O}_{E}^{\times 2}$. Note from \cite[63:3]{omeara_quadratic_1963} that $F(\sqrt{c})=F(\sqrt{\Delta_{F}})$ is unramified over $F$. So $c\in \mathcal{O}_{E}^{\times 2}$ if and only if  $F(\sqrt{\Delta_{F}})\subseteq K $, where $K $ is the unique maximal unramified extension over $F$ contained in $E$ and $[K:F]=f_{\pmid}$. Hence, $c\in \mathcal{O}_{E}^{\times 2}$ if and only if $f_{\pmid}$ is even.
 	
 	(c) If $d(c)=\infty$, i.e., $c\in \mathcal{O}_{F}^{\times 2}$, then $c\in \mathcal{O}_{E}^{\times 2}$, i.e., $ \widetilde{d}(c)=\infty$. 
 	
 	 If $[E:F]$ is odd, by \cite[16:4]{omeara_quadratic_1963}, both $e_{\pmid}$ and $f_{\pmid}$ are odd. So one of conditions (a)-(c) holds, so does the equality.
 \end{proof} 
 \begin{re}\label{re:defectP-p}
 	When $F$ is non-dyadic and $c\in F^{\times}$, if $\ord_{\mathfrak{p}}(c)$ is even, applying a similar argument as Lemma \ref{lem:defectP-p}(ii)(b) and (c), we see that $\widetilde{c}\in E^{\times 2}$ if and only if $c\in F^{\times 2}$ or $c\in \Delta_{F}F^{\times 2}$ and $f_{\pmid}$ is even. 
 \end{re}

\begin{prop}\label{prop:scale-norm-volume}
	Suppose $F$ to be non-dyadic or dyadic. Then
		\begin{enumerate}[itemindent=-0.5em,label=\rm (\roman*)]
    		\item $\ord_{\mathfrak{P}}(\mathfrak{s}(\widetilde{M}))=\ord_{\mathfrak{p}}(\mathfrak{s}(M))e_{\pmid}$, $\ord_{\mathfrak{P}}(\mathfrak{n}(\widetilde{M}))=\ord_{\mathfrak{p}}(\mathfrak{n}(M))e_{\pmid}$ and $\ord_{\mathfrak{P}}(\mathfrak{v}(\widetilde{M}))=\ord_{\mathfrak{p}}(\mathfrak{v}(M))e_{\pmid}$.
    		
    		\item  $\mathfrak{s}(\widetilde{M})=\mathfrak{s}(M)\mathcal{O}_{E}$, $\mathfrak{n}(\widetilde{M})=\mathfrak{n}(M)\mathcal{O}_{E}$ and $\mathfrak{v}(\widetilde{M})=\mathfrak{v}(M)\mathcal{O}_{E}$.
    		
    		\item For any fractional ideal $\mathfrak{a}$ of $F$, $\widetilde{M}$ is $\widetilde{\mathfrak{a}}$-modular if and only if $M$ is $\mathfrak{a}$-modular.
    	\end{enumerate} 
     \end{prop}
\begin{proof}
	  (i) This follows from \cite[82:8]{omeara_quadratic_1963}, the definition of volume (cf. \cite[p.\hskip 0.1cm 229]{omeara_quadratic_1963}), the property $\ord(\mathfrak{a}+\mathfrak{b})=\min\{\ord(\mathfrak{a}),\ord(\mathfrak{b})\}$ (\cite[\S22]{omeara_quadratic_1963}) and Lemma \ref{lem:defectP-p}(i).
	  
	  (ii) This follows from (i). 
	  
	  (iii) This follows from (ii) and the property $\mathfrak{v}(M)=\mathfrak{s}(M)^{m}$ of modular lattices.
\end{proof}

\begin{lem}\label{lem:extension-Jordansplittings}
	 Suppose $F$ to be non-dyadic. If $M$ has a Jordan splitting $ M=J_{0}(M)\perp J_{1}(M)\perp \cdots \perp J_{t}(M)$, then $\widetilde{M}$ has a Jordan splitting 
	  \begin{align*}
	 	\widetilde{M}&=J_{0}(\widetilde{M})\perp J_{e_{\pmid}}(\widetilde{M})\perp \cdots \perp J_{te_{\pmid}}(\widetilde{M})=\widetilde{J_{0}(M)}\perp \widetilde{J_{e_{\pmid}}(M)}\perp \cdots \perp \widetilde{J_{te_{\pmid}}(M)}
	 \end{align*}
	  and $\rank J_{ie_{\pmid}}(\widetilde{M})=\rank \widetilde{J_{i}(M)}=\rank J_{i}(M)$ for $0\le i\le t$.
\end{lem}
\begin{proof}
 This follows by Proposition \ref{prop:scale-norm-volume} and the fact that  $J_{ie_{\pmid}}(\widetilde{M})=\widetilde{J_{i}(M)}$.
\end{proof}

\begin{lem}\label{lem:extension-BONG}
		Suppose $F$ to be dyadic. If  $M\cong \prec a_{1},\ldots, a_{m}\succ$ relative to a good BONG $ x_{1},\ldots, x_{m} $, then $\widetilde{M}\cong \prec \widetilde{a}_{1},\ldots,\widetilde{a}_{m}\succ $ relative to the good BONG $ \widetilde{x}_{1},\ldots, \widetilde{x}_{m} $.
\end{lem}
\begin{proof}
	See \cite[Lemma 8.1(iii)]{He22}.
\end{proof}

 In the rest of this section, we assume that $F$ is dyadic and write $\widetilde{R}_{i}$, $\widetilde{T}_{j}^{(i)}$ and $\widetilde{\alpha}_{i}$ for the liftings of $R_{i}$, $T_{j}^{(i)}$ and $\alpha_{i}$ accordingly. Namely, $\widetilde{R}_{i}=R_{i}(\widetilde{M})$, $\widetilde{T}_{j}^{(i)}=T_{j}^{(i)}(\widetilde{M})$ and $\widetilde{\alpha}_{i}=\alpha_{i}(\widetilde{M})$.

			\begin{prop}\label{prop:R-extension}
			
			Let $1\le i\le m$.
				\begin{enumerate}[itemindent=-0.5em,label=\rm (\roman*)]
				\item $\widetilde{R}_{i}=R_{i}e_{\pmid}$.
				
				\item   $\widetilde{R}_{i}=0$ if and only if $R_{i}=0$.
				
				\item  $\widetilde{R}_{i}=-2e_{\mathfrak{P}}$  if and only if $R_{i}=-2e_{\mathfrak{p}}$.
				
				\item   $\widetilde{R}_{i}=1$ if and only if $R_{i}=1$ and $e_{\pmid}=1$.	
			\end{enumerate} 	
 
		\end{prop}
	 	\begin{proof}
		For (i), from Lemma \ref{lem:extension-BONG}, we have $\widetilde{R}_{i}=\ord_{\mathfrak{P}}(a_{i})$ and $R_{i}=\ord_{\mathfrak{p}}(a_{i})$. So the equality holds by Lemma \ref{lem:defectP-p}(i). Then (ii)-(iv) are clear from (i). 
	\end{proof}
 \begin{prop}\label{prop:alpha-extension}	
		 Let $1\le i\le m-1$.
				\begin{enumerate}[itemindent=-0.5em,label=\rm (\roman*)]
				\item $\widetilde{T}_{0}^{(i)}=T_{0}^{(i)}e_{\pmid}$.
				
				\item  For $j=1,\ldots,m-1$, $\widetilde{T}_{j}^{(i)}\ge T_{j}^{(i)}e_{\pmid}$, and the equality holds if $e_{\pmid}$ is odd and $d(-a_{j}a_{j+1})\not=2e_{\mathfrak{p}}$, or $f_{\pmid}$ is odd and $d(-a_{j}a_{j+1})=2e_{\mathfrak{p}}$.
				
				\item  $\widetilde{\alpha}_{i}\ge \alpha_{i}e_{\pmid}$, and the equality holds if $e_{\pmid}$ is odd, in particular, if $[E:F]$ is odd.
				
				\item $\widetilde{\alpha}_{i}=0$ if and only if $\alpha_{i}=0$.
				
				\item  $\widetilde{\alpha}_{i}=1$ if and only if $\alpha_{i}=1$ and $e_{\pmid}=1$.		
			\end{enumerate} 	
	 
	\end{prop}	 
		\begin{proof}	
		 (i) This follows from Proposition \ref{prop:R-extension}(i).
		 
		 (ii) This follows from Proposition \ref{prop:R-extension}(i) and Lemma \ref{lem:defectP-p}(ii). 
			
		(iii) First, the inequality holds from (i), (ii) and \eqref{T}. By \cite[16:4]{omeara_quadratic_1963}, we may assume that $e_{\pmid}$ is odd. It is sufficient to show $\widetilde{\alpha}_{i}\le \alpha_{i}e_{\pmid}$.
		
	    If $\alpha_{i}=T_{0}^{(i)}$, by (i), $\widetilde{\alpha}_{i}\le \widetilde{T}_{0}^{(i)}=T_{0}^{(i)}e_{\pmid}=\alpha_{i}e_{\pmid}$. If $\alpha_{i}\not=T_{0}^{(i)}$, then $\alpha_{i}=T_{j}^{(i)}<T_{0}^{(i)}$ for some $1\le j\le m-1$.  Hence, Proposition \ref{prop:T} implies that $d(-a_{j}a_{j+1})<2e_{\mathfrak{p}}$. So, by (ii),
	     \begin{align*}
	     	\widetilde{\alpha}_{i}\le \widetilde{T}_{j}^{(i)}=T_{j}^{(i)}e_{\pmid}=\alpha_{i}e_{\pmid}.
	     \end{align*}

	     (iv) This follows from Propositions \ref{prop:Rproperty}(ii), \ref{prop:R-extension}(i) and $e_{\mathfrak{P}}=e_{\mathfrak{p}}e_{\pmid}$.
	     
	     (v) This follows from (iii) immediately.
 \end{proof}

Let $N$ be another $\mathcal{O}_{F}$-lattice of rank $n$ and $\widetilde{N}=N\otimes_{\mathcal{O}_{F}} \mathcal{O}_{E}$. Then we write $\widetilde{d}[ca_{1,i}b_{1,j}]$ and $\widetilde{A}_{i}$ for the liftings of the quantities $d[ca_{1,i}b_{1,j}]$ and $A_{i}$, as defined in \eqref{defn:d[ab]} and \eqref{A-i-M-N}. Precisely, $\widetilde{d}[ca_{1,i}b_{1,j}]\coloneqq\min\{\widetilde{d}(ca_{1,i}b_{1,j}),\widetilde{\alpha}_{i},\widetilde{\beta}_{j}\}$ and $\widetilde{A}_{i}\coloneqq A_{i}(\widetilde{M},\widetilde{N})$. 

Based on Lemma \ref{lem:defectP-p}(ii), Propositions \ref{prop:R-extension}(i) and \ref{prop:alpha-extension}(iii), one can show the following three propositions by multiplying or dividing those equalities or inequalities by $e_{\pmid}$.
	 \begin{prop}
	Let $c\in F^{\times}$.
	\begin{enumerate}[itemindent=-0.5em,label=\rm (\roman*)]
		\item For $1\le i\le m$ and $1\le j\le n$, we have $\widetilde{d}[ca_{1,i}b_{1,j}]\ge d[ca_{1,i}b_{1,j}]e_{\pmid}$, and the equality holds if $[E:F]$ is odd.
		
		\item   For $1\le i\le n+1$, we have $\widetilde{A}_{j}\ge A_{j}e_{\pmid}$, and the equality holds if $[E:F]$ is odd.
	\end{enumerate} 
 
\end{prop}

\begin{prop}\label{prop:beliconditions-iso}
	Suppose that $[E:F]$ is odd and $m=n$.
	\begin{enumerate}[itemindent=-0.5em,label=\rm (\roman*)]
		\item For $ 1\le i\le n $, $ \widetilde{R}_{i}=\widetilde{S}_{i} $ if and only if $ R_{i}= S_{i} $.
		
		\item  For $ 1\le i\le n-1 $,  $ \widetilde{\alpha}_{i}=\widetilde{\beta}_{i} $  if and only if $  \alpha_{i}=\beta_{i} $.
		
		\item For $1\le i\le n-1$, $\widetilde{d}(a_{1,i}b_{1,i})\ge \widetilde{\alpha}_{i}$  if and only if $d(a_{1,i}b_{1,i})\ge \alpha_{i}$.
		
		\item For $1<i<n$, $\widetilde{\alpha}_{i-1}+\widetilde{\alpha}_{i}>2e_{\mathfrak{P}}$  if and only if $\alpha_{i-1}+\alpha_{i}>2e_{\mathfrak{p}}$.	
	\end{enumerate} 
\end{prop}
\begin{prop}\label{prop:beliconditions-rep}
 Suppose that $[E:F]$ is odd and $m\ge n$.
		\begin{enumerate}[itemindent=-0.5em,label=\rm (\roman*)]
		\item For $ 1\le i\le n $, $ \widetilde{R}_{i}\le \widetilde{S}_{i} $  if and only if $ R_{i}\le S_{i} $; for $ 1<i<m $, $ \widetilde{R}_{i}+\widetilde{R}_{i+1}\le \widetilde{S}_{i-1}+\widetilde{S}_{i} $  if and only if $ R_{i}+R_{i+1}\le S_{i-1}+S_{i} $.
		
		\item  For $ 1\le i\le \min\{m-1,n\} $,  $ \widetilde{d}[a_{1,i}b_{1,i}]\ge \widetilde{A}_{i} $  if and only if $ d[a_{1,i}b_{1,i}]\ge A_{i} $.

		\item For $ 1<i\le \min\{m-1,n+1\} $,  $\widetilde{R}_{i+1}>\widetilde{S}_{i-1}$  if and only if $R_{i+1}>S_{i-1}$; $\widetilde{A}_{i-1}+\widetilde{A}_{i}>2e_{\mathfrak{P}}+\widetilde{R}_{i}-\widetilde{S}_{i}$  if and only if $A_{i-1}+A_{i}>2e_{\mathfrak{p}}+R_{i}-S_{i}$.
		
		\item For $ 1<i\le \min\{m-2,n+1\} $, $ \widetilde{S}_{i}\ge \widetilde{R}_{i+2}>\widetilde{S}_{i-1}+2e_{\mathfrak{P}}\ge \widetilde{R}_{i+1}+2e_{\mathfrak{P}}$  if and only if $ S_{i}\ge R_{i+2}>S_{i-1}+2e_{\mathfrak{p}}\ge R_{i+1}+2e_{\mathfrak{p}}$.
		\end{enumerate} 
\end{prop}
With Propositions \ref{prop:beliconditions-iso} and \ref{prop:beliconditions-rep}, the following result is a direct application of Springer Theorem and Beli’s theorems  (\cite[Theorem 3.1]{beli_Anew_2010} and Theorem \ref{thm:beligeneral}).

\begin{thm}\label{thm:springer-thm-dyadic}
	Suppose that $[E:F]$ is odd. 
	\begin{enumerate}[itemindent=-0.5em,label=\rm (\roman*)]
		\item   If $\widetilde{M}\cong \widetilde{N}$, then $M\cong N$.
		
		\item If $\widetilde{M}$ represents $\widetilde{N}$, then $M$ represents $N$.
	\end{enumerate} 
 
\end{thm}
\begin{proof}

	 (i) This is clear from \cite[Theorem 3.1]{beli_Anew_2010}, Proposition \ref{prop:beliconditions-iso} and Springer Theorem.
	 
	 (ii) This is clear from Theorem \ref{thm:beligeneral}, Proposition \ref{prop:beliconditions-rep} and Springer Theorem.
\end{proof}
	\begin{re}\label{re:springer-thm-non-dyadic}
		Theorem \ref{thm:springer-thm-dyadic} is also true for non-dyadic local fields from \cite[Theorem 1]{omeara_integral_1958} and Springer Theorem (cf. \cite[Theorem 5.2]{xu_springer_1999}).
	\end{re}

		\section{Norm principles for spinor norms} \label{sec:norm-principle}
		
Throughout this section, we assume that $F$ is a dyadic local field and $E/F$ is a finite extension of local fields at the primes $\mathfrak{P}\mid \mathfrak{p}$, and let $a=\pi^{R}\varepsilon $, with $R\in \mathbb{Z}$ and $\varepsilon\in \mathcal{O}_{F}^{\times}$. 	Put $S= e_{\mathfrak{p}} -R/2 $. Let $\Pi$ be a uniformizer of $E$ and $\pi=\lambda\Pi^{e_{\pmid}}$ for some $\lambda\in \mathcal{O}_{E}^{\times}$. Then, in $E^{\times}$, we have $a=\pi^{R}\varepsilon=\Pi^{Re_{\pmid}}\lambda^{R}\varepsilon$ and denote by $\widetilde{a}\coloneqq\Pi^{Re_{\pmid}}\lambda^{R}\varepsilon$.  Then, by Proposition \ref{prop:R-extension}(i), $\widetilde{R}=\ord_{\mathfrak{P}}(a)=Re_{\pmid}$ and $\widetilde{S}=Se_{\pmid}=e_{\mathfrak{P}}-\widetilde{R}/2$.

	By convention, we write $\widetilde{\phi}\colon\mathscr{S}_{E}\to \mathscr{A}_{E}$, $\widetilde{N}(\tilde{a})\coloneqq N(E(\sqrt{\tilde{a}})/E) $,  $\widetilde{G}\colon E^{\times}/\mathcal{O}_{E}^{\times 2}\to \sgp(E^{\times}/E^{\times 2})$ and $\widetilde{g}\colon\mathscr{A}_{E}\to \sgp(\mathcal{O}_{E}^{\times}/\mathcal{O}_{E}^{\times 2})$ for the lifting over $E$ of the map  $\phi$ (cf. \eqref{phi-a}),  the subgroup $N(a)$ of $F^{\times}/F^{\times 2}$ (cf. \eqref{Na}) and the functions $G,g$ introduced in Section \ref{sec:BONGs}, respectively. We also  abuse $\theta$ for the spinor norm $\theta_{\mathfrak{p}}$ in ground fields and $\theta_{\mathfrak{P}}$ in extension fields. 
	
	In this section, we mainly prove the following two theorems.

\begin{thm}\label{thm:norm-principle-M}
	Let $M$ be an $\mathcal{O}_{F}$-lattice. Then $N_{E/F}(\theta(\widetilde{M}))\subseteq \theta(M)$.
\end{thm}
\begin{thm}\label{thm:norm-principle-M-N}
	Let $M,N$ be $\mathcal{O}_{F}$-lattices and $N\subseteq M$. Then $N_{E/F}(\theta(\widetilde{M}/\widetilde{N}))\subseteq \theta(M/N)$.
\end{thm}

		\begin{re}\label{re:norm-principle-non-dyadic}
		In the non-dyadic case, the proof of Theorem \ref{thm:norm-principle-M} follows from Kneser's computation for the spinor norm group, see \cite[Proposition 3.1]{earnest-hisa_spinorgeneraI-1977}; the proof of Theorem \ref{thm:norm-principle-M-N} is referred to \cite[Theorem 3.3]{xu_springer_1999}.
	      \end{re}
 \begin{prop}\label{prop:a=b-inOE}
	 	 Let $a,b\in E^{\times}$. Then $a\in b\mathcal{O}_{E}^{\times 2}$ if and only if 
	 	 	  $a\in b E^{\times 2}$ and $ \ord_{\mathfrak{P}}(a)=\ord_{\mathfrak{P}}(b)$.
	 \end{prop}
		\begin{proof}
			Necessity is clear. For sufficiency, let $a=bc^{2}$ for some $c\in E^{\times}$. Since $\ord_{\mathfrak{P}}(b)=\ord_{\mathfrak{P}}(a)=\ord_{\mathfrak{P}}(b)+2\ord_{\mathfrak{P}}(c)$, $\ord_{\mathfrak{P}}(c)=0$ and thus $c\in \mathcal{O}_{E}^{\times}$, so $a=bc^{2}\in b\mathcal{O}_{E}^{\times 2}$.
		\end{proof}
		 \begin{lem}\label{lem:phia-extension}
		 	Let $a\in \mathscr{A}_{F}$. 
		 		\begin{enumerate}[itemindent=-0.5em,label=\rm (\roman*)]
		 		
		 		\item $\widetilde{\alpha}(a)\ge \alpha(a)e_{\pmid}$.
		 		
		 		\item If $a\in\mathscr{S}_{F}$, then $ \widetilde{a}\in \mathscr{S}_{E}, \widetilde{\phi(a)},\widetilde{\phi}(\tilde{a})\in \mathscr{A}_{E}$ and $\widetilde{\phi}(\tilde{a})=\widetilde{\phi(a)}=\widetilde{a}$ in $E^{\times}/E^{\times 2}$. 
		 		
		 		\item Suppose $a\in \mathscr{S}_{F}$. If $R>2e_{\mathfrak{p}}$ or $a\in \mathscr{S}_{F}^{0}$, then $\widetilde{\phi}(\tilde{a})=\widetilde{\phi(a)}$ in $E^{\times}/\mathcal{O}_{E}^{\times 2}$.
		 	\end{enumerate}
		 \end{lem}
		  \begin{proof}
		  	(i) By Remark \ref{re:alpha-binary} and Proposition \ref{prop:alpha-extension}(iii), we have $\widetilde{\alpha}(a)=\alpha_{1}(\widetilde{L_{a}})\ge \alpha_{1}(L_{a})e_{\pmid}=\alpha(a)e_{\pmid}$ for some binary $\mathcal{O}_{F}$-lattice $L_{a}$.

		  	(ii) Since $a\in \mathscr{S}_{F}$, by (i) and \eqref{alpha-A-equiv}, $\widetilde{\alpha}(a)\ge \alpha(a)e_{\pmid}\ge 0$ and thus $\widetilde{a}\in \mathscr{A}_{E}$. Since $d(-a)>S$, by Lemma \ref{lem:defectP-p}(ii), $\widetilde{d}(-a)\ge d(-a)e_{\pmid}>Se_{\pmid}=\widetilde{S}$. So $\widetilde{a}\in \mathscr{S}_{E}$. Thus both $\widetilde{\phi(a)}$ and  $\widetilde{\phi}(\tilde{a})$ are defined. Combining with Proposition \ref{prop:phia-a}(i), we have $\widetilde{\phi(a)}\in \mathscr{A}_{E} $ and $\widetilde{\phi}(\tilde{a})\in \mathscr{A}_{E}$. This shows (ii) except for the equalities. 
		  	
		  	By Proposition \ref{prop:phia-a}(i),   $\widetilde{\phi}(\tilde{a})\in \widetilde{a}E^{\times 2}$, and $\phi(a)\in aF^{\times 2}$, which implies   $\widetilde{\phi(a)}\in \widetilde{a}E^{\times 2}$. So the equalities are proved.
		  	
		  	(iii)  By Proposition \ref{prop:a=b-inOE} and (ii), it suffices to show that $\ord_{\mathfrak{P}}(\widetilde{\phi}(\tilde{a}))=\ord_{\mathfrak{P}}(\widetilde{\phi(a)})$. If $R>2e_{\mathfrak{p}}$, then
		  	\begin{align*}
		  		\ord_{\mathfrak{P}}(\widetilde{\phi}(\tilde{a}))&=Re_{\pmid}-2e_{\mathfrak{P}}=(R-2e_{\mathfrak{p}})e_{\pmid}= \ord_{\mathfrak{P}}(\widetilde{\phi(a)});
		  	\end{align*}
		  	if $R\le 2e_{\mathfrak{p}}$ and $a\in \mathscr{S}_{F}^{0}$, since $S$ is even,
		  	\begin{align*}
		  		 \ord_{\mathfrak{P}}(\widetilde{\phi}(\tilde{a}))=-2\lfloor Se_{\pmid}/2\rfloor=-2\lfloor S/2\rfloor e_{\pmid}= \ord_{\mathfrak{P}}(\widetilde{\phi(a)}),
		  	\end{align*}
		  	as desired. 
		  \end{proof}
 \begin{cor}\label{cor:alpha-N-g-tilde=}
		Let $a\in \mathscr{S}_{F}$.
		 
		\begin{enumerate}[itemindent=-0.5em,label=\rm (\roman*)]
		 	
		 	\item  We have $\widetilde{d}(-\widetilde{\phi}(\tilde{a}))=\widetilde{d}(-\widetilde{\phi(a)}))=\widetilde{d}(-\tilde{a})$ and $\widetilde{N}(-\widetilde{\phi}(\tilde{a}))=\widetilde{N}(-\widetilde{\phi(a)})=\widetilde{N}(-\tilde{a})$.
		 	
		 	\item If $R>2e_{\mathfrak{p}}$ or  $a\in \mathscr{S}_{F}^{0}$, then  $\widetilde{\alpha}(\widetilde{\phi}(\tilde{a}))=\widetilde{\alpha}(\widetilde{\phi(a)})$ and $\widetilde{g}(\widetilde{\phi}(\tilde{a}))=\widetilde{g}(\widetilde{\phi(a)})$.
		 	 \end{enumerate}
		  \end{cor}
\begin{re}\label{re:alpha-N-g-tilde=}
 			In our convention in General Settings, the above equalities can be rewritten as
				\begin{align*}
					&\widetilde{d}(-\widetilde{\phi}(a))=\widetilde{d}(-\phi(a)))=\widetilde{d}(-a),\widetilde{N}(-\widetilde{\phi}(a))=\widetilde{N}(-\phi(a))=\widetilde{N}(-a),\\
					&\widetilde{\alpha}(\widetilde{\phi}(a))=\widetilde{\alpha}( \phi(a)) \quad\text{and}\quad\widetilde{g}(\widetilde{\phi}(a))=\widetilde{g}(\phi(a)),
				\end{align*}
			by dropping out the tilde symbols above the elements $a$ and $\phi(a)$.
		 \end{re}
\begin{proof}
		 	(i) Since $\widetilde{d}$ is defined on $E^{\times}/E^{\times 2}$ and $\widetilde{N}(\cdot)$ is unchanged under scaling a square of $E^{\times}$, the first two equalities follows from Lemma \ref{lem:phia-extension}(ii). 
		 	
		 	(ii) Since $\widetilde{\alpha}$ is defined on $\mathscr{A}_{E}\subseteq E^{\times}/\mathcal{O}_{E}^{\times 2}$, the equality $\widetilde{\alpha}(\widetilde{\phi}(\tilde{a}))=\widetilde{\alpha}(\widetilde{\phi(a)})$ follows from Lemma \ref{lem:phia-extension}(iii). So, by Lemma \ref{lem:ga-formula}, $\widetilde{g}(\widetilde{\phi}(\tilde{a}))=\widetilde{g}(\widetilde{\phi(a)}) $.
		 \end{proof}
 If $a\in \mathscr{S}_{F}$, from Lemma \ref{lem:phia-extension}(ii), $\widetilde{\phi(a)}$ and $ \widetilde{\phi}(\tilde{a})$ are defined, so $\widetilde{\alpha}(\widetilde{\phi(a)})$ and $ \widetilde{\alpha}(\widetilde{\phi}(\tilde{a}))$ are also defined. However, they are not equal in general when  $a\in \mathscr{S}_{F}^{1}$, because $\widetilde{\phi}(\tilde{a})=\widetilde{\phi(a)}$ may fail in $E^{\times}/\mathcal{O}_{E}^{\times 2}$. Thus, we also need the following proposition to treat this case.

\begin{prop}\label{prop:phia-Sodd}
		 		Let $a\in \mathscr{S}_{F} $, $\alpha=\alpha(\phi(a))$ and $\widetilde{\alpha}=\widetilde{\alpha}(\widetilde{\phi}(\tilde{a}))$. Assume that $a\in \mathscr{S}_{F}^{1}$. Then $\alpha,\widetilde{\alpha}\in \mathbb{Z}$ and $\widetilde{\alpha}>(\alpha-1)e_{\pmid} $.
		 \end{prop}
		 \begin{proof}
First, note that $S$ is odd. From  Proposition \ref{prop:phia-a}(iii),  
		  \begin{align*}
		  	 \alpha&=\min\{e_{\mathfrak{p}}-S/2+1/2,d(-a)-S+1\}\in \mathbb{Z},\\
		  \intertext{and then}
		 		(\alpha-1)e_{\pmid}&=\min\{e_{\mathfrak{p}}-S/2-1/2,d(-a)-S \}e_{\pmid} \\
		 		&=\min\{e_{\mathfrak{P}}-\widetilde{S}/2-e_{\pmid}/2,d(-a)e_{\pmid}-\widetilde{S}\}.
		 \end{align*}  
		 	On the other hand, $\widetilde{d}(-\widetilde{\phi}(\tilde{a}))=\widetilde{d}(-a)$ and $\ord_{\mathfrak{P}}(\widetilde{\phi}(\tilde{a}))=-2\lfloor  \widetilde{S}/2\rfloor$. Hence $\widetilde{\alpha}=\min\{e_{\mathfrak{P}}-\lfloor  \widetilde{S}/2\rfloor,	\widetilde{d}(-a)-2\lfloor \widetilde{S}/2\rfloor\}\in \mathbb{Z}$.
		 	
		 	If $\widetilde{\alpha}=e_{\mathfrak{P}}-\lfloor \widetilde{S}/2\rfloor$, then $
		 		 \widetilde{\alpha}\ge e_{\mathfrak{P}}-\widetilde{S}/2>e_{\mathfrak{P}}-\widetilde{S}/2-e_{\pmid}/2\ge  (\alpha-1)e_{\pmid}$,
		 	as desired.
		 	
		 	We may assume that $\widetilde{\alpha}=\widetilde{d}(-a)-2\lfloor \widetilde{S}/2\rfloor< e_{\mathfrak{P}}-\lfloor \widetilde{S}/2\rfloor$. By Proposition \ref{prop:phia-a}(iii), $S\le 2e_{\mathfrak{p}}$ and thus $\widetilde{d}(-a)< e_{\mathfrak{P}}+\lfloor \widetilde{S}/2\rfloor \le e_{\mathfrak{P}}+\lfloor 2e_{\mathfrak{P}}/2\rfloor=2e_{\mathfrak{P}}$. Hence, by Lemma \ref{lem:defectP-p}(ii), $d(-a)<2e_{\mathfrak{p}}$ and it follows that
		 	\begin{align*}
		 			\widetilde{d}(-a)\ge d(-a)e_{\pmid},
		 	\end{align*}
		 	 and the equality holds if and only if $e_{\pmid}$ is odd. Recall that $S$ is odd. So 
		 	 \begin{align*}
		 	 	-2\lfloor \widetilde{S}/2\rfloor \ge -\widetilde{S},
		 	 \end{align*}
		 	  and the equaltiy holds if and only if $\widetilde{S}$ is even, i.e, $e_{\pmid}$ is even. Since $e_{\pmid}$ cannot be odd and even simultaneously, one of these two inequalities must be strict. Hence, when adding them, we get the strict inequality
		 	\begin{align*}
		 		 \widetilde{\alpha}=\widetilde{d}(-a)-2\lfloor \widetilde{S}/2 \rfloor>d(-a)e_{\pmid}-\widetilde{S}\ge (\alpha-1)e_{\pmid}.
		 	\end{align*}  
		 \end{proof}
	 
		 \begin{lem}\label{lem:norm-principle-hk}
			Let $h,k\in \mathbb{Z}$. If $k>(h-1)e_{\pmid} $, then $N_{E/F}((1+\mathfrak{P}^{k})E^{\times 2})\subseteq (1+\mathfrak{p}^{h})F^{\times 2}$. 
			
			In particular, if $k\ge he_{\pmid}$, then $N_{E/F}((1+\mathfrak{P}^{k})E^{\times 2})\subseteq (1+\mathfrak{p}^{h})F^{\times 2}$.
 		\end{lem}
		\begin{proof}
			If $h\le 0$, by Proposition \ref{prop:1+p}(iii), $(1+\mathfrak{p}^{h})F^{\times 2}=F^{\times}$, so the inclusion is trivial.
			
			If $h\ge 2e_{\mathfrak{p}}+1$, then 
			\begin{align*}
				k\ge (h-1)e_{\pmid}+1\ge 2e_{\mathfrak{p}}e_{\pmid}+1=2e_{\mathfrak{P}}+1.
			\end{align*}
			 Hence $ (1+\mathfrak{p}^{h})F^{\times 2}=F^{\times 2}$ and $(1+\mathfrak{P}^{k})E^{\times 2}=E^{\times 2}$. Since $N_{E/F}(E^{\times 2})\subseteq F^{\times 2}$, we are done.  
			 
			 	Suppose that $1\le h \le 2e_{\mathfrak{p}}$. By Proposition \ref{prop:hsharp}, we have $h+h^{\#}\ge 2e_{\mathfrak{p}}+1$ and  $((1+\mathfrak{p}^{h^{\#}})F^{\times 2})^{\perp}=(1+\mathfrak{p}^{h})F^{\times 2}$.
			 
			 Let $c\in (1+\mathfrak{P}^{k})E^{\times 2}$. For any $b\in (1+\mathfrak{p}^{h^{\#}})F^{\times 2}$, by Lemma \ref{lem:defectP-p}(ii), $\widetilde{d}(b)\ge d(b)e_{\pmid}\ge h^{\#}e_{\pmid}$. So
			 \begin{align*}
			 	\widetilde{d}(c)+\widetilde{d}(b)&\ge   k +h^{\#}e_{\pmid}\ge  (h-1)e_{\pmid}+1+h^{\#}e_{\pmid} \\
			 	&=(h+h^{\#})e_{\pmid}-e_{\pmid}+1\ge (2e_{\mathfrak{p}}+1)e_{\pmid}-e_{\pmid}+1=2e_{\mathfrak{P}}+1.
			 \end{align*}
			 Hence, by \cite[Corollary]{bender_lifting_1973}, we deduce that
			 \begin{align*}
			 	(N_{E/F}(c),b)_{\mathfrak{p}}=(c,b)_{\mathfrak{P}}=1.
			 \end{align*}
			 Hence $N_{E/F}(c)\in N(b)$.
			 So, by the arbitrariness of $b$ and \cite[Lemma 1.2(ii)]{beli_integral_2003}, $N_{E/F}(c)\in ((1+\mathfrak{p}^{h^{\#}})F^{\times 2})^{\perp}=(1+\mathfrak{p}^{h})F^{\times 2}$, as required. 
		\end{proof}

		 \begin{cor}\label{cor:norm-principle-hk-R}
				Let $h,k\in \mathbb{R}$. If $k\ge he_{\pmid} $, then $N_{E/F}((1+\mathfrak{P}^{k})E^{\times 2})\subseteq (1+\mathfrak{p}^{h})F^{\times 2}$. 				
		\end{cor}
		 \begin{proof}
			From Proposition \ref{prop:1+p}(iii), we may let $h>0$. First, by Proposition \ref{prop:1+p}(iv), we have
			\begin{align*}
				(1+\mathfrak{p}^{h})F^{\times 2}=(1+\mathfrak{p}^{\lceil h\rceil})F^{\times 2}\quad\text{and}\quad (1+\mathfrak{P}^{k})E^{\times 2}=(1+\mathfrak{P}^{\lceil k\rceil})E^{\times 2}.
			\end{align*}
			Since $h>\lceil h\rceil-1$, we have $\lceil k\rceil \ge k\ge he_{\pmid}>(\lceil h\rceil-1)e_{\pmid}$. So the assertion follows by Lemma \ref{lem:norm-principle-hk}.
		\end{proof}
  \begin{lem}\label{lem:norm-principle-pieces}
			Let $a\in F^{\times}$.			
			\begin{enumerate}[itemindent=-0.5em,label=\rm (\roman*)]
			 \item  $N_{E/F}( \langle \tilde{a}\rangle E^{\times 2})\subseteq \langle a\rangle F^{\times 2}$.
			 
			  \item $N_{E/F}(\widetilde{N}(a))\subseteq N(a)$.
			 
			 \item If $a\in \mathscr{A}_{F}$, then $N_{E/F}((1+\mathfrak{P}^{\widetilde{\alpha}(a)})E^{\times 2})\subseteq (1+\mathfrak{p}^{\alpha(a)})F^{\times 2}$.
			 
			 \item  If $a\in\mathscr{A}_{F}$, then $N_{E/F}(\widetilde{g}(a)E^{\times 2})\subseteq g(a)F^{\times 2}$.
			\end{enumerate} 		
		\end{lem}
		\begin{proof}
			(i)  We have $N_{E/F}(\tilde{a})=a^{[E:F]}\in \langle a\rangle$ and $N_{E/F}(E^{\times 2})\subseteq F^{\times 2}$, as desired.
			
			(ii) Let $c\in \widetilde{N}(a)$. By \cite[Corollary]{bender_lifting_1973}, $(N_{E/F}(c),a)_{\mathfrak{p}}=(c,a)_{\mathfrak{P}}=1$, so $N_{E/F}(c)\in N(a)$.
			
			(iii) By Lemma \ref{lem:phia-extension}(i), we have $\widetilde{\alpha}(a)\ge\alpha(a)e_{\pmid}$, where $\alpha(a),\widetilde{\alpha}(a)\in \mathbb{Q}$, from Proposition \ref{prop:alphaproperty}(i). Hence, by Corollary \ref{cor:norm-principle-hk-R}, we see that
			\begin{align*}
				N_{E/F}((1+\mathfrak{P}^{\widetilde{\alpha}(a)})E^{\times 2})\subseteq   (1+\mathfrak{p}^{\alpha(a)})F^{\times 2}.
			\end{align*} 
	 
			(iv) For sets $S_{1}, S_{2}, S_{3}, S_{4}$, if $S_{1}\subseteq S_{3}$ and $S_{2}\subseteq S_{4}$, then $S_{1}\cap S_{2}\subseteq S_{3}\cap S_{4}$. So the assertion follows from  \cite[63:16]{omeara_quadratic_1963}, (ii), (iii) and Lemma \ref{lem:ga-formula-F2}.
\end{proof} 
\begin{lem}\label{lem:phia-norm-Sodd}
	Let $a\in \mathscr{S}_{F} $, $\alpha=\alpha(\phi(a))$ and $\widetilde{\alpha}=\widetilde{\alpha}(\widetilde{\phi}(\tilde{a}))$. Assume that $a\in \mathscr{S}_{F}^{1}$. Then $N_{E/F}((1+\mathfrak{P}^{\widetilde{\alpha}})E^{\times 2})\subseteq (1+\mathfrak{p}^{\alpha})F^{\times 2}$.
\end{lem}
\begin{proof}
	By Proposition \ref{prop:phia-Sodd}(iii), we have $\widetilde{\alpha}> (\alpha-1)e_{\pmid}$. Take $h=\alpha$ and $k=\widetilde{\alpha}$ in Lemma \ref{lem:norm-principle-hk}, as required.
\end{proof}
\begin{lem}\label{lem:norm-principle-ga}
	Let $a\in \mathscr{S}_{F} $. Then $N_{E/F}(\widetilde{g}(\widetilde{\phi}(a))E^{\times 2})\subseteq g(\phi(a))F^{\times 2}$.
\end{lem}
\begin{proof}
	If $R>2e_{\mathfrak{p}} $ or $a\in \mathscr{S}_{F}^{0}$, then, by Corollary \ref{cor:alpha-N-g-tilde=}(ii) and Remark \ref{re:alpha-N-g-tilde=}, $\widetilde{g}(\widetilde{\phi}(a))=\widetilde{g}(\phi(a))$. By Proposition \ref{prop:phia-a}(i), $\phi(a)\in \mathscr{A}_{F}$, so we are done by Lemma \ref{lem:norm-principle-pieces}(iv).
	
	If  $a\in \mathscr{S}_{F}^{1}$, then $\widetilde{g}(\widetilde{\phi}(a))=\widetilde{g}(\widetilde{\phi}(\tilde{a}))$ (from our convention), and by Corollary \ref{cor:alpha-N-g-tilde=}(i), $N(-\phi(a))=N(-a)$ and $\widetilde{N}(-\widetilde{\phi}(\tilde{a}))=\widetilde{N}(-\phi(a))=\widetilde{N}(-a)$. Hence, by Lemma \ref{lem:ga-formula-F2},
	\begin{align*}
				g(\phi(a))F^{\times 2}&=\mathcal{O}_{F}^{\times}F^{\times 2}\cap(1+\mathfrak{p}^{\alpha})F^{\times 2}\cap N(-a), \\
		\widetilde{g}(\widetilde{\phi}(a))E^{\times 2} &=\mathcal{O}_{E}^{\times}E^{\times 2}\cap(1+\mathfrak{P}^{\widetilde{\alpha}})E^{\times 2}\cap \widetilde{N}(-a),
	\end{align*}
	where $\alpha=\alpha(\phi(a))$ and $\widetilde{\alpha}=\widetilde{\alpha}(\widetilde{\phi}(\tilde{a}))$. So we are done by \cite[63:16]{omeara_quadratic_1963}, Lemmas \ref{lem:phia-norm-Sodd} and   \ref{lem:norm-principle-pieces}(ii).
\end{proof}

		\begin{lem}\label{lem:norm-principle-Ga}
			Let $a\in F^{\times}$. Then $N_{E/F}(\widetilde{G}(a))\subseteq G(a)$. 
		\end{lem}
		\begin{proof}
			 If $a\in \mathscr{S}_{F}$, then, by Lemma \ref{lem:phia-extension}(ii), $\widetilde{a}\in \mathscr{S}_{E}$. Hence, by Proposition \ref{prop:G-g}(i), $G(a)=\langle a\rangle g(\phi(a))F^{\times 2}$ and $\widetilde{G}=\langle \tilde{a}\rangle \widetilde{g}(\widetilde{\phi}(a))E^{\times 2}$. So, by Lemmas \ref{lem:norm-principle-pieces}(i) and \ref{lem:norm-principle-ga},  
				\begin{equation*} 
					\begin{aligned}
							N_{E/F}(\widetilde{G}(a))=N_{E/F}(\langle \tilde{a}\rangle\widetilde{g}(\widetilde{\phi}(a))E^{\times 2})&=N_{E/F}( \langle a\rangle E^{\times 2})N_{E/F}(\widetilde{g}(\phi(a))E^{\times 2})\\
						&\subseteq \langle a\rangle g(\phi(a))F^{\times 2}=G(a).
					\end{aligned}			
			\end{equation*}
			
			If $a\not\in \mathscr{S}_{F}$,  by Proposition \ref{prop:G-g}(ii) and (iii), $G(a)=N(-a)$ and  $\widetilde{G}(a)\subseteq\widetilde{N}(-a) $. Hence, by Lemma \ref{lem:norm-principle-pieces}(ii),
			\begin{align*}
				N_{E/F}(\widetilde{G}(a))\subseteq N_{E/F}(\widetilde{N}(-a))\subseteq N(-a)=G(a).
			\end{align*}
		\end{proof}
\begin{lem}\label{lem:Ga-subseteq-O_F2}
			Suppose that either $a\in \mathscr{A}_{F}$ or $a\not\in \mathscr{A}_{F} $ and $\Delta_{F}\not\in E^{\times 2}$. If $G(a)\subseteq \mathcal{O}_{F}^{\times}F^{\times 2}$, then $\widetilde{G}(a)\subseteq \mathcal{O}_{E}^{\times}E^{\times 2}$.
		\end{lem}
		\begin{proof}
			Since $G(a)\subseteq \mathcal{O}_{F}^{\times}F^{\times 2}$, by Proposition \ref{prop:GaOF-equiv}, $R$ is even, so is $\widetilde{R}$.

			Assume $a\in \mathscr{A}_{F}$. Then, by Proposition \ref{prop:GaOF-equiv}(i), either 
			\begin{align*}
				d(-a)=2e_{\mathfrak{p}}=-R\quad\text{or} \quad a\in \mathscr{S}_{F}.
			\end{align*}
			If $d(-a)=2e_{\mathfrak{p}}=-R$, by Lemma \ref{lem:defectP-p}(ii), $ \widetilde{d}(-a)=2e_{\mathfrak{P}}$ or $\infty$. In the first case, $\widetilde{d}(-a)=2e_{\mathfrak{P}}=-\widetilde{R}$; in the second case, $\widetilde{d}(-a)=\infty>\widetilde{S}$ and thus $\widetilde{a}\in \mathscr{S}_{E}$. If $a\in \mathscr{S}_{F}$, by Lemma \ref{lem:phia-extension}(ii), we also have $\widetilde{a}\in \mathscr{S}_{E}$.
			Hence, in both cases, $\widetilde{R}$ is even and either
			\begin{align*}
				\widetilde{d}(-a)=2e_{\mathfrak{P}}=-\widetilde{R}\quad\text{or}\quad  \widetilde{a}\in \mathscr{S}_{E},
			\end{align*}
			as desired.
			
			Assume that $a\not\in \mathscr{A}_{F}$ and $\Delta_{F}\not\in E^{\times 2}$. Then, by Proposition \ref{prop:GaOF-equiv}(ii),  $d(-a)=2e_{\mathfrak{p}}<-R$. So, by Lemma \ref{lem:defectP-p}(ii), $\widetilde{d}(-a)=2e_{\mathfrak{P}}<-\widetilde{R} $. 
			
			Now, by Proposition \ref{prop:GaOF-equiv}, we deduce that $\widetilde{G}(a)\subseteq \mathcal{O}_{E}^{\times}E^{\times 2}$, whenever $a\in \mathscr{A}_{F}$ or $a\not\in \mathscr{A}_{F}$ and $\Delta_{F}\not\in E^{\times 2}$.
		\end{proof}

		In the rest of this section, we keep the setting as in Section \ref{sec:BONGs}. Let $M\cong \prec a_{1},\ldots,a_{m}\succ$  relative to a good BONG $x_{1},\ldots,x_{m}$ with $R_{i}=\ord(a_{i})$. Then, by Proposition \ref{prop:R-extension} and Lemma \ref{lem:extension-BONG}, $\widetilde{R}_{i}=\ord_{\mathfrak{P}}(\widetilde{a}_{i})=R_{i}e_{\pmid}$ and
	$ \widetilde{M} =M\otimes_{\mathcal{O}_{F}} \mathcal{O}_{E}\cong \prec \widetilde{a}_{1},\ldots,\widetilde{a}_{m}\succ $ relative to the good BONG  $\widetilde{x}_{1},\ldots,\widetilde{x}_{m}$.
		\begin{lem}\label{lem:extension-property-A}
			If $M$ has property A, then $\widetilde{M}$ has property A.
		\end{lem}
		\begin{proof}
			This is clear from Definition  \ref{defn:A} and Proposition \ref{prop:scale-norm-volume}(i).
		\end{proof}
		\begin{lem}\label{lem:norm-principle-A}
			If $M$ has property A, then $N_{E/F}(\theta(\widetilde{M}))\subseteq \theta(M)$.
		\end{lem}
		\begin{proof}
			By Theorem \ref{thm:integralspinornorm-A}, we have
			\begin{align*}
				\theta(M)=G(a_{2}/a_{1}) \cdots G(a_{m}/a_{m-1})(1+\mathfrak{p}^{\gamma})F^{\times 2},
			\end{align*}
			where $\gamma=\gamma(M)$. Also, by \eqref{gammaM} and Remark \ref{re:property-A-B}(ii), $\gamma=\lfloor k\rfloor$ for some $k\in \mathbb{R}$ with $k>0 $.
			
				By Lemma \ref{lem:extension-property-A}, $\widetilde{M}$ also satisfies property $A$. Hence, by Lemma \ref{lem:extension-BONG} and Theorem \ref{thm:integralspinornorm-A}, we also have 
			 \begin{align*}
			 	\theta(\widetilde{M})=\widetilde{G}(a_{2}/a_{1}) \cdots \widetilde{G}(a_{m}/a_{m-1})(1+\mathfrak{P}^{\widetilde{\gamma}})E^{\times 2},
			 \end{align*}
			 where $\widetilde{\gamma}=\gamma(\widetilde{M})=\lfloor ke_{\pmid}\rfloor$. Since $\widetilde{\gamma}=\lfloor ke_{\pmid}\rfloor \ge \lfloor k\rfloor e_{\pmid}=\gamma e_{\pmid}$, by Lemma \ref{lem:norm-principle-hk}, we deduce that
			\begin{align*}
				N_{E/F}((1+\mathfrak{P}^{\widetilde{\gamma}})E^{\times 2}) \subseteq (1+\mathfrak{p}^{\gamma})F^{\times 2}.
			\end{align*}
			 By Lemma \ref{lem:norm-principle-Ga}, we also see that $
				N_{E/F}(\widetilde{G}(a_{i+1}/a_{i}))\subseteq G(a_{i+1}/a_{i})$
		 for $1\le i\le m-1$. So we are done by the multiplicity of $N_{E/F}$.
		\end{proof}

\begin{lem}\label{lem:thetaM-subseteq-O_F2}
If $\theta(M)\subseteq \mathcal{O}_{F}^{\times}F^{\times 2}$, then $\theta(\widetilde{M})\subseteq \mathcal{O}_{E}^{\times}E^{\times 2}$.
\end{lem}
\begin{proof}
If $\theta(M)\subseteq\mathcal{O}_{F}^{\times}F^{\times 2}$, then, by \cite[Theorem 3]{beli_integral_2003}, we have
\begin{align*}
	G(a_{i+1}/a_{i})&\subseteq \mathcal{O}_{F}^{\times}F^{\times 2}\quad\text{for any $1\le i\le m-1$},\quad\text{and}\\
	(R_{i+1}-R_{i})/2&\equiv e_{\mathfrak{p}}\pmod{2}\quad\text{for any $1\le i\le m-2$ with $R_{i}=R_{i+2}$},
\end{align*}
where $a_{i+1}/a_{i}\in \mathscr{A}_{F}$ from \eqref{eq:BONGs}. By Lemma \ref{lem:Ga-subseteq-O_F2}, we further have
\begin{align*}
	\widetilde{G}(a_{i+1}/a_{i})&\subseteq \mathcal{O}_{E}^{\times}E^{\times 2}\quad\text{for any $1\le i\le m-1$},\quad\text{and}\\
	(\widetilde{R}_{i+1}-\widetilde{R}_{i})/2&\equiv e_{\mathfrak{P}}\pmod{2}\quad\text{for any $1\le i\le m-2$ with $\widetilde{R}_{i}=\widetilde{R}_{i+2}$}.	 	
\end{align*}
So, again by \cite[Theorem 3]{beli_integral_2003}, we have  $\theta(\widetilde{M})\subseteq \mathcal{O}_{E}^{\times}E^{\times 2}$.
\end{proof}
\begin{cor}\label{cor:norm-principle-not-A}
If $M$ does not have property A, then $N_{E/F}(\theta(\widetilde{M}))\subseteq \theta(M)$.
\end{cor}
\begin{proof}
From Remark \ref{re:property-A-B}(i), $\theta(M)=F^{\times}$ or $\mathcal{O}_{F}^{\times}F^{\times 2}$. The former case is clear. For the latter case, by Lemma \ref{lem:thetaM-subseteq-O_F2}, we see that 
\begin{align*}
	 	N_{E/F}(\theta(\widetilde{M}))\subseteq N_{E/F}(\mathcal{O}_{E}^{\times}E^{\times 2})\subseteq \mathcal{O}_{F}^{\times}F^{\times 2}=\theta(M).
	 \end{align*}	 
		\end{proof}
		\begin{proof}[Proof of Theorem \ref{thm:norm-principle-M}]
			This immediately follows by Lemma \ref{lem:norm-principle-A} or Corollary \ref{cor:norm-principle-not-A}, according as $M$ satisfies property A or not.
		\end{proof}

		As in Section \ref{sec:BONGs}, let $M\cong \prec a_{1},\ldots,a_{m}\succ$ and $N\cong \prec b_{1},\ldots, b_{n}\succ$ relative to some good BONGs $x_{1},\ldots,x_{m}$ and $y_{1},\ldots,y_{n}$ with $a_{i}=\pi^{R_{i}}\varepsilon_{i}$ and $b_{i}=\pi^{S_{i}}\eta_{i}$, where $R_{i},S_{i}\in \mathbb{Z}$ and $\varepsilon_{i},\eta_{i}\in \mathcal{O}_{F}^{\times}$. Then  $\widetilde{a}_{i}=\Pi^{\widetilde{R}_{i}} \varepsilon_{i}^{\prime}$ and $\widetilde{b}_{i}=\Pi^{\widetilde{S}_{i}}\eta_{i}^{\prime}$, where $\widetilde{R}_{i}=R_{i}e_{\pmid}$, $\widetilde{S}_{i}=S_{i}e_{\pmid}$, $\varepsilon_{i}^{\prime}=\varepsilon_{i}\lambda^{R_{i}}$ and $\eta_{i}^{\prime}=\eta_{i}\lambda^{S_{i}}$. Thus, from Lemma \ref{lem:extension-BONG},
		\begin{align*}
			\widetilde{M}&=M\otimes_{\mathcal{O}_{F}} \mathcal{O}_{E}\cong \prec \widetilde{a}_{1},\ldots,\widetilde{a}_{m}\succ\quad\text{and}\quad\widetilde{N} =N\otimes_{\mathcal{O}_{F}} \mathcal{O}_{E}\cong \prec \widetilde{b}_{1},\ldots,\widetilde{b}_{n}\succ
		\end{align*}
		relative to the good BONGs $\widetilde{x}_{1},\ldots,\widetilde{x}_{m}$ and $\widetilde{y}_{1},\ldots,\widetilde{y}_{n}$.
		
		 For $1\le i\le \min\{m-1,n\}$, let $T_{i}=\max\{S_{i-1},R_{i+1}\}-\min\{S_{i},R_{i+2}\}$ and $c_{i}=\pi^{T_{i}}\xi_{i}$, with $ \xi_{i}= \varepsilon_{1,i+1}\eta_{1,i-1}\in \mathcal{O}_{F}^{\times }$.
		 
		 If $R_{i+2}>S_{i}$ for each $1\le i\le m-2$, then
		 \begin{align*}
		 	T_{i}=R_{i+1}-S_{i}\quad\text{and}\quad\pi^{R_{i+1}-S_{i}}\xi_{i}=c_{i};
		 \end{align*}
		 and if we put $Z_{i} =\sum_{k=1}^{i}(R_{k}+S_{k})$, then $Z_{i}\equiv\sum_{k=1}^{i}(R_{k}-S_{k})\pmod{2} $. Hence, \eqref{Gi} in Theorem \ref{thm:Beli-II2} can be rephrased as  
		\begin{align}\label{Gi-new}
			G_{i}=\begin{cases}
				G(c_{i}) &\text{if $i\le n$ and $Z_{i}\equiv 0\pmod{2}$}, \\
				N(-a_{1,i+1} b_{1,i-1})  &\text{otherwise}.
			\end{cases}
		\end{align}
	 Now, let $\widetilde{T}_{i}=T_{i}e_{\pmid}$ and $c_{i}^{\prime}=\Pi^{\widetilde{T}_{i}}\xi_{i}^{\prime}$, with $ \xi_{i}^{\prime}=\varepsilon_{1,i+1}^{\prime}\eta_{1,i-1}^{\prime}$. Similarly, put $\widetilde{Z}_{i} =Z_{i} e_{\pmid}$. Then $  \widetilde{Z}_{i}\equiv\sum_{k=1}^{i}(\widetilde{R}_{k}-\widetilde{S}_{k})\pmod{2} $. Hence, \eqref{Gi} for $\widetilde{M}$ and $\widetilde{N}$ is given by 
		 \begin{align}\label{Gi-new-extension}
			\widetilde{G}_{i}=\begin{cases}
				\widetilde{G}(c_{i}^{\prime}) &\text{if $i\le n$ and $\widetilde{Z}_{i}\equiv 0\pmod{2}$}, \\
				\widetilde{N}(-a_{1,i+1} b_{1,i-1})  &\text{otherwise}.
			\end{cases}
		\end{align}
 
     \begin{lem}\label{lem:c=ab}
     		Let $1\le i\le \min\{m-1,n\}$.
     	\begin{enumerate}[itemindent=-0.5em,label=\rm (\roman*)]
     	\item If $T_{i}=R_{i+1}-S_{i}$ and  $Z_{i}\equiv 0\pmod{2}$, then $c_{i}=a_{1,i+1}b_{1,i-1}$ in $F^{\times}/F^{\times 2}$.
     	
     	\item If $R_{1}\equiv\cdots\equiv R_{m}\equiv S_{1}\equiv \cdots \equiv S_{n}\pmod{2}$, then $c_{i}=a_{1,i+1}b_{1,i-1}=\xi_{i}$ in $F^{\times}/F^{\times 2}$.
     \end{enumerate}
     \end{lem}
    \begin{proof}
    	We have $c_{i}=\pi^{T_{i}}\xi_{i}$ and $a_{1,i+1}b_{1,i-1}=\pi^{Z_{i}+T_{i}}\xi_{i}$. If $Z_{i}\equiv 0\pmod{2}$, then, in $F^{\times}/F^{\times 2}$, $a_{1,i+1}b_{1,i-1}=\pi^{T_{i}}\xi_{i}=c_{i}$. If $R_{1}\equiv\cdots\equiv R_{m}\equiv S_{1}\equiv \cdots \equiv S_{n}\pmod{2}$, then $T_{i}$ and $Z_{i}$ are even, so, in $F^{\times}/F^{\times 2}$, $c_{i}=a_{1,i+1}b_{1,i-1}=\xi_{i}$.
    \end{proof}

		 \begin{lem}\label{lem:ci=Ci}
		 				Let $1\le i\le \min\{m-1,n\}$.
		 \begin{enumerate}[itemindent=-0.5em,label=\rm (\roman*)]
		 	\item	Suppose that $T_{i}=R_{i+1}-S_{i}$.
		   	\begin{enumerate}[itemindent=-0.5em,label=\rm (\roman*)]
		 		\item[(a)]	If $Z_{i}\equiv 0\pmod{2}$, then $c_{i}^{\prime}\in c_{i}\mathcal{O}_{E}^{\times 2}$.
		 		
		 		\item[(b)]  If $e_{\pmid}\equiv 0\pmod{2}$, then $c_{i}^{\prime}\in a_{1,i+1}b_{1,i-1}E^{\times 2}$.
		 		
		 	\end{enumerate}
		 	
		 	\item    If $R_{1}\equiv\cdots\equiv R_{m}\equiv S_{1}\equiv \cdots \equiv S_{n}\pmod{2}$, then $\xi_{i}^{\prime}\in \xi_{i}E^{\times 2}$ and $c_{i}^{\prime}\in c_{i}\mathcal{O}_{E}^{\times 2}$.
		\end{enumerate}
		 	 \end{lem}
		 	 \begin{proof}
		 	 	For (i)(a) and (ii), if  $Z_{i}\equiv 0\pmod{2}$, then $\widetilde{Z}_{i}\equiv 0\pmod{2}$, applying Lemma \ref{lem:c=ab} (i) to $c_{i}$ and $c_{i}^{\prime}$, we have $c_{i}=a_{1,i+1}b_{1,i-1}$ in $F^{\times}/F^{\times 2}$ and thus in $E^{\times}/E^{\times 2}$, and $c_{i}^{\prime}=a_{1,i+1}b_{1,i-1}$ in  $E^{\times}/E^{\times 2}$. Thus, in $E^{\times}/E^{\times 2}$, $c_{i}^{\prime}= c_{i} $. If $R_{1}\equiv\cdots\equiv R_{m}\equiv S_{1}\equiv \cdots \equiv S_{n}\pmod{2}$, by Lemma \ref{lem:c=ab}(ii), $\xi_{i}=c_{i}=a_{1,i+1}b_{1,i-1}$ in $F^{\times}/F^{\times 2} (\subseteq E^{\times}/E^{\times 2})$ and	$\xi_{i}^{\prime}=c_{i}^{\prime}=a_{1,i+1}b_{1,i-1}$ in $E^{\times}/E^{\times 2}$. Hence $\xi_{i}^{\prime}=\xi_{i}=c_{i}^{\prime}=c_{i}$ in $E^{\times}/E^{\times 2}$.
		 	 	
		 	 	In both cases,  $\ord_{\mathfrak{P}}(c_{i}^{\prime})=\widetilde{T}_{i}=T_{i}e_{\pmid}=\ord_{\mathfrak{p}}(c_{i})e_{\pmid}=\ord_{\mathfrak{P}}(c_{i})$. So, by Proposition \ref{prop:a=b-inOE}, $c_{i}^{\prime}\in c_{i}\mathcal{O}_{E}^{\times 2}$.
		 	 	
		 	 	For (i)(b), if $e_{\pmid}\equiv 0\pmod{2}$, then $\widetilde{Z}_{i}\equiv 0\pmod{2}$ and $\widetilde{T}_{i}=\widetilde{R}_{i+1}-\widetilde{S}_{i}$. So, by Lemma \ref{lem:c=ab}(i), $c_{i}^{\prime}=a_{1,i+1}b_{1,i-1}$ in $E^{\times}/E^{\times 2}$.
\end{proof}
 \begin{lem}\label{lem:tilde-Gi-GCi}
			Let $1\le i\le \min\{m-1,n\}$.
			 \begin{enumerate}[itemindent=-0.5em,label=\rm (\roman*)]
			\item 	Suppose that $T_{i}=R_{i+1}-S_{i}$.
			\begin{enumerate}[itemindent=-0.5em,label=\rm (\roman*)]
				\item[(a)]	If $Z_{i}\equiv 0\pmod{2}$, then $\widetilde{G}_{i}=\widetilde{G}(c_{i})$.
 			
				\item[(b)]   If  $Z_{i}\not\equiv 0\pmod{2}$, then $\widetilde{G}_{i}\subseteq \widetilde{N}(-a_{1,i+1}b_{1,i-1})$.
			 \end{enumerate}
			 
				\item If $R_{1}\equiv\cdots\equiv R_{m}\equiv S_{1}\equiv \cdots \equiv S_{n}\pmod{2}$, then $\widetilde{d}(-\xi_{i}^{\prime})=\widetilde{d}(-\xi_{i})$ and  $\widetilde{G}(c_{i}^{\prime})=\widetilde{G}(c_{i})$.
			\end{enumerate}
		\end{lem}
		\begin{proof}
			 (i) 	If $Z_{i}\equiv 0\pmod{2}$, then  $ \widetilde{Z}_{i} \equiv 0\pmod{2}$. Hence, by \eqref{Gi-new-extension} and Lemma \ref{lem:ci=Ci}(i)(a), $\widetilde{G}_{i}=\widetilde{G}(c_{i}^{\prime})=\widetilde{G}(c_{i})$.
			
			(ii)  Suppose that $Z_{i}\not\equiv 0\pmod{2}$. If $e_{\pmid}$ is odd, then $\widetilde{Z}_{i} \not\equiv 0\pmod{2}$ and thus by \eqref{Gi-new-extension},  $\widetilde{G}_{i} =\widetilde{N}(-a_{1,i+1}b_{1,i-1})$. If $e_{\pmid}$ is even, then, by \eqref{Gi-new-extension}, Proposition \ref{prop:G-g}(iii) and Lemma \ref{lem:ci=Ci}(i)(b), $
			 \widetilde{G}_{i}=\widetilde{G}(c_{i}^{\prime})\subseteq\widetilde{N}(-c_{i}^{\prime})=\widetilde{N}(-a_{1,i+1}b_{1,i-1})$.
			
			 (iii) This is clear from Lemma \ref{lem:ci=Ci}(ii).
		\end{proof}
		
	\begin{lem}\label{lem:norm-principle-G}
	 Let $1\le i\le m-1$. Suppose that $T_{i}=R_{i+1}-S_{i}$. Then $N_{E/F}(\widetilde{G}_{i})\subseteq G_{i}$.
	\end{lem}
		\begin{proof}
			 If $Z_{i}\equiv 0\pmod{2}$, then  $  \widetilde{Z}_{i}\equiv 0\pmod{2}$. By \eqref{Gi-new}, $G_{i}=G(c_{i})$ and by Lemma \ref{lem:tilde-Gi-GCi}(i), $\widetilde{G}_{i} =\widetilde{G}(c_{i})$. So, by Lemma \ref{lem:norm-principle-Ga},
			\begin{align*}
				N_{E/F}(\widetilde{G}_{i})=N_{E/F}(\widetilde{G}(c_{i}))\subseteq G(c_{i})=G_{i}.
			\end{align*}
			
			Assume that $Z_{i}\not\equiv 0\pmod{2}$. By \eqref{Gi-new}, $G_{i}=N(-a_{1,i+1}b_{1,i-1})$ and by Lemma \ref{lem:tilde-Gi-GCi}(ii), $\widetilde{G}_{i}\subseteq \widetilde{N}(-a_{1,i+1}b_{1,i-1})$. Hence, by Proposition \ref{prop:G-g}(iii) and Lemma \ref{lem:norm-principle-pieces}(ii),
			  \begin{align*}
			 	N_{E/F}(\widetilde{G}_{i})\subseteq N_{E/F}(\widetilde{N}(-a_{1,i+1}b_{1,i-1})) \subseteq N(-a_{1,i+1}b_{1,i-1})=G_{i}.
			 \end{align*} 
		\end{proof}
\begin{lem}\label{lem:norm-principle-M-N-I}
If $m-n\le 2$ and $R_{i+2}>S_{i}$ for each $1\le i\le m-2$, then $N_{E/F}(\theta(\widetilde{M}/\widetilde{N}))\subseteq \theta(M/N)$.  
\end{lem}
\begin{proof}
By Proposition \ref{prop:R-extension}(i), $\widetilde{R}_{i+2}>\widetilde{S}_{i}$ for each $1\le i\le m-2$. Hence, applying Theorem \ref{thm:Beli-II2}(i) to $M,N$ (resp. $\widetilde{M},\widetilde{N}$), we see that
\begin{align*}
	\theta(M/N)&=\theta(M)G_{1}\cdots G_{m-1}(1+\mathfrak{p}^{\gamma})F^{\times 2}, \\
				\theta(\widetilde{M}/\widetilde{N})&=\theta(\widetilde{M})\widetilde{G}_{1}\cdots \widetilde{G}_{m-1}(1+\mathfrak{P}^{\widetilde{\gamma}})E^{\times 2}. 
\end{align*}
From \eqref{gammaMN} and the hypothesis, $\gamma=\lfloor k\rfloor$ for some $k\in \mathbb{R}$ with $k>0$ and $\widetilde{\gamma}=\gamma(\widetilde{M},\widetilde{N})=\lfloor ke_{\pmid} \rfloor $.  As in Lemma \ref{lem:norm-principle-A}, we have $\widetilde{\gamma}\ge  \gamma e_{\pmid}$. So, by Lemma \ref{lem:norm-principle-hk}, $N_{E/F}((1+\mathfrak{P}^{\widetilde{\gamma}})E^{\times 2})\subseteq (1+\mathfrak{p}^{\gamma})F^{\times 2}$. By Theorem \ref{thm:norm-principle-M} and Lemma \ref{lem:norm-principle-G}, $N_{E/F}(\theta(\widetilde{M}))\subseteq \theta(M)$ and $N_{E/F}(\widetilde{G}_{i})\subseteq G_{i}$ for $1\le i\le m-1$. So the assertion follows from the multiplicity of $N_{E/F}$.
\end{proof}
		 
\begin{lem}\label{lem:dc=2ep-exist}
			Let $c\in F^{\times}$. If   $d(c)=2e_{\mathfrak{p}}$, then $\widetilde{d}(c)=2e_{\mathfrak{P}}$ or $N_{E/F}(E^{\times})\subseteq \mathcal{O}_{F}^{\times}F^{\times 2}$.
\end{lem}
			\begin{proof}
			If $c\not\in E^{\times 2}$, by Lemma \ref{lem:defectP-p}(ii),  $\widetilde{d}(c)=d(c)e_{\pmid}=2e_{\mathfrak{p}}e_{\pmid}=2e_{\mathfrak{P}}$, as desired.
			
			Assume that  $c\in E^{\times 2}$. Recall that $d(c)=2e_{\mathfrak{p}}$ if and only if $c\in \Delta_{F} F^{\times 2}$. Then $K\coloneqq F(\sqrt{c})=F(\sqrt{\Delta_{F}})$ is an unramified quadratic extension over $F$ contained in $E$. Hence, by \cite[63:16]{omeara_quadratic_1963}
			\begin{align*}
			N_{E/F}(E^{\times})=N_{K/F}(N_{E/K}(E^{\times}))\subseteq N_{K/F}(K^{\times})=\mathcal{O}_{F}^{\times}F^{\times 2}.		
			\end{align*}
		\end{proof}
\begin{lem}\label{lem:thetaOF}
If $\theta(M/N)\subseteq \mathcal{O}_{F}^{\times}F^{\times 2}$, then either $\theta(\widetilde{M}/\widetilde{N})\subseteq \mathcal{O}_{E}^{\times}E^{\times 2}$ or $N_{E/F}(E^{\times})\subseteq \mathcal{O}_{F}^{\times}F^{\times 2}$.
\end{lem}
		\begin{proof}
			 	First, by Theorem \ref{thm:Beli-II3}, conditions (i)-(iv) hold for $M$ and $N$. We denote by (i')-(iv') the conditions (i)-(iv) applied to $\widetilde{M}$ and $\widetilde{N}$. Then, multiplying the congruences in (i) and inequalities and equalities  in (ii) by $e_{\pmid}$ and using Proposition \ref{prop:R-extension}(i), we see that conditions (i') and (ii') are fulfilled for $\widetilde{M}$ and $\widetilde{N}$. By Lemma \ref{lem:thetaM-subseteq-O_F2}, condition (iv') is also fulfilled.
			 	
			 	It suffices to show that either condition (iii') is fulfilled or $N_{E/F}(E^{\times})\subseteq \mathcal{O}_{F}^{\times}F^{\times 2}$. Recall from condition (i) that  $R_{1}\equiv\cdots\equiv R_{m}\equiv S_{1}\equiv \cdots \equiv S_{n}\pmod{2}$. Thus Lemma \ref{lem:tilde-Gi-GCi}(ii) can be applied.
			 	
			 	For $1\le i\le m-1$, from condition (iii) we have either $d(-\xi_{i})=2e_{\mathfrak{p}}$, or $i\le n$ and $G(c_{i})\subseteq \mathcal{O}_{F}^{\times}F^{\times 2}$. 
			 	
			 	In the first case, Lemmas \ref{lem:tilde-Gi-GCi}(ii) and \ref{lem:dc=2ep-exist} imply that $\widetilde{d}(-\xi_{i}^{\prime})=\widetilde{d}(-\xi_{i})=2e_{\mathfrak{P}}$ or  $N_{E/F}(E^{\times})\subseteq \mathcal{O}_{F}^{\times}F^{\times 2}$, as desired.
			 	
			 	In the second case, from Proposition \ref{prop:GaOF-equiv}, we have either $c_{i}\in \mathscr{A}_{F}$ or $d(-\xi_{i})=d(-c_{i})=2e_{\mathfrak{p}}$. In the latter case, the assertion holds by Lemma \ref{lem:dc=2ep-exist} as the first case; in the former case, Lemmas \ref{lem:tilde-Gi-GCi}(ii) and \ref{lem:Ga-subseteq-O_F2} imply that $\widetilde{G}(c_{i}^{\prime})=\widetilde{G}(c_{i})\subseteq \mathcal{O}_{E}^{\times }E^{\times 2}$, which means that condition (iii') is fulfilled.
		\end{proof}
		\begin{lem}\label{lem:norm-principle-M-N-II}
			 If $m-n\le 2$ and $R_{j+2}\le S_{j}$ for some $1\le j\le m-2$, then $N_{E/F}(\theta(\widetilde{M}/\widetilde{N}))\subseteq \theta(M/N)$. 
 	\end{lem}
		\begin{proof}
		 	   By Theorem \ref{thm:Beli-II2}(ii), $\theta(M/N)\supseteq \mathcal{O}_{F}^{\times}F^{\times 2}$. We may assume $\theta(M/N)=\mathcal{O}_{F}^{\times}F^{\times 2}$. Then, by Lemma \ref{lem:thetaOF}, 
		 	\begin{align*}
		 		N_{E/F}(\theta(\widetilde{M}/\widetilde{N}))&\subseteq  N_{E/F}(\mathcal{O}_{E}^{\times}E^{\times 2})\subseteq \mathcal{O}_{F}^{\times}F^{\times 2}\quad\text{or}\\
		 		N_{E/F}(\theta(\widetilde{M}/\widetilde{N}))&\subseteq  N_{E/F}(E^{\times})\subseteq \mathcal{O}_{F}^{\times}F^{\times 2}.
		 	\end{align*} 
		\end{proof}

		\begin{proof}[Proof of Theorem \ref{thm:norm-principle-M-N}]
		The theorem follows from Remark \ref{re:m-n3}, Lemmas \ref{lem:norm-principle-M-N-I} and \ref{lem:norm-principle-M-N-II}.
		\end{proof}

\section{Arithmetic Springer theorem}\label{sec:AST}

For Theorems \ref{thm:AST-genus}, \ref{thm:AST-spn} and \ref{thm:AST}, some certain cases have already been proved in \cite{earnest-hisa_spinorgeneraI-1977,earnest-hisa_spinorgeneraII-1978,xu_springer_1999}, but for completeness, we will provide a unified proof for these theorems, particularly for Theorem \ref{thm:AST-spn}, by slightly modifying the arguments presented therein.
		
Let $E/F$ be a finite extension of algebraic number fields.  Write $J_{F}$ (resp. $J_{E}$) for the id\`{e}le group of $F$ (resp. $E$). Let $L$ and $N$ be $\mathcal{O}_{F}$-lattices with $\rank L=\ell\ge \rank N$, and $V=FL$. For convenience, we formally put $H_{\mathbb{A}}=O_{\mathbb{A}}^{+}(L)$ or $X_{\mathbb{A}}(L/N)$, and $H_{\mathfrak{p}}=O^{+}(L_{\mathfrak{p}})$ or $X(L_{\mathfrak{p}}/N_{\mathfrak{p}})$ naturally for $\mathfrak{p}\in  \Omega_{F}\backslash \infty_{F}$ (recall \eqref{local-XMN}). 

Put $\widetilde{L}=L\otimes_{\mathcal{O}_{F}} \mathcal{O}_{E}$ and $\widetilde{N}=N\otimes_{\mathcal{O}_{F}} \mathcal{O}_{E}$. Then $\widetilde{H}_{\mathbb{A}}=O_{\mathbb{A}}^{+}(\widetilde{L})$ or $X_{\mathbb{A}}(\widetilde{L}/\widetilde{N})$, and $\widetilde{H}_{\mathfrak{p}}=O^{+}(\widetilde{L}_{\mathfrak{p}})$ or $X(\widetilde{L}_{\mathfrak{p}}/\widetilde{N}_{\mathfrak{p}})$.
	
\begin{thm}\label{thm:AST-genus}
	Suppose that $[E:F]$ is odd.
		\begin{enumerate}[itemindent=-0.5em,label=\rm (\roman*)]
			 \item  If $\gen(\widetilde{L})=\gen(\widetilde{N})$, then $\gen(L)=\gen(N) $.
			
			\item   If $\gen(\widetilde{L})$ represents $\widetilde{N}$, then $\gen(L)$ represents $N$.
		\end{enumerate} 
\end{thm}
\begin{proof}
Since $[E:F]$ is odd, by \cite[15:3]{omeara_quadratic_1963}, $[E_{\mathfrak{P}}:F_{\mathfrak{p}}]$ is odd for some pair primes $\mathfrak{P}|\mathfrak{p}$. Hence these two assertions follow from Theorem \ref{thm:springer-thm-dyadic} and Remark \ref{re:springer-thm-non-dyadic}.
\end{proof}
			\begin{thm}\label{thm:AST-spn}
			Suppose that $[E:F]$ is odd.
			\begin{enumerate}[itemindent=-0.5em,label=\rm (\roman*)]
				\item  If $\spn^{+}(\widetilde{L})= \spn^{+}(\widetilde{N}) $, then $\spn^{+}(L)=\spn^{+}(N) $.
				
				\item If $\spn^{+}(\widetilde{L})$ represents $\widetilde{N}$, then $\spn^{+}(L)$ represents $N$.				
			\end{enumerate} 			
		\end{thm}
		 \begin{re}\label{re:xu-l2}
			 Theorems \ref{thm:AST-genus}(ii) and \ref{thm:AST-spn}(ii) for $  \ell\le 2$ have been stated by Xu (see \cite[Remark 6.3]{xu_springer_1999}).
		\end{re}
		We need some lemmas to treat remaining cases for Theorem \ref{thm:AST-spn}.
		
		\begin{lem}\label{lem:psi-inj}
		 Suppose that $\ell \ge 2$ and one of the following conditions holds:
			\begin{enumerate}[itemindent=-0.5em,label=\rm (\roman*)]
				 \item   $H_{\mathbb{A}}=O_{\mathbb{A}}^{+}(L)$;
				\item   $H_{\mathbb{A}}=X_{\mathbb{A}}(L/N)$.
			\end{enumerate} 
			Then the following diagram   
				 \begin{align}\label{comm-lifting}
				\begin{CD}
				  O_{\mathbb{A}}(\widetilde{V})/\widetilde{H}_{\mathbb{A}}O^{+}(\widetilde{V})O_{\mathbb{A}}^{'}(\widetilde{V})    @>\theta_{\widetilde{H}}>>   J_{E}/\theta(O^{+}(\widetilde{V}))\theta(\widetilde{H}_{\mathbb{A}})     \\ 
					@A\varphi AA     @AA\psi A\\
					O_{\mathbb{A}}(V)/H_{\mathbb{A}}O^{+}(V)O_{\mathbb{A}}^{'}(V)    	    @>>\theta_{H}>       J_{F}/\theta(O^{+}(V))\theta(H_{\mathbb{A}}) \\ 
				\end{CD}
			\end{align}
			is commutative. Also, both maps $\theta_{H}$ and $\theta_{\widetilde{H}}$ are injective.
			
			Thus, if $\psi$ is injective, then $\varphi $ is injective.
		\end{lem} 
		\begin{proof}
			See \cite[\S1]{earnest-hisa_spinorgeneraI-1977} for $H_{\mathbb{A}}=O_{\mathbb{A}}^{+}(L)$ and \cite[\S 1]{xu_springer_1999} for $H_{\mathbb{A}}=X_{\mathbb{A}}(L/N)$.
		\end{proof}
			\begin{lem}\label{lem:spn-equiv}
			Let $\sigma_{\mathbb{A}}\in O_{\mathbb{A}}(V)$. 
			\begin{enumerate}[itemindent=-0.5em,label=\rm (\roman*)]
				\item If $N=\sigma_{\mathbb{A}}(L)$, then $ \spn^{+}(L)=\spn^{+}(N)$ if and only if $\sigma_{\mathbb{A}}$ is trivial in the quotient group
			 $	O_{\mathbb{A}}(V)/O_{\mathbb{A}}^{+}(L)O^{+}(V)O_{\mathbb{A}}^{'}(V)  $.
				
				\item  If $N\subseteq\sigma_{\mathbb{A}}(L)$, then $ \spn^{+}(L)$ represents $N$  if and only if $\sigma_{\mathbb{A}}$ is trivial in the quotient group $	O_{\mathbb{A}}(V)/X_{\mathbb{A}}(L/N)O^{+}(V)O_{\mathbb{A}}^{'}(V)  $.
			\end{enumerate} 
			Also, both of the equivalent conditions are true for the lifting lattices $\widetilde{L}$ and $\widetilde{N}$.
		\end{lem}
		\begin{proof}
		See \cite[102:7]{omeara_quadratic_1963} and \cite[\S 1]{xu_springer_1999}.
		\end{proof}

   	\begin{lem}\label{lem:norm-principle-Nc-jp}
  Let $\mathfrak{p}\in \Omega_{F}\backslash \infty_{F}$, $c\in E^{\times}$ and $j_{\mathfrak{p}}\in F_{\mathfrak{p}}^{\times}$. Suppose that $[E:F]$ is odd. If $cj_{\mathfrak{p}}\in \theta(\widetilde{H}_{\mathfrak{P}})$ for each $\mathfrak{P}|\mathfrak{p}$, then $N_{E/F}(c)j_{\mathfrak{p}} \in \theta (H_{\mathfrak{p}}) $.
   \end{lem}
   \begin{proof}
   By \cite[15:3]{omeara_quadratic_1963}, $[E:F]=\sum_{\pmid}n(\mathfrak{P}|\mathfrak{p})$. Since $[E:F]$ is odd, in $F_{\mathfrak{p}}^{\times 2}$, we see that
   	\begin{align*}
   	N_{E/F}(c)j_{\mathfrak{p}}=j_{\mathfrak{p}}^{[E:F]}\prod_{\pmid}N_{\pmid}(c)=\prod_{\pmid}j_{\mathfrak{p}}^{n(\pmid)}\prod_{\pmid}N_{\pmid}(c)   =\prod_{\pmid}N_{\pmid}(cj_{\mathfrak{p}}).
   	\end{align*} 
    Since $cj_{\mathfrak{p}}\in \theta(\widetilde{H}_{\mathfrak{P}})$ for each $\mathfrak{P}|\mathfrak{p}$, if $\mathfrak{p}$ is non-dyadic, then, by Remark \ref{re:norm-principle-non-dyadic}, $N_{\pmid}(cj_{\mathfrak{p}})\in \theta(H_{\mathfrak{p}})$; if $\mathfrak{p}$ is dyadic, then, by Theorems \ref{thm:norm-principle-M} and \ref{thm:norm-principle-M-N}, $N_{\pmid}(cj_{\mathfrak{p}})\in \theta(H_{\mathfrak{p}})$. So $N_{E/F}(c)j_{\mathfrak{p}}\in \theta(H_{\mathfrak{p}})$.
   \end{proof}

	\begin{proof}[Proof of Theorem \ref{thm:AST-spn}]
          For (i), let $\ell\ge 2$. If $\spn^{+}(\widetilde{L})=\spn^{+}(\widetilde{N})$, then $\gen(\widetilde{L})=\gen(\widetilde{N})$. By Theorem \ref{thm:AST-genus}(i), $\gen(L)=\gen(N)$, so $N=\sigma_{\mathbb{A}}(L)$ for some $\sigma_{\mathbb{A}}\in O_{\mathbb{A}}(V)$. Hence, by Lemma \ref{lem:spn-equiv}(i),  $\widetilde{N}=\widetilde{\sigma}_{\mathbb{A}}(\widetilde{L})$ and $\widetilde{\sigma}$ is trivial. From diagram \eqref{comm-lifting}, if $\varphi$ is injective, then so is $\sigma_{\mathbb{A}}$ trivial, by Lemma \ref{lem:spn-equiv}(i), which is equivalent to $\spn^{+}(L)=\spn^{+}(N)$, as desired. 
          
          A similar argument can be applied to (ii) by Theorem \ref{thm:AST-genus}(ii) and Lemma \ref{lem:spn-equiv}(ii) instead. So, it remains to show that $\varphi$ is injective for $H_{\mathbb{A}}=O_{\mathbb{A}}^{+}(L)$ and $X_{\mathbb{A}}(L/N)$. 
         
        To do so, by Lemma \ref{lem:psi-inj}, it suffices to show that $\psi$ is injective.
		 Let $\bar{j_{\mathbb{A}}}\in \Ker\, \psi$ with $j_{\mathbb{A}}\in J_{F}$. Then there exists $c \in \theta(O^{+}(\widetilde{V}))$ such that $cj_{\mathbb{A}}\in \theta(H_{\mathbb{A}})$. Hence for each $\mathfrak{p}\in \Omega_{F}\backslash \infty_{F}$, we have $cj_{\mathfrak{p}}\in \theta(\widetilde{H}_{\mathfrak{P}})$ for all $\mathfrak{P}|\mathfrak{p}$. So, by Lemma \ref{lem:norm-principle-Nc-jp}, $N_{E/F}(c)j_{\mathfrak{p}}\in \theta(H_{\mathfrak{p}})$. Hence $N_{E/F}(c)j_{\mathbb{A}}\in \theta(H_{\mathbb{A}})$. Since $c\in \theta(O^{+}(\widetilde{V}))$,
	\cite[Lemma 2.2]{earnest-hisa_spinorgeneraI-1977} implies $N_{E/F}(c)\in \theta(O^{+}(V))$. So $j_{\mathbb{A}}\in \theta(O^{+}(V))\theta(H_{\mathbb{A}})$, as desired.
	\end{proof}

\begin{proof}[Proof of Theorem \ref{thm:AST}]
	Let $\ell\ge 3$. Since $L$ is indefinite, by \cite[104:5]{omeara_quadratic_1963}, $\cls^{+}(L)=\spn^{+}(L)$. This is also true for $\widetilde{L}$.
	
	(i) If $\widetilde{L}\cong \widetilde{N}$, then $\spn^{+}(\widetilde{L})=\spn^{+}(\widetilde{N})$ and $\rank N=\ell\ge 3$. By Theorem \ref{thm:AST-spn}, $\spn^{+}(L)=\spn^{+}(N)$. Hence, by \cite[104:5]{omeara_quadratic_1963}, $L\cong N$.

	(ii) If $ \widetilde{L} $ represents $\widetilde{N}$, then  $ \spn^{+}(\widetilde{L}) $ represents $\widetilde{N}$. By Theorem \ref{thm:AST-spn}, $\spn^{+}(L)$ represents $N$, so does $L$.
\end{proof}

	\begin{ex}\label{ex:binary}
		Let $F=\mathbb{Q}(\sqrt{-47})$. Then its discriminant is $d(F)=-47$ and Minkowski bound is $M_{F}<5$. A standard argument shows that 
		\begin{align*}
			2\mathcal{O}_{F}=\mathfrak{p}_{2}^{+}\mathfrak{p}_{2}^{-}\quad\text{with}\quad \mathfrak{p}_{2}^{\pm}=(2,(1\pm\sqrt{-47})/2), \\
			3\mathcal{O}_{F}=\mathfrak{p}_{3}^{+}\mathfrak{p}_{3}^{-}\quad\text{with}\quad \mathfrak{p}_{3}^{\pm}=(3,(1\pm\sqrt{-47})/2),
		\end{align*}
		and the ideal class group $\cl(F)$ of $F$ is a cyclic group of order $5$ generated by the ideal class $[\mathfrak{p}_{2}^{+}]$.
		
		Take $\mathfrak{a}=\mathfrak{p}_{2}^{+}$ and $\mathfrak{b}=\mathfrak{a}^{2}$. Then both of them are non-principal. Also, $\mathfrak{a}^{5}=a\mathcal{O}_{F}$ and $\mathfrak{b}^{5}=b\mathcal{O}_{F}$, with $a=(9+\sqrt{-47})/2 $ and $b=(17+9\sqrt{-47})/2$. Put $E=F(\sqrt[5]{a})$. Then $[E:F]=5$ is odd. For $\mathfrak{c}\in \{\mathfrak{a},\mathfrak{b}\}$, construct the binary $\mathcal{O}_{F}$-lattice $L(\mathfrak{c})$ as follows,
		\begin{align*}
			L (\mathfrak{c}) =\mathfrak{c}x+\mathfrak{c}^{-1}y\quad\text{with}\quad Q(x)=Q(y)=0\quad\text{and}\quad B(x,y)=1/2.
		\end{align*}
		Then $\widetilde{L(\mathfrak{c})}=L(\mathfrak{c})\otimes_{\mathcal{O}_{F}}\mathcal{O}_{E}$.
		
		We first show the following two assertions.
		
		(i)  $L(\mathfrak{c})$ is $1$-universal for each $\mathfrak{q}\in \Omega_{F}$; $\widetilde{L(\mathfrak{c})}$ is $1$-universal for each $\mathfrak{q}\in \Omega_{E}$.
		
		(ii)  $L(\mathfrak{c})$ is not $1$-universal over $F$, but $1$-universal over $E$. 
		
		For (i), let $\mathfrak{q}\in \Omega_{F}$. By \cite[Theorem 2.3]{hhx_indefinite_2021}, the binary quadratic space $F_{\mathfrak{q}}L(\mathfrak{c})_{\mathfrak{q}} $ is clearly $1$-universal for complex prime $\mathfrak{q}$. By \cite[82:8]{omeara_quadratic_1963}, we have $\mathfrak{n}(L(\mathfrak{c}))=2\mathfrak{s}(L(\mathfrak{c}))=\mathcal{O}_{F}$ and so $\mathfrak{n}(L(\mathfrak{c})_{\mathfrak{q}})=2\mathfrak{s}(L(\mathfrak{c})_{\mathfrak{q}})=\mathcal{O}_{F_{\mathfrak{q}}}$. Hence $L(\mathfrak{c})_{\mathfrak{q}}\cong \mathbf{H}_{F_{\mathfrak{q}}}$ is $1$-universal  for any non-dyadic prime $\mathfrak{q} $, by \cite[Proposition 2.3]{xu_indefinite_2020} and for any dyadic prime $\mathfrak{q} $, by \cite[Corollary 2.9]{xu_indefinite_2020}. Similarly for the lifting lattice $\widetilde{L(\mathfrak{c})}$.
		
		For (ii),  by \cite[Corollary 3.6]{xu_indefinite_2020}, it is sufficient to show that $1\not\in Q(L(\mathfrak{c}))$ and $1\in Q(\widetilde{L(\mathfrak{c})})$. Assume that $1\in Q(L(\mathfrak{c}))$. Then there exist $c\in \mathfrak{c}$ and $d\in \mathfrak{c}^{-1}$ such that $1=Q(cx+dy)=cd$. 
		Hence $c\mathfrak{c}^{-1}\subseteq \mathcal{O}_{F}$ and $d\mathfrak{c}\subseteq \mathcal{O}_{F}$, so by \cite[22:7]{omeara_quadratic_1963}, $\mathfrak{c}=c\mathcal{O}_{F}$ and $\mathfrak{c}^{-1}=d\mathcal{O}_{F}$, which contradicts the definition of $\mathfrak{c}$. Thus $1\not\in Q(L(\mathfrak{c}))$. But since $\mathfrak{c}\mathcal{O}_{E}=\sqrt[5]{c}\mathcal{O}_{E}$ with $c\in \{a,b\}$, we have $1=Q(\sqrt[5]{c}x+\sqrt[5]{c}^{-1}y)\in Q(\widetilde{L(\mathfrak{c})})$, as desired.
		
		Now let $N=\langle 1\rangle$ be a unary $\mathcal{O}_{F}$-lattice. Then, by (ii), $N$ is not represented by $L(\mathfrak{c})$, but $\widetilde{N}$ is represented by $\widetilde{L(\mathfrak{c})}$. This gives a counterexample for Theorem \ref{thm:AST}(ii) when $\ell=2$.
		
		 In $\cl(F)$, $[\mathfrak{b}]\not=[\mathfrak{a}]$ and $[\mathfrak{b}]\not=[\mathfrak{a}^{-1}]$, but in $\cl(E)$, $[\mathfrak{a}]=[\mathfrak{b}]=[1]$, so we see that
		\begin{align*}
			L(\mathfrak{a})\not\cong L(\mathfrak{b})\quad\text{and}\quad \widetilde{L(\mathfrak{a})}\cong \widetilde{L(\mathfrak{b})},
		\end{align*} 
		by (i), (ii) and the one-to-one correspondence in \cite[Theorem 3.5(1) and Remark 3.7]{xu_indefinite_2020}. This gives a counterexample for Theorem \ref{thm:AST}(i) when $\ell=2$.
 		
       Note that assertion (ii) also provides a counterexample for Corollary \ref{cor:globally-n-universal-extension-E-odd} with $\rank\,L=2$.
	\end{ex}

 \section{Liftings of $n$-universality}\label{sec:lifting-universal}

Unless otherwise stated, we assume that $F$ is a non-archimedean local field and $E$ is a finite extension of $F$ at the primes $\mathfrak{P}|\mathfrak{p}$. Following the notations in Section \ref{sec:lifting}, we let $M$ be an integral $ \mathcal{O}_{F} $-lattice of rank $ m\ge n\ge 1 $. When $F$ is non-dyadic, we consider a Jordan splitting $ M=J_{0}(M)\perp J_{1}(M)\perp\cdots \perp J_{t}(M)$, same as in Section \ref{sec:lifting}. When $F$ is dyadic, we suppose that
           $  M\cong \prec a_{1},\ldots,a_{m}\succ $  relative to some good BONG, and write $ R_{i}=R_{i}(M) $ for $ 1\le i\le m $ and $ \alpha_{i}=\alpha_{i}(M) $ for $ 1\le i\le m-1 $. For $h,k\in \mathbb{Z}$, we also write $[h,k]^{E}$ (resp. $[h,k]^{O}$) for the set of all even (resp. odd) integers $i$ with $h\le i\le k$.

           First, let us recall the necessary and sufficient conditions for a local lattice $M$ to be $n$-universal, which were obtained in \cite[Proposition 2.3]{xu_indefinite_2020}, \cite[Propositions 3.3 and 3.4]{hhx_indefinite_2021}, \cite[Theorem 2.1]{beli_universal_2020} and \cite[Theorems 4.7 and 5.1]{HeHu2}.

           	\begin{thm}\label{thm:n-universal-nondyadic}
           	If $F$ is non-dyadic, then $ M $ is $n$-universal over $F$ if and only if one of the following conditions holds:
           	
           	\begin{enumerate}[itemindent=-0.5em,label=\rm (\roman*)]
           		\item  $n=1$ and one of conditions (a)-(c) holds:
           		\begin{enumerate}[itemindent=-0.5em,label=\rm (\roman*)]
           			\item[(a)] $J_{0}(M)\cong \mathbf{H}_{F}$;
           			
           			\item[(b)] $J_{0}(M)\cong \mathbf{A}_{F}$ and $ \rank  J_{1}(M) \ge 2 $;
           			
           			\item[(c)] $\rank J_{0}(M)\ge 3$.
           		\end{enumerate}
           		\item $n=2$ and one of conditions (a)-(c) holds:
           		
           		\begin{enumerate}[itemindent=-0.5em,label=\rm (\roman*)]
           			\item[(a)]  $\rank J_{0}(M)=3$ and $\rank J_{1}(M)\ge 2$;
           			
           			\item[(b)] $J_{0}(M)\cong \mathbf{H}_{F}\perp \mathbf{H}_{F}$;
           			
           			\item[(c)]  $J_{0}(M)\cong \mathbf{H}_{F}\perp\mathbf{A}_{F}$ and $ \rank  J_{1}(M) \ge 1 $;
           			
           			\item[(d)] $\rank J_{0}(M)\ge 5$.
           		\end{enumerate}
           		
           		\item $n\ge 3$ and one of conditions (a)-(c) holds:
           		
           		\begin{enumerate}[itemindent=-0.5em,label=\rm (\roman*)]
           			\item[(a)] $\rank J_{0}(M)=n+1$ and $\rank J_{1}(M)\ge 2$;
           			
           			\item[(b)]   $\rank J_{0}(M)=n+2$ and $\rank J_{1}(M)\ge 1$;
           			
           			\item[(c)]  $\rank J_{0}(M)\ge n+3$.
           		\end{enumerate}		 
           	\end{enumerate} 
           \end{thm}
           
             \begin{thm}\label{thm:1-universaldyadic}
           		If $F$ is dyadic, then $ M $ is $1$-universal over $F$ if and only if $m\ge 2$, $R_{1}=0$ and one of the following conditions holds:
           	\begin{enumerate}[itemindent=-0.5em,label=\rm (\roman*)]
           		\item $\alpha_{1}=0$, i.e., $R_{2}=-2e_{\mathfrak{p}}$ and  conditions (a) and (b) hold:
           		
           		\begin{enumerate}[itemindent=-0.5em,label=\rm (\roman*)]
           			\item[(a)] if $m=2$ or $R_{3}>1$, then $[a_{1},a_{2}]\cong \mathbb{H}_{F}$;
           			
           			\item[(b)] if $m\ge 3$, $R_{3}=1$ and either $m=3$ or $R_{4}>2e_{\mathfrak{p}}+1$, then $[a_{1},a_{2}]\cong \mathbb{H}_{F}$.
           		\end{enumerate} 
           		\item $m\ge 3$, $\alpha_{1}=1$ and conditions (a) and (b) hold:
           		\begin{enumerate}[itemindent=-0.5em,label=\rm (\roman*)]

           			\item[(a)]  if $R_{2}=1$ or $R_{3}>1$, then $m\ge 4$ and $\alpha_{3}\le 2(e_{\mathfrak{p}}-\lfloor(R_{3}-R_{2})/2\rfloor)-1$;
           			
           			\item[(b)]  if $R_{2}\le 0$, $R_{3}\le 1$ and either $m=3$ or $R_{4}-R_{3}>2e_{\mathfrak{p}}$, then  $[a_{1},a_{2},a_{3}]$ is isotropic.
           		\end{enumerate} 
           	\end{enumerate} 
           	
           \end{thm}
           \begin{thm}\label{thm:even-n-universaldyadic} 
           		If $F$ is dyadic and $n\ge 2$ is even, then $ M $ is $n$-universal over $F$ if and only if $ m\ge n+3 $ or $m=n+2=4 $, and the following conditions hold:
           	\begin{enumerate}[itemindent=-0.5em,label=\rm (\roman*)]
           		\item $ R_{i}=0 $ for  $ i\in [1,n+1]^{O} $ and $ R_{i}=-2e_{\mathfrak{p}} $ for $ i\in [1,n]^{E} $.
           		
           		\item   If $ m=n+2=4 $, then $ M\cong \mathbf{H}_{F}\perp \mathbf{H}_{F} $.
           		
           		\item If $ m\ge n+3 $, then conditions (a)-(c) hold:
           		\begin{enumerate}
           			\item[(a)] $ \alpha_{n+1}=0$ or $\alpha_{n+1}=1$.
           			
           			\item[(b)]  If $ R_{n+3}-R_{n+2}>2e_{\mathfrak{p}}  $, then $ R_{n+2}=-2e_{\mathfrak{p}} $; and if moreover either $n\ge 4$, or $n=2$ and $ d(a_{1,4})=2e_{\mathfrak{p}} $, then $ R_{n+3}=1 $.
           			
           			\item[(c)]   If $ R_{n+3}-R_{n+2}=2e_{\mathfrak{p}} $ and $ R_{n+2}=2-2e_{\mathfrak{p}} $, then $ d(-a_{n+1}a_{n+2})=2e_{\mathfrak{p}}-1 $.
           		\end{enumerate}
           		
           	\end{enumerate} 
           	
           \end{thm}

           \begin{thm}\label{thm:odd-n-universaldyadic} 
           	If $F$ is dyadic and $n\ge 3$ is odd, then  $ M $ is $n$-universal over $F$ if and only if $ m\ge n+3 $ and the following conditions hold:
           	\begin{enumerate}[itemindent=-0.5em,label=\rm (\roman*)]
           		\item  $ R_{i}=0 $ for $ i\in [1,n]^{O} $, $ R_{i}=-2e_{\mathfrak{p}} $ for $ i\in [1,n]^{E} $, and  $ \alpha_{n}=0 $ or $ \alpha_{n}=1 $.

           		\item    If $ \alpha_{n}=0 $, then $ R_{n+2}=0$ or $R_{n+2}=1$.
           		
           		\item[] If $ \alpha_{n}=1 $ and either $ R_{n+1}=1 $ or $ R_{n+2}>1 $, then 
           		\begin{align*}
           			\alpha_{n+2}\le 2(e_{\mathfrak{p}}-\lfloor(R_{n+2}-R_{n+1})/2\rfloor)-1.
           		\end{align*} 
           		
           		\item  $ R_{n+3}-R_{n+2}\le 2e_{\mathfrak{p}} $. 
           	\end{enumerate} 
           \end{thm}	
           
         We will show the following theorems by applying these results to $\widetilde{M}$ and analyzing invariants under extensions in Section \ref{sec:lifting}.

%
%
%
%
%
%
%
%
%
%
%

           	\begin{thm}\label{thm:n-universal-nondyadic-overE}
      If $F$ is non-dyadic, then $ M $ is $n$-universal over $E$ if and only if one of the following conditions holds:
           	\begin{enumerate}[itemindent=-0.5em,label=\rm (\roman*)]
           		\item  $n=1$ and one of conditions (a)-(c) holds:
           		\begin{enumerate}[itemindent=-0.5em,label=\rm (\roman*)]
           			\item[(a)] $J_{0}(M)\cong \mathbf{H}_{F}$, or $J_{0}(M)\cong \mathbf{A}_{F}$ and $f_{\pmid}$ is even;
           			
           			\item[(b)] $J_{0}(M)\cong \mathbf{A}_{F}$, $ \rank  J_{1}(M) \ge 2 $, $e_{\pmid}=1$ and $f_{\pmid}$ is odd;
           			
           			\item[(c)] $\rank J_{0}(M)\ge 3$.
           		\end{enumerate}
           		\item  $n=2$ and one of conditions (a)-(c) holds:
           		
           		\begin{enumerate}[itemindent=-0.5em,label=\rm (\roman*)]
           			\item[(a)] $\rank J_{0}(M)=3$, $\rank J_{1}(M)\ge 2$ and $e_{\pmid}=1$;
           			
           			\item[(b)] $J_{0}(M)\cong \mathbf{H}_{F}\perp \mathbf{H}_{F}$, or $J_{0}(M)\cong\mathbf{H}_{F}\perp \mathbf{A}_{F}$ and $f_{\pmid}$ is even;
           			
           			\item[(c)]  $J_{0}(M)\cong \mathbf{H}_{F}\perp\mathbf{A}_{F}$, $ \rank  J_{1}(M) \ge 1 $, $e_{\pmid}=1$ and $f_{\pmid}$ is odd;
           			
           			\item[(d)] $\rank J_{0}(M)\ge 5$.
           		\end{enumerate}
           		
           		\item  $n\ge 3$ and one of conditions (a)-(c) holds:
           		
           		\begin{enumerate}[itemindent=-0.5em,label=\rm (\roman*)]
           			\item[(a)] $\rank J_{0}(M)=n+1$, $\rank J_{1}(M)\ge 2$ and  $e_{\pmid}=1$;
           			
           			\item[(b)]  $\rank J_{0}(M)=n+2$, $\rank J_{1}(M)\ge 1$ and  $e_{\pmid}=1$; 
           			
           			\item[(c)]  $\rank J_{0}(M)\ge n+3$.
           		\end{enumerate}		 
           	\end{enumerate} 
           	
           \end{thm}

	 \begin{thm}\label{thm:1-universaldyadic-overE}
	 If $F$ is dyadic, then $ M $ is $1$-universal over $E$ if and only if $m\ge 2$, $R_{1}=0$ and one of the following conditions holds:
	 	\begin{enumerate}[itemindent=-0.5em,label=\rm (\roman*)]
	 		\item $\alpha_{1}=0$, i.e., $R_{2}=-2e_{\mathfrak{p}}$ and  conditions (a) and (b) hold:
	 		
	 		\begin{enumerate}[itemindent=-0.5em,label=\rm (\roman*)]
	 			\item[(a)] if $m=2$ or $R_{3}>1$ or $R_{3}=1$ and $e_{\pmid}>1$, then $[a_{1},a_{2}]\cong \mathbb{H}_{F}$, or $[a_{1},a_{2}]\cong \mathbb{A}_{F}$ and $f_{\pmid}$ is even;
	 			
	 			\item[(b)] if $m\ge 3$, $R_{3}=e_{\pmid}=1$ and either $m=3$ or $R_{4}>2e_{\mathfrak{p}}+1$, then $[a_{1},a_{2}]\cong \mathbb{H}_{F}$, or $[a_{1},a_{2}]\cong \mathbb{A}_{F}$ and $f_{\pmid}$ is even.
	 		\end{enumerate} 
	 		\item $m\ge 3$, $\alpha_{1}=e_{\pmid}=1$ and conditions (a) and (b) hold:
	 		\begin{enumerate}[itemindent=-0.5em,label=\rm (\roman*)]

	 			\item[(a)]  if $R_{2}=1$ or $R_{3}>1$, then $m\ge 4$ and $\alpha_{3}\le 2(e_{\mathfrak{p}}-\lfloor(R_{3}-R_{2})/2\rfloor)-1$;
	 			
	 			\item[(b)]  if $f_{\pmid}$ is odd, $R_{2}\le 0$, $R_{3}\le 1$ and either $m=3$ or $R_{4}-R_{3}>2e_{\mathfrak{p}}$, then  $[a_{1},a_{2},a_{3}]$ is isotropic over $F$.
	 		\end{enumerate} 
	 	\end{enumerate} 
	 	
	 \end{thm}

	 \begin{thm}\label{thm:even-n-universaldyadic-overE}
	 	If $F$ is dyadic and $n\ge 2$ is even, then $ M $ is $n$-universal over $E$ if and only if either $m\ge n+3$ or $m=n+2=4$, and the following conditions hold:
	 	\begin{enumerate}[itemindent=-0.5em,label=\rm (\roman*)]
	 		\item  $ R_{i}=0 $ for  $ i\in [1,n+1]^{O} $ and $ R_{i}=-2e_{\mathfrak{p}} $ for $ i\in [1,n]^{E} $.
	 		
	 		\item   If $ m=n+2=4 $, then $M\cong \mathbf{H}_{F}\perp \mathbf{H}_{F}$, or $M\cong \mathbf{H}_{F}\perp \mathbf{A}_{F}$ and $f_{\pmid}$ is even.
	 		
	 		\item If $ m\ge n+3 $, then the conditions (a)-(c) hold:
	 		
	 		\begin{enumerate}[itemindent=-0.5em,label=\rm (\roman*)]
	 			\item[(a)] $ \alpha_{n+1}=0$ or $\alpha_{n+1}=e_{\pmid}=1$.
	 			
	 			\item[(b)]  If $ R_{n+3}-R_{n+2}>2e_{\mathfrak{p}}  $, then $ R_{n+2}=-2e_{\mathfrak{p}} $; and if moreover either $n\ge 4$, or $n=2$, $d(a_{1,4})=2e_{\mathfrak{p}} $ and $f_{\pmid}$ is odd, then $ R_{n+3}=e_{\pmid}=1$.
	 			
	 			\item[(c)]  If $ R_{n+3}-R_{n+2}=2e_{\mathfrak{p}} $ and $ R_{n+2}=2-2e_{\mathfrak{p}}  $, then $ d(-a_{n+1}a_{n+2})=2e_{\mathfrak{p}}-1 $.
	 		\end{enumerate}
	 		
	 	\end{enumerate}
	 \end{thm}
	 
	 		 	\begin{thm}\label{thm:odd-n-universaldyadic-overE}
	 		If $F$ is dyadic and $n\ge 3$ is odd, then $ M $ is $n$-universal over $E$ if and only if $ m\ge n+3 $ and the following conditions hold:
	 	\begin{enumerate}[itemindent=-0.5em,label=\rm (\roman*)]
	 		\item   $ R_{i}=0 $ for $ i\in [1,n]^{O} $, $ R_{i}=-2e_{\mathfrak{p}} $ for $ i\in [1,n]^{E} $, and $ \alpha_{n}=0 $ or $ \alpha_{n}=e_{\pmid}=1$.

	 		\item    If $ \alpha_{n}=0 $, then $ R_{n+2}=0$ or $R_{n+2}=e_{\pmid}=1$.
	 		
	 		\item[] If $ \alpha_{n}=1 $ and either $ R_{n+1}=1 $ or $ R_{n+2}>1 $, then 
	 		\begin{align*}
	 			\alpha_{n+2}\le 2(e_{\mathfrak{p}}-\lfloor(R_{n+2}-R_{n+1})/2\rfloor)-1.
	 		\end{align*}

	 		\item   $ R_{n+3}-R_{n+2}\le 2e_{\mathfrak{p}} $.  
	 	\end{enumerate} 	 	
	 \end{thm}
 
 Comparing Theorems \ref{thm:n-universal-nondyadic},  \ref{thm:1-universaldyadic}, \ref{thm:even-n-universaldyadic} and \ref{thm:odd-n-universaldyadic} with Theorems \ref{thm:n-universal-nondyadic-overE}, \ref{thm:1-universaldyadic-overE}, \ref{thm:even-n-universaldyadic-overE} and \ref{thm:odd-n-universaldyadic-overE}, respectively, one can easily show the corollaries below.

	 \begin{cor}\label{cor:non-dyadic-n-universal-overE-implies-overF}
	 	If $F$ is non-dyadic and $M$ is $n$-universal over $E$, then $M$ is $n$-universal over $F$ except that $f_{\pmid}$ is even and one of the following cases happens:
	 	\begin{enumerate}[itemindent=-0.5em,label=\rm (\roman*)]
	 		\item $n=1$ and $M\cong \mathbf{A}_{F}$;
	 		
	 		\item $n=2$ and $M\cong \mathbf{H}_{F}\perp\mathbf{A}_{F}$.
	 		
	 	\end{enumerate}
	 	\item Thus, if $f_{\pmid}$ is odd or $n\ge 3$ or $m\ge n+3=5$ or $m\ge n+2=3$, then $M$ is $n$-universal over $F$. 
	 \end{cor}

	 \begin{cor}\label{cor:dyadic-n-universal-overE-implies-overF}
	 If $F$ is dyadic and $M$ is $n$-universal over $E$, then $M$ is $n$-universal over $F$ except that $f_{\pmid}$ is even and one of the following cases happens:
	 	\begin{enumerate}[itemindent=-0.5em,label=\rm (\roman*)]
	 		\item  $n=1$ and $[a_{1},a_{2}]\cong \mathbb{A}_{F}$;
	 		
	 		\item  $n=2$ and $M\cong \mathbf{H}_{F}\perp\mathbf{A}_{F}$.
	 		
	 	\end{enumerate}
	 	Thus, if $f_{\pmid}$ is odd or $n\ge 3$ or $m\ge n+3=5$, then $M$ is $n$-universal over $F$.
	 	
	 \end{cor}
	 
	 To prove the theorems above, we divide into non-dyadic and dyadic cases naturally.

		\medskip
		\noindent\textbf{Case I: $F$ is non-dyadic.}

		\begin{lem}\label{lem:n-universal-nondyadic-overE}
				An $\mathcal{O}_{F}$-lattice $ M $ is $n$-universal over $E$ if and only if one of the following conditions holds:
				\begin{enumerate}[itemindent=-0.5em,label=\rm (\roman*)]
				\item   $n=1$ and one of conditions (a)-(c) holds:
				\begin{enumerate}[itemindent=-0.5em,label=\rm (\roman*)]
					\item[(a)] $J_{0}(\widetilde{M})\cong \mathbf{H}_{E}$;
					
					\item[(b)] $J_{0}(\widetilde{M})\cong \mathbf{A}_{E}$, $ \rank  J_{1}(M) \ge 2 $ and $e_{\pmid}=1$;
					
					\item[(c)] $\rank J_{0}(M)\ge 3$.
				\end{enumerate}
				\item  $n=2$ and one of conditions (a)-(c) holds:
				
				\begin{enumerate}[itemindent=-0.5em,label=\rm (\roman*)]
					\item[(a)] $\rank J_{0}(M)=3$, $\rank J_{1}(M)\ge 2$ and $e_{\pmid}=1$;
					
					\item[(b)] $J_{0}(\widetilde{M})\cong \mathbf{H}_{E}\perp \mathbf{H}_{E}$;
					
					\item[(c)]  $J_{0}(\widetilde{M})\cong \mathbf{H}_{E}\perp\mathbf{A}_{E}$, $ \rank  J_{1}(M) \ge 1 $ and $e_{\pmid}=1$;
					
					\item[(d)] $\rank J_{0}(M)\ge 5$.
				\end{enumerate}
				
				\item  $n\ge 3$ and one of conditions (a)-(c) holds:
				
				\begin{enumerate}[itemindent=-0.5em,label=\rm (\roman*)]
					\item[(a)] $\rank J_{0}(M)=n+1$, $\rank J_{1}(M)\ge 2$ and  $e_{\pmid}=1$;
					
					\item[(b)]  $\rank J_{0}(M)=n+2$, $\rank J_{1}(M)\ge 1$ and  $e_{\pmid}=1$; 
					
					\item[(c)]  $\rank J_{0}(M)\ge n+3$.
				\end{enumerate}		 
			\end{enumerate} 
		 
		\end{lem}	 
			\begin{proof}
			By Lemma \ref{lem:extension-Jordansplittings}, $\widetilde{M}$ has a Jordan splitting 
			\begin{align}\label{Jordansplitting-over-E}
				\widetilde{M}&=J_{0}(\widetilde{M})\perp J_{e_{\pmid}}(\widetilde{M})\perp \cdots \perp J_{te_{\pmid}}(\widetilde{M})=\widetilde{J_{0}(M)}\perp \widetilde{J_{e_{\pmid}}(M)}\perp \cdots \perp \widetilde{J_{te_{\pmid}}(M)},
			\end{align}
			and we have
			\begin{align}\label{rank}
				\rank J_{ie_{\pmid}}(\widetilde{M})=\rank \widetilde{J_{i}(M)}=\rank J_{i}(M),
			\end{align}	
			for $0\le i\le t$.
			
			Suppose that $M$ satisfies one of conditions (i)-(iii). Then, from \eqref{Jordansplitting-over-E}, $\widetilde{M}$ satisfies one of Theorem \ref{thm:n-universal-nondyadic}(i)-(iii) accordingly. Hence, by Theorem \ref{thm:n-universal-nondyadic}, $M$ is $n$-universal over $E$. This shows the sufficiency.
			
			Suppose that $M$ is $n$-universal over $E$. Then $\widetilde{M}$ satisfies one of conditions (i)-(iii) in Theorem \ref{thm:n-universal-nondyadic}. If $\rank J_{1}(\widetilde{M})>0$, then, by \eqref{Jordansplitting-over-E}, $e_{\pmid}=1$. The remaining conditions are clear from \eqref{rank}. This shows the necessity.
		\end{proof}
		
		 Now, we will refine the conditions on $J_{0}(\widetilde{M})$ in Lemma \ref{lem:n-universal-nondyadic-overE}. The following proposition is clear from \cite[92:1]{omeara_quadratic_1963}.
		
		\begin{prop}\label{prop:binary-quaternary-unimodular}
			Let $N$ be a unimodular $\mathcal{O}_{F}$-lattice of rank $n$. Write $dN$ for the discriminant of $N$.
			
			\begin{enumerate}[itemindent=-0.5em,label=\rm (\roman*)]
				\item  If $n=2$, then $N\cong \mathbf{H}_{F}$ or $\mathbf{A}_{F}$, according as $-dN\in \mathcal{O}_{F}^{\times 2}$ or $   \Delta_{F}\mathcal{O}_{F}^{\times 2}$.
								
				\item If  $n=4$, then $N\cong \mathbf{H}_{F}\perp \mathbf{H}_{F}$ or $\mathbf{H}_{F}\perp\mathbf{A}_{F}$, according as $dN\in \mathcal{O}_{F}^{\times 2}$ or $  \Delta_{F}\mathcal{O}_{F}^{\times 2}$. 
			\end{enumerate}
		\end{prop}
		\begin{lem}\label{lem:binary-quaternary-lifting}
			Let $N$ be a unimodular $\mathcal{O}_{F}$-lattice of rank $n$.
			\begin{enumerate}[itemindent=-0.5em,label=\rm (\roman*)]
				\item  If $n=2$, then $\widetilde{N}\cong \mathbf{H}_{E}$ if and only if $N\cong \mathbf{H}_{F}$, or $N\cong \mathbf{A}_{F}$ and $f_{\pmid}$ is even.
				
				\item  If $n=2$, then $\widetilde{N}\cong   \mathbf{A}_{E}$ if and only if $N\cong \mathbf{A}_{F}$ and $f_{\pmid}$ is odd.
				
				\item  If $n=4$, then $\widetilde{N}\cong \mathbf{H}_{E}\perp \mathbf{H}_{E}$ if and only if $N\cong \mathbf{H}_{F}\perp \mathbf{H}_{F}$, or $N\cong \mathbf{H}_{F}\perp\mathbf{A}_{F}$ and $f_{\pmid}$ is even.
				
				\item  If $n=4$, then $\widetilde{N}\cong \mathbf{H}_{E}\perp \mathbf{A}_{E}$ if and only if $N\cong \mathbf{H}_{F}\perp \mathbf{A}_{F}$ and $f_{\pmid}$ is odd.
			\end{enumerate}
		\end{lem}
		
		\begin{proof}
			
			If $n=2$, since $N$ is unimodular, $-dN\in \mathcal{O}_{F}^{\times 2}\cup  \Delta_{F}\mathcal{O}_{F}^{\times 2}$. From Remark \ref{re:defectP-p}, $-d\widetilde{N}\in \mathcal{O}_{E}^{\times 2} $ if and only if $-dN\in \mathcal{O}_{F}^{\times 2}$, or $-dN \in \Delta_{F}\mathcal{O}_{F}^{\times 2}$ and $f_{\pmid}$ is even. So  (i) and (ii) follow from Proposition \ref{prop:binary-quaternary-unimodular}(i).
			
			If $n=4$, by Lemma \ref{lem:extension-Jordansplittings}, one can reduce to the binary case. So (iii) and (iv) follow from (i), (ii) and Proposition \ref{prop:binary-quaternary-unimodular}(ii).
		\end{proof}

		\begin{proof}[Proof of Theorem \ref{thm:n-universal-nondyadic-overE}]
	By Lemma \ref{lem:extension-Jordansplittings}, $J_{0}(\widetilde{M})=\widetilde{J_{0}(M)}$ and thus, the conditions in Lemma \ref{lem:n-universal-nondyadic-overE} can be refined by Lemma \ref{lem:binary-quaternary-lifting} with $N=J_{0}(M)$. This shows the theorem.
		\end{proof}

		\medskip
		
		\noindent\textbf{Case II: $F$ is dyadic.}

		From \cite[Proposition 3.7]{HeHu2}, we have the following proposition.
			
		\begin{prop}\label{prop:binary-quaternary-dyadic}
			Let $N$  be an $\mathcal{O}_{F}$-lattice of rank $n$ and $c=(-1)^{n/2}dN$. Suppose that $R_{1}(N)=R_{2}(N)+2e_{\mathfrak{p}}=0$.
			
			\begin{enumerate}[itemindent=-0.5em,label=\rm (\roman*)]
				\item  If $n=2$, then $ N\cong \mathbf{H}_{F}$ or $\mathbf{A}_{F}$, according as $d(c)=\infty$ or $ 2e_{\mathfrak{p}}$; thus $FN\cong \mathbb{H}_{F}$ or $\mathbb{A}_{F}$, according as $d(c)=\infty$ or $ 2e_{\mathfrak{p}}$.

				\item If $n=4$ and $ R_{3}(N)=R_{4}(N)+2e_{\mathfrak{p}}=0$, then $ N\cong \mathbf{H}_{F}\perp \mathbf{H}_{F}$ or $\mathbf{H}_{F}\perp\mathbf{A}_{F}$, according as $ d(c)=\infty$ or $  2e_{\mathfrak{p}}$. 
			\end{enumerate}
		\end{prop}

			 \begin{lem}\label{lem:binary-quaternary-lifting-dyadic}
			Let $N$ be an $\mathcal{O}_{F}$-lattice of rank $n$. Suppose that $R_{1}(N)=R_{2}(N)+2e_{\mathfrak{p}}=0$.
			\begin{enumerate}[itemindent=-0.5em,label=\rm (\roman*)]
				\item  If $n=2$, then $ N\cong \mathbf{H}_{E}$ if and only if $ N\cong \mathbf{H}_{F}$, or $ N\cong \mathbf{A}_{F}$ and $f_{\pmid}$ is even; and $ FN\cong \mathbb{H}_{E}$ if and only if $ FN\cong \mathbb{H}_{F}$, or $ FN\cong \mathbb{A}_{F}$ and $f_{\pmid}$ is even.
				
				\item  If $n=4$ and $R_{3}(N)=R_{4}(N)+2e_{\mathfrak{p}}=0$, then $N\cong \mathbf{H}_{E}\perp \mathbf{H}_{E}$ if and only if $N\cong \mathbf{H}_{F}\perp \mathbf{H}_{F}$, or $N\cong \mathbf{H}_{F}\perp \mathbf{A}_{F}$ and $f_{\pmid}$ is even.
			\end{enumerate}
		\end{lem}
		\begin{proof}
			 (i) If $n=2$, since $R_{1}(N)=R_{2}(N)+2e_{\mathfrak{p}}=0$, \cite[Corollary 2.3(ii)]{HeHu2} implies that $d(c)=\infty$ or $2e_{\mathfrak{p}}$. By Lemma \ref{lem:defectP-p}(ii), $\widetilde{d}(c)=\infty$ if and only if $d(c)=\infty$, or $d(c)=2e_{\mathfrak{p}}$ and $f_{\pmid}$ is even. So the first assertion follows by Proposition \ref{prop:binary-quaternary-dyadic}(i). Similarly for the second assertion.
			 
			 (ii) If $n=4$, let $N\cong \prec b_{1},b_{2},b_{3},b_{4}\succ$. Then from \cite[Corollary 4.4(i)]{beli_integral_2003}, we have $N\cong \prec b_{1},b_{2}\succ \perp \prec b_{3},b_{4}\succ$. So the assertion follows by the first assertion of (i) and Proposition \ref{prop:binary-quaternary-dyadic}(ii).
		\end{proof}
 
  	We will prove the cases $n=1$, even $n\ge 2$, odd $n\ge 3$ in sequence. And we will use Proposition \ref{prop:R-extension} to treat the relation between $R$ and $\widetilde{R}$ repeatedly.
		
		\begin{proof}[Proof of Theorem \ref{thm:1-universaldyadic-overE}]
			\textbf{Necessity.} Suppose that $M$ is $1$-universal over $E$.

			By Theorem \ref{thm:1-universaldyadic}, we have $m\ge 2$, $R_{1}=\widetilde{R}_{1}=0$ and $\widetilde{\alpha}_{1}\in \{0,1\}$.
			
			Suppose $\widetilde{\alpha}_{1}=0$. By Proposition \ref{prop:alpha-extension}(iv), $\alpha_{1}=0$. If $m=2$ or $R_{3}>1$ or $R_{3}=1$ and $e_{\pmid}>1$, then $m=2$ or $\widetilde{R}_{3}>1$. By Theorem \ref{thm:1-universaldyadic}(i)(a), $[a_{1},a_{2}]\cong \mathbb{H}_{E}$. If $m\ge 3$, $R_{3}=e_{\pmid}=1$ and either $m=3$ or $R_{4}>2e_{\mathfrak{p}}+1$, then
			\begin{align*}
				m\ge 3,\quad   \widetilde{R}_{3}=1,\quad\text{and}\quad m=3\quad\text{or}\quad \widetilde{R}_{4}=R_{4}>2e_{\mathfrak{p}}+1=2e_{\mathfrak{P}}+1.
			\end{align*}
			Hence, by Theorem \ref{thm:1-universaldyadic}(i)(b), $[a_{1},a_{2}]\cong\mathbb{H}_{E}$. Therefore, condition (i) holds by Lemma \ref{lem:binary-quaternary-lifting-dyadic}(i).

			Now, suppose $m\ge 3$ and $\widetilde{\alpha}_{1}=1$. By Proposition \ref{prop:alpha-extension}(v), $\alpha_{1}=e_{\pmid}=1$. Hence the invariants $R_{i}$ and $\alpha_{i}$ are unchanged under field extensions. Thus condition (ii)(a) holds.

			Assume that  $f_{\pmid}$ is odd (and thus $[E:F]$ is odd),  $R_{2}\le 0$, $R_{3}\le 1$ and either $m=3$ or $R_{4}-R_{3}>2e_{\mathfrak{p}}$. Since $e_{\pmid}=1$, we have 
			\begin{align*}
				\widetilde{R}_{2}\le 0,\quad \widetilde{R}_{3}\le 1,\quad\text{and either}\quad m=3\quad\text{or}\quad \widetilde{R}_{4}-\widetilde{R}_{3}>2e_{\mathfrak{P}}.
			\end{align*}
			Hence, by Theorem \ref{thm:1-universaldyadic}(ii)(b), $[a_{1},a_{2},a_{3}]$ is isotropic over $E$, i.e.,  $(-a_{1}a_{2},-a_{1}a_{3})_{\mathfrak{P}}=1 $. So,  by \cite[Corollary]{bender_lifting_1973}, $(-a_{1}a_{2},-a_{1}a_{3})_{\mathfrak{p}}=1 $ and thus $[a_{1},a_{2},a_{3}]$ is isotropic over $F$. Therefore, condition (ii)(b) holds.

			\textbf{Sufficiency.} First, we have $m\ge 2$, $R_{1}=0$ and $\alpha_{1}\in \{0,1\}$. Then, by Propositions \ref{prop:R-extension}(ii) and \ref{prop:alpha-extension}(iv)(v),
			\begin{align*}
				\widetilde{R}_{1}=0\quad\text{and}\quad\widetilde{\alpha}_{1}\in \{0,1\}.
			\end{align*}

			Assume $\widetilde{\alpha}_{1}=0$. If $\widetilde{R}_{3}>1$, then $R_{3}>1$ or $R_{3}=1$ and $e_{\pmid}>1$; thus condition (i)(a) is satisfied.  If $\widetilde{R}_{3}=1$, then $R_{3}=e_{\pmid}=1$ and $R_{4}=\widetilde{R}_{4}>2e_{\mathfrak{P}}+1=2e_{\mathfrak{p}}+1$; thus condition (i)(b) is satisfied. Hence, in both cases, $[a_{1},a_{2}]\cong \mathbb{H}_{F}$, or $[a_{1},a_{2}]\cong \mathbb{A}_{F}$ and $f_{\pmid}$ is even. Thus, by Lemma \ref{lem:binary-quaternary-lifting-dyadic}(i), $[a_{1},a_{2}]\cong \mathbb{H}_{E}$. So, by Theorem \ref{thm:1-universaldyadic}, $M$ is 1-universal over $E$.

			Assume $\widetilde{\alpha}_{1}=1$. By Proposition \ref{prop:alpha-extension}(v), $\alpha_{1}=e_{\pmid}=1$.
			
			If $\widetilde{R}_{2}=1$ or $\widetilde{R}_{3}>1$, then $R_{2}=1$ or $R_{3}>1$. Hence $m\ge 4$ and $\alpha_{3}\le 2(e_{\mathfrak{p}}-\lfloor(R_{3}-R_{2})/2\rfloor)-1$. Since $e_{\pmid}=1$, the same conditions hold for the liftings of inequalities. Hence, by Theorem \ref{thm:1-universaldyadic}, $M$ is 1-universal over $E$.
			
			If $\widetilde{R}_{2}\le 0$, $\widetilde{R}_{3}\le 1$ and either $m=3$ or $\widetilde{R}_{4}-\widetilde{R}_{3}>2e_{\mathfrak{P}}$, then
			\begin{align*}
				R_{2}\le 0,\quad R_{3}\le 1,\quad\text{and}\quad m=3\quad\text{or}\quad R_{4}-R_{3}>2e_{\mathfrak{p}}.
			\end{align*}
			Recall that $e_{\pmid}=1$, so $f_{\pmid}$ and $[E:F]$ have the same parity. Hence if $f_{\pmid}$ is even, then, by \cite[Corollary]{bender_lifting_1973}, $(-a_{1}a_{2},-a_{1}a_{3})_{\mathfrak{P}}=1 $; if $f_{\pmid}$ is odd, then, by \cite[Corollary]{bender_lifting_1973} and condition (ii)(b), $(-a_{1}a_{2},-a_{1}a_{3})_{\mathfrak{P}}=(-a_{1}a_{2},-a_{1}a_{3})_{\mathfrak{p}}=1$. Hence in both cases, $[a_{1},a_{2},a_{3}]$ is isotropic over $E$. So, again by Theorem \ref{thm:1-universaldyadic}, $M$ is 1-universal over $E$.
		\end{proof}

		 \begin{proof}[Proof of Theorem \ref{thm:even-n-universaldyadic-overE}]
		 	\textbf{Necessity.} Suppose that $M$ is $n$-universal over $E$.
		 	
		 	 From Theorem \ref{thm:even-n-universaldyadic}(i), we have $\widetilde{R}_{i}=0$ for $i\in [1,n+1]^{O}$ and $\widetilde{R}_{i}=-2e_{\mathfrak{P}}$ for $i\in [1,n]^{E}$. Then 
		 	 \begin{align*}
		 	 	R_{i}=0\quad\text{for $i\in [1,n+1]^{O}$}\quad\text{and}\quad R_{i}=-2e_{\mathfrak{p}}\quad\text{for $i\in [1,n]^{E}$}.
		 	 \end{align*}
		 	 This shows condition (i).
		 	 
		 	  If $m=n+2=4$, then, by Theorem \ref{thm:even-n-universaldyadic}(ii), $M\cong \mathbf{H}_{E}\perp \mathbf{H}_{E}$, which, by Lemma \ref{lem:binary-quaternary-lifting-dyadic}(ii), is equivalent to that $ M\cong \mathbf{H}_{F}\perp \mathbf{H}_{F}$, or $M\cong \mathbf{H}_{F}\perp \mathbf{A}_{F}$ and $f_{\pmid}$ is even. This shows condition (ii).
		 	 
		 	 Suppose $m\ge n+3$. Then, from Theorem \ref{thm:even-n-universaldyadic}(iii)(a), $\widetilde{\alpha}_{n+1}\in \{0,1\}$. By Proposition \ref{prop:alpha-extension}(iv) and (v), we have
		 	 \begin{align}\label{alphan+1}
		 	 	 \alpha_{n+1}=0\quad\text{or}\quad\alpha_{n+1}=e_{\pmid}=1.
		 	 \end{align}

		 	 If $R_{n+3}-R_{n+2}>2e_{\mathfrak{p}}$, then  $\widetilde{R}_{n+3}-\widetilde{R}_{n+2}>2e_{\mathfrak{P}} $. So, by Theorem \ref{thm:even-n-universaldyadic}(iii)(b), $\widetilde{R}_{n+2}=-2e_{\mathfrak{P}}$. Hence $R_{n+2}=-2e_{\mathfrak{p}}$. Assume further that either $n\ge 4$, or $n=2$, $d(a_{1,4})=2e_{\mathfrak{p}}$ and $f_{\pmid}$ is odd. In the latter case, we have $\widetilde{d}(a_{1,4})=2e_{\mathfrak{P}} $. Hence, by Theorem \ref{thm:even-n-universaldyadic}(iii)(b), $\widetilde{R}_{n+3}=1$. So Proposition \ref{prop:R-extension}(iv) implies $ R_{n+3}=1 $ and $e_{\pmid}=1$. 
		 	 
		 	 Therefore, condition (iii)(a) and (b) hold.

		 If $R_{n+3}-R_{n+2}=2e_{\mathfrak{p}}$ and $R_{n+2}=2-2e_{\mathfrak{p}}$, then Proposition \ref{prop:alphaproperty}(ii) implies that $\alpha_{n+1}\not=0$. Hence, by \eqref{alphan+1}, $\alpha_{n+1}=e_{\pmid}=1$. It follows that $\widetilde{R}_{n+3}-\widetilde{R}_{n+2}=2e_{\mathfrak{P}}$ and $\widetilde{R}_{n+2}=2-2e_{\mathfrak{P}}$. Hence, by Theorem \ref{thm:even-n-universaldyadic}(iii)(c),  $\widetilde{d}(-a_{n+1}a_{n+2})=2e_{\mathfrak{P}}-1$. Since $e_{\pmid}=1$, Lemma \ref{lem:defectP-p}(ii) implies that
		 \begin{align*}
		 	 d(-a_{n+1}a_{n+2})=d(-a_{n+1}a_{n+2})e_{\pmid}\le \widetilde{d}(-a_{n+1}a_{n+2})=2e_{\mathfrak{P}}-1=2e_{\mathfrak{p}}-1.
		 \end{align*}
		 On the other hand, from \eqref{eq:alpha-defn} we have 
		  \begin{align*}
		  	d(-a_{n+1}a_{n+2})\ge R_{n+1}-R_{n+2}+\alpha_{n+1}=0-(2-2e_{\mathfrak{p}})+1=2e_{\mathfrak{p}}-1.
		  \end{align*}
		   So $d(-a_{n+1}a_{n+2})=2e_{\mathfrak{p}}-1$. Thus condition (iii)(c) also holds.
		 
		 \textbf{Sufficiency.}   Suppose that $M$ satisfies conditions (i)-(iii).
		 
		First, we have $R_{i}=0$ for $i\in [1,n+1]^{O}$ and $R_{i}=-2e_{\mathfrak{p}}$ for $i\in [1,n]^{E}$.
		 	 
		 If $m=n+2=4$, then, by Lemma \ref{lem:binary-quaternary-lifting-dyadic}(ii), $M\cong \mathbf{H}_{E}\perp \mathbf{H}_{E}$. Hence, by Theorem \ref{thm:even-n-universaldyadic}, $ M $ is $2$-universal over $E$.
		 	 
		 	 Suppose $m\ge n+3$. If $\alpha_{n+1}=0$, by Proposition \ref{prop:alpha-extension}(iv), $\widetilde{\alpha}_{n+1}=0$; if $\alpha_{n+1}=e_{\pmid}=1$, by Proposition \ref{prop:alpha-extension}(v), $\widetilde{\alpha}_{n+1}= 1$. If $\widetilde{R}_{n+3}-\widetilde{R}_{n+2}\le 2e_{\mathfrak{P}}$, then, by Theorem \ref{thm:even-n-universaldyadic}, $M$ is $n$-universal over $E$.
		 	 
		 	  Assume $\widetilde{R}_{n+3}-\widetilde{R}_{n+2}>2e_{\mathfrak{P}} $. Then $R_{n+3}-R_{n+2}>2e_{\mathfrak{p}}$. Hence, by condition (iii)(b),  $ R_{n+2}=-2e_{\mathfrak{p}}$, so $\widetilde{R}_{n+2}=-2e_{\mathfrak{P}}$. In particular, when $n=2$, $\widetilde{R}_{4}=-2e_{\mathfrak{P}}$, and \cite[Proposition 2.7(iii)]{HeHu2} implies that $d(a_{1,4})=2e_{\mathfrak{p}}$ or $\infty$.

		 	  If $n\ge 4$, or $n=2$, $d(a_{1,4})=2e_{\mathfrak{p}}$ and $f_{\pmid}$ is odd, then, by Lemma \ref{lem:defectP-p}(ii), $\widetilde{d}(a_{1,4})=2e_{\mathfrak{P}}$. Also, $R_{n+3}=e_{\pmid}=1$ from condition (iii)(b) and so $\widetilde{R}_{n+3}=1$. Hence, by Theorem \ref{thm:even-n-universaldyadic}, $M$ is $n$-universal over $E$.
		 	 
		 	 If $n=2$ and either $d(a_{1,4})=\infty $, or $d(a_{1,4})=2e_{\mathfrak{p}}$ and $f_{\pmid}$ is even, then, by Lemma \ref{lem:defectP-p}(ii), $\widetilde{d}(a_{1,4})=\infty$. Hence, by Lemma \ref{lem:binary-quaternary-lifting-dyadic}(ii), $\prec a_{1},a_{2},a_{3},a_{4}\succ\cong \mathbf{H}_{E}\perp \mathbf{H}_{E}$. Thus the $\mathcal{O}_{F}$-sublattice $\prec a_{1},a_{2},a_{3},a_{4}\succ $ of $M $ is $2$-universal over $E$, so is $M$.

		 	 Suppose that $\widetilde{R}_{n+3}-\widetilde{R}_{n+2}=2e_{\mathfrak{P}}$ and $\widetilde{R}_{n+2}=2-2e_{\mathfrak{P}}$. By Proposition \ref{prop:R-extension}(i),  $(R_{n+2}+2e_{\mathfrak{p}})e_{\pmid}=\widetilde{R}_{n+2}+2e_{\mathfrak{P}}=2$. So $e_{\pmid}$ divides $2$ and hence $e_{\pmid}\in \{1,2\}$.
		 	 
		 	 If $e_{\pmid}=1$, then $R_{n+2}=\widetilde{R}_{n+2}=2-2e_{\mathfrak{P}}=2-2e_{\mathfrak{p}}$; if $e_{\pmid}=2$, then $2R_{n+2}=\widetilde{R}_{n+2}=2-2e_{\mathfrak{P}}$ and thus $R_{n+2}=1-e_{\mathfrak{P}}$. Hence
		 	 \begin{align*}
		 	 	(e_{\pmid},R_{n+2})=(1,2-2e_{\mathfrak{p}})\quad\text{or}\quad (2,1-e_{\mathfrak{P}}).
		 	 \end{align*}
		 
		 	  Since $\alpha_{n+1}=1$, Proposition \ref{prop:alphaproperty}(iii) implies that $R_{n+2}=R_{n+2}-R_{n+1}\in [2-2e_{\mathfrak{p}},0]^{E}\cup \{1\}$. Assume that $(e_{\pmid},R_{n+2})=  (2,1-e_{\mathfrak{P}})$. Then  $R_{n+2}\not=1$, otherwise, $e_{\mathfrak{P}}=0$. Hence $R_{n+2}$ must be even. But $R_{n+2}=1-e_{\mathfrak{P}}\not\equiv e_{\mathfrak{P}}\pmod{2}$ and thus $e_{\mathfrak{P}}$ is odd, so is $e_{\pmid}$. This is a contradiction. Hence
		 	  \begin{align*}
		 	  	 e_{\pmid}=1\quad\text{and}\quad R_{n+2}=2-2e_{\mathfrak{p}}.
		 	  \end{align*}
		 	   From the hypothesis, we see that $R_{n+3}-R_{n+2}=\widetilde{R}_{n+3}-\widetilde{R}_{n+2}=2e_{\mathfrak{P}}=2e_{\mathfrak{p}}$. So, by Lemma \ref{lem:defectP-p}(ii) and condition (iii)(c), we conclude that
		 	 \begin{align*}
		 	 	\widetilde{d}(-a_{n+1}a_{n+2})=d(-a_{n+1}a_{n+2})=2e_{\mathfrak{p}}-1=2e_{\mathfrak{P}}-1.
		 	 \end{align*}
		 	  Hence, by Theorem \ref{thm:even-n-universaldyadic}, $M$ is $n$-universal over $E$.
		 \end{proof}

	\begin{proof}[Proof of Theorem \ref{thm:odd-n-universaldyadic-overE}]
		For necessity, suppose that $M$ is $n$-universal over $E$. By Theorem \ref{thm:odd-n-universaldyadic}(i), we have $ \widetilde{R}_{i}=0 $ for $ i\in [1,n]^{O} $, $ \widetilde{R}_{i}=-2e_{\mathfrak{P}} $ for $ i\in [1,n]^{E} $,  $ \widetilde{\alpha}_{n}\in \{0,1\} $ and $\widetilde{R}_{n+3}-\widetilde{R}_{n+2}\le 2e_{\mathfrak{P}}$. Then, by Propositions \ref{prop:R-extension} and \ref{prop:alpha-extension}(iv)(v), one has the following conditions:
		\begin{align*}
			&R_{i}=0 \quad\text{for $i\in [1,n]^{O}$},\quad 	R_{i}=-2e_{\mathfrak{p}} \quad\text{for $i\in [1,n]^{E}$},\quad R_{n+3}-R_{n+2}\le 2e_{\mathfrak{p}}\quad\text{and} \\
			&\alpha_{n}=0\quad\text{or}\quad  \alpha_{n}=e_{\pmid}=1.
		\end{align*}
		 Therefore, conditions (i) and (iii) follow.
		
		If $\alpha_{n}=0$, by Proposition \ref{prop:alpha-extension}(iv), $\widetilde{\alpha}_{n}=0$. So, by Theorem \ref{thm:odd-n-universaldyadic}(ii), $\widetilde{R}_{n+2}\in \{0,1\}$. Hence 
		\begin{align*}
			R_{n+2}=\widetilde{R}_{n+2}=0\quad\text{or}\quad R_{n+2}=e_{\pmid}=\widetilde{R}_{n+2}=1.
		\end{align*}
		This shows the first part of condition (ii).
		
		If $\alpha_{n}=e_{\pmid}=1$, then the quantities $e_{\mathfrak{p}},R_{n+1},R_{n+2},\alpha_{n},\alpha_{n+2}$ are unchanged under field extensions, which are equal to $e_{\mathfrak{P}},\widetilde{R}_{n+1},\widetilde{R}_{n+2}, \widetilde{\alpha}_{n},\widetilde{\alpha}_{n+2}$, respectively. This shows the second part of condition (ii). 
		
		A similar argument can be applied to the sufficiency by using Propositions \ref{prop:R-extension} and  \ref{prop:alpha-extension}(iv)(v).
 	\end{proof}
 
In the rest, we assume that $F$ is an algebraic number field and $E$ a finite extension of $F$. To show Theorem \ref{thm:globally-n-universal-extension-E}, we also need to know when $n$-universality satisfies the local-global principle over $F$. Although Definition \ref{defn:n-uni} differs from \cite[Definition 1.4(3)]{hhx_indefinite_2021} when $n\ge 2$, we have a similar result by adapting the reasoning in \cite[Theorem 1.1]{hhx_indefinite_2021}, because the $n$-universality is trivial at real primes from our definition and at complex primes from \cite[Theorem 2.3]{hhx_indefinite_2021}.

\begin{lem}\label{lem:local-global}
	Let $L$ be an indefinite $\mathcal{O}_{F}$-lattice. Suppose that  $n\ge 3$, or $n=2$ and the class number of $F$ is odd. Then $L$ is $n$-universal over $F$ if and only if for $\mathfrak{p}\in \Omega_{F}\backslash \infty_{F}$, $L_{\mathfrak{p}}$ is $n$-universal over $F_{\mathfrak{p}}$.
\end{lem}

 \begin{proof}[Proof of Theorem \ref{thm:globally-n-universal-extension-E}]
 	 Suppose that $L$ is $n$-universal over $E$. Then for each $\mathfrak{a}\in \Omega_{F}$, $L_{\mathfrak{a}}$ is $n$-universal over $E_{\mathfrak{A}}$ for all primes $\mathfrak{A}$ dividing $ \mathfrak{a}$. By Lemma \ref{lem:local-global}, it is sufficient to show that $L_{\mathfrak{p}}$ is $n$-universal over $F_{\mathfrak{p}}$ for each $\mathfrak{p}\in \Omega_{F}\backslash \infty_{F}$. 
 	 
 	 Let $\mathfrak{p}\in \Omega_{F}\backslash\infty_{F}$. Then there exists $\mathfrak{P}\in \Omega_{E}\backslash \infty_{E}$  dividing $\mathfrak{p}$. Since $L_{\mathfrak{p}}$ is $n$-universal over $E_{\mathfrak{P}}$, by Corollaries \ref{cor:non-dyadic-n-universal-overE-implies-overF} and \ref{cor:dyadic-n-universal-overE-implies-overF}, $L_{\mathfrak{p}}$ is also $n$-universal over $F_{\mathfrak{p}}$, as desired.
 \end{proof}


 \section*{Acknowledgments}
I would like to express my gratitude to Prof. Fei Xu for helpful comments and suggestions, and to Prof. Constantin-Nicolae Beli for detailed discussions, particularly for guidance on the formulas related to relative integral spinor norms in terms of BONGs. This work was supported by the National Natural Science Foundation of China (Project No. 12301013).

\end{document}